\documentclass{article}

\usepackage{amsmath,amssymb,amsthm,mathrsfs}
\usepackage[overload]{empheq}
\usepackage{graphicx}
\usepackage[colorlinks=true]{hyperref}
\usepackage{pdfsync}

\newtheorem{theorem}{Theorem}
\newtheorem{proposition}{Proposition}
\newtheorem{lemma}{Lemma}

\theoremstyle{definition}
\theoremstyle{definition}\newtheorem{definition}{Definition}
\theoremstyle{definition}\newtheorem{remark}{Remark}

\DeclareMathOperator*{\supess}{sup\,ess}
\DeclareMathOperator*{\argmax}{arg\,max}
\newcommand{\fonction}[5]{\begin{array}[t]{lrcl}#1 :&#2 &\longrightarrow &#3\\&#4& \longmapsto &#5 \end{array}}

\def\R{\mathbb{R}}
\def\N{\mathbb{N}}

\def\T{\mathbb{T}}
\def\TU{{\mathbb{T}_1}}

\def\bs{\backslash}
\def\di{\displaystyle}

\def\RS{\mathrm{RS}}
\def\RD{\mathrm{RD}}

\def\DD{\Delta}
\def\AC{\mathrm{AC}}
\def\CC{\mathrm{C}}

\def\UU{\mathcal{UQ}^{b}}

\def\UUEB{\mathcal{U\bar{Q}}^{b^*}}

\def\B{\overline{\mathrm{B}}}
\def\t{\tau}
\def\iP{\amalg}

\def\L{\mathrm{L}}
\def\LL{\mathscr{L}}
\def\E{\mathrm{E}}
\def\S{\mathrm{S}}

\def\V{\mathcal{V}}

\def\O{\mathcal{O}}
\def\P{\mathrm{P}}

\def\OSCP{{\bf (OSDCP)}^{\T}_{\TU}}
\def\OSCPDEUX{{\bf (OSDCP)}^{\T}_{\T_2}}

\def\OSCPI{{\bf (OSDCP)}^{\R^+}_{\N T}}

\def\OF{{\bf (O)}}

\title{Optimal sampled-data control, and generalizations on time scales}

\author{Lo\"ic Bourdin\footnote{Universit\'e de Limoges, Institut de recherche XLIM, D\'epartement de Math\'ematiques et d'Informatique. UMR CNRS 7252. Limoges, France (\texttt{loic.bourdin@unilim.fr}).}
\and
Emmanuel Tr\'elat\footnote{Sorbonne Universit\'es, UPMC Univ Paris 06, CNRS UMR 7598, Laboratoire Jacques-Louis Lions, Institut Universitaire de France, F-75005, Paris, France (\texttt{emmanuel.trelat@upmc.fr}).}
}

\date{}

\begin{document}

\maketitle

\begin{abstract}
In this paper, we derive a version of the Pontryagin maximum principle for general finite-dimensional nonlinear optimal sampled-data control problems. Our framework is actually much more general, and we treat optimal control problems for which the state variable evolves on a given time scale (arbitrary non-empty closed subset of $\R$), and the control variable evolves on a smaller time scale. Sampled-data systems are then a particular case. Our proof is based on the construction of appropriate needle-like variations and on the Ekeland variational principle.
\end{abstract}

\bigskip


\bigskip



\section{Introduction}

Optimal control theory is concerned with the analysis of controlled dynamical systems, where one aims at steering such a system from a given configuration to some desired target by minimizing some criterion. The Pontryagin maximum principle (in short, PMP), established at the end of the 50's for general finite-dimensional nonlinear continuous-time dynamics (see \cite{pont}, and see \cite{gamk} for the history of this discovery), is certainly the milestone of the classical optimal control theory. It provides a first-order necessary condition for optimality, by asserting that any optimal trajectory must be the projection of an extremal. The PMP then reduces the search of optimal trajectories to a boundary value problem posed on extremals. Optimal control theory, and in particular the PMP, has an immense field of applications in various domains, and it is not our aim here to list them.

\medskip

We speak of a {\em purely continuous-time optimal control problem}, when both the state $q$ and the control $u$ evolve continuously in time, and the control system under consideration has the form
$$
\dot q(t) = f(t,q(t),u(t)), \qquad \text{for a.e. } t \in \R^+,
$$
where $q(t) \in \R^n$ and $u(t) \in \Omega \subset \R^m$. Such models assume that the control is permanent, that is, the value of $u(t)$ can be chosen at each time $t \in \R^+$. We refer the reader to textbooks on continuous optimal control theory such as \cite{agrach,Bon-Chy03,trel2,bres,brys,BulloLewis,hest,Jurdjevic,lee,pont,Schattler,seth,trel} for many examples of theoretical or practical applications.

\medskip

We speak of a {\em purely discrete-time optimal control problem}, when both the state $q$ and the control $u$ evolve in a discrete way in time, and the control system under consideration has the form
$$
q_{k+1}-q_k = f(k,q_k,u_k), \qquad k \in \N,
$$
where $q_k \in \R^n$ and $u_k \in \Omega \subset \R^m$. As in the continuous case, such models assume that the control is permanent, that is, the value of $u_k$ can be chosen at each time $k \in \N$. A version of the PMP for such discrete-time control systems has been established in \cite{Halkin,holt2,holt} under appropriate convexity assumptions. The considerable development of the discrete-time control theory was in particular motivated by the need of considering digital systems or discrete approximations in numerical simulations of differential control systems (see the textbooks \cite{bolt,cano,mord,seth}).
It can be noted that some early works devoted to the discrete-time PMP (like \cite{fan}) are mathematically incorrect. Some counterexamples were provided in \cite{bolt} (see also \cite{mord}), showing that, as is now well known, the exact analogous of the continuous-time PMP does not hold at the discrete level. More precisely, the maximization condition of the continuous-time PMP cannot be expected to hold in general in the discrete-time case. Nevertheless, a weaker condition can be derived, in terms of nonpositive gradient condition (see \cite[Theorem 42.1]{bolt}).

\medskip

We speak of an {\em optimal sampled-data control problem}, when the state $q$ evolves continuously in time, whereas the control $u$ evolves in a discrete way in time. This hybrid situation is often considered in practice for problems in which the evolution of the state is very quick (and thus can be considered continuous) with respect to that of the control. We often speak, in that case, of {\em digital control}. This refers to a situation where, due for instance to hardware limitations or to technical difficulties, the value $u(t)$ of the control can be chosen only at times $t=kT$, where $T>0$ is fixed and $k \in \N$. This means that, once the value $u(kT)$ is fixed, $u(t)$ remains constant over the time interval $[kT,(k+1)T)$. Hence the trajectory $q$ evolves according to
$$
\dot q(t) = f(t,q(t),u(kT)), \qquad \text{for a.e. } t \in [kT,(k+1)T),\quad k\in\N.
$$
In other words, this {\em sample-and-hold} procedure consists of ``freezing" the value of $u$ at each \textit{controlling time} $t=kT$ on the corresponding \textit{sampling time interval} $[kT,(k+1)T)$, where $T$ is called the \textit{sampling period}. In this situation, the control of the system is clearly nonpermanent.

\medskip

To the best of our knowledge, the classical optimal control theory does not treat general nonlinear optimal sampled-data control problems, but concerns either purely continuous-time, or purely discrete-time optimal control problems.
It is one of our objectives to derive, in this paper, a PMP which can be applied to general nonlinear optimal sampled-data control problems.

Actually, we will be able to establish a PMP in the much more general framework of {\em time scales}, which unifies and extends continuous-time and discrete-time issues.
But, before coming to that point, we feel that it is of interest to enunciate a PMP in the particular case of sampled-data systems and where the set $\Omega$ of pointwise constraints on the controls is convex.

\paragraph{PMP for optimal sampled-data control problems and $\Omega$ convex.}
Let $n$, $m$ and $j$ be nonzero integers. Let $T>0$ be an arbitrary sampling period. In what follows, for any real number $t$, we denote by $E(t)$ the integer part of $t$, defined as the unique integer such that $E(t)\leq t< E(t)+1$. Note that $k = E(t/T)$ whenever $kT\leq t<(k+1)T$.
We consider the general nonlinear optimal sampled-data control problem
\begin{equation*}
\OSCPI \qquad \left\{\begin{split}
& \min \; \int_0^{t_f} f^0 ( \t, q(\t), u( k(\t) T ) ) \, d\t ,  \quad \text{with } k(\t) = E(\t /T), \\
& \dot q(t) = f ( t,q(t), u( k(t) T )  ) ,  \quad \text{with } k(t) = E(t /T), \\
& u(kT) \in \Omega ,  \\
& g(q(0),q(t_f)) \in \S .
\end{split}\right.
\end{equation*}
Here, $f: \R\times \R^{n} \times \R^{m}  \rightarrow \R^{n}$ and $f^0: \R\times \R^{n} \times \R^{m}  \rightarrow \R$ are continuous, and of class $\CC^1$ in $(q,u)$,
$g: \R^n \times \R^n \rightarrow \R^j$ is of class $\CC^1$,
and $\Omega$ (resp., $\S$) is a non-empty closed convex subset of $\R^m$ (resp., of $\R^j$).
The final time $t_f \geq 0$ can be fixed or not.

Note that, under appropriate (usual) compactness and convexity assumptions, the optimal control problem $\OSCPI$ has at least one solution (see Theorem~\ref{propexistence} in Section~\ref{subsectionOSCP}).

Recall that $g$ is said to be submersive at a point $(q_1,q_2) \in \R^n \times \R^n$ if the differential of $g$ at this point is surjective. We define the Hamiltonian $H:\R\times\R^n \times \R^n \times \R \times \R^m\rightarrow \R$, as usual, by
$H(t,q,p,p^0,u) = \langle p, f(t,q,u) \rangle_{\R^n} + p^0 f^0(t,q,u)$.

\begin{theorem}[Pontryagin maximum principle for $\OSCPI$]\label{thmmainintro}
If a trajectory $q^*$, defined on $[0,t^*_f]$ and associated with a sampled-data control $u^*$, is an optimal solution of $\OSCPI$, then there exists a nontrivial couple $(p,p^0)$, where $p : [0,t^*_f] \rightarrow \R^n$ is an absolutely continuous mapping (called adjoint vector) and $p^0 \leq 0$, such that the following conditions hold:
\begin{itemize}
\item Extremal equations:
\begin{equation*}
\dot q^*(t) = \frac{\partial H}{\partial p} (t,q^*(t),p(t),p^0,u^*(k(t)T)), \qquad
\dot p(t) = - \frac{\partial H}{\partial q} (t,q^*(t),p(t),p^0,u^*(k(t)T)) ,
\end{equation*}
for almost every $t\in[0,t^*_f)$, with $k(t)=E(t/T)$.
\item Maximization condition:\\
For every controlling time $kT \in [0,t^*_f)$ such that $(k+1)T \leq  t^*_f$, we have
\begin{equation}\label{secondconditionintro}
\left\langle \int_{kT}^{(k+1)T} \frac{\partial H}{\partial u} (\t,q^*(\t),p(\t),p^0,u^*(kT)) \; d\t \; , \; y-u^*(kT) \right\rangle_{\R^m}  \leq 0,
\end{equation}
for every $y \in \Omega$. In the case where $kT \in [0,t^*_f)$ with $(k+1)T >  t^*_f$, the above maximization condition is still valid provided $(k+1)T$ is replaced with $t^*_f$.
\item Transversality conditions on the adjoint vector:\\
If $g$ is submersive at $(q^*(0),q^*(t^*_f))$, then the nontrivial couple $(p,p^0)$ can be selected to satisfy
\begin{equation*}
p(0) = - \left( \frac{\partial g}{\partial q_1} (q^*(0),q^*(t^*_f)) \right)^\top  \psi,\qquad
p(t^*_f) =  \left( \frac{\partial g}{\partial q_2} (q^*(0),q^*(t^*_f)) \right)^\top  \psi,
\end{equation*}
where $-\psi$ belongs to the orthogonal of $\S$ at the point $g (q^*(0),q^*(t^*_f)) \in \S$.
\item Transversality condition on the final time:\\
If the final time is left free in the optimal control problem $\OSCPI$, if $t^*_f>0$ and if $f$ and $f^0$ are of class $\CC^1$ with respect to $t$ in a neighborhood of $t^*_f$, then the nontrivial couple $(p,p^0)$ can be moreover selected to satisfy
\begin{equation*}
H(t^*_f, q^*(t^*_f), p(t^*_f),p^0,u^*(k^*_f T) ) = 0,
\end{equation*}
where $k^*_f=E(t^*_f/T)$ whenever $t^*_f \notin \N T$, and $k^*_f=E(t^*_f/T)-1$ whenever $t^*_f \in \N T$.
\end{itemize}
\end{theorem}

Note that the only difference with the usual statement of the PMP for purely continuous-time optimal control problems is in the maximization condition. Here, for sampled-data control systems, the usual pointwise maximization condition of the Hamiltonian is replaced with the inequality~\eqref{secondconditionintro}.
This is not a surprise, because already in the purely discrete case, as mentioned earlier, the pointwise maximization condition fails to be true in general, and must be replaced with a weaker condition.

The condition~\eqref{secondconditionintro}, which is satisfied for every $y \in \Omega$, gives a necessary condition allowing to compute $u^*(kT)$ in general, and this, for all controlling times $kT \in [0,t^*_f)$. We will provide in Section~\ref{sectionexemple} a simple optimal consumption problem with sampled-data control, and show how these computations can be done in a simple way.

Note that the optimal sampled-data control problem $\OSCPI$ can of course be seen as a finite-dimensional optimization problem where the unknowns are $u^*(kT)$, with $k \in \N$ such that $kT \in [0,t^*_f)$. The same remark holds, by the way, for purely discrete-time optimal control problems. One could then apply classical Lagrange multiplier (or KKT) rules to such optimization problems with constraints (numerically, this leads to direct methods). The Pontryagin maximum principle is a far-reaching version of the Lagrange multiplier rule, yielding more precise information and reducing the initial optimal control problem to a shooting problem (see, \textit{e.g.}, \cite{Trelat_JOTA2012} for such a discussion).

\paragraph{Extension to the time scale framework.}
In this paper, we actually establish a version of Theorem~\ref{thmmainintro} in a much more general framework, allowing for example to study sampled-data control systems where the control can be permanent on a first time interval, then sampled on a finite set, then permanent again, etc. 
More precisely, Theorem~\ref{thmmainintro} can be extended to a general framework in which the set of controlling times is not $\N T$ but some arbitrary non-empty closed subset of $\R$ (\textit{i.e.}, a \textit{time scale}), and also the state may evolve on another time scale.
We will state a PMP for such general systems in Section~\ref{secPMP} (see Theorem~\ref{thmmain}). Since such systems, where we have two time scales (one for the state and one for the control), can be viewed as a generalization of sampled-data control systems, we will refer to them as {\em sampled-data control systems on time scales}.

Let us first recall and motivate the notion of time scale. The time scale theory was introduced in \cite{hilg} in order to unify discrete and continuous analysis. By definition, a time scale $\T$ is an arbitrary non-empty closed subset of $\R$, and a dynamical system is said to be posed on the time scale $\T$ whenever the time variable evolves along this set $\T$. The time scale theory aims at closing the gap between continuous and discrete cases, and allows one to treat general processes involving both continuous-time and discrete-time variables. The purely continuous-time case corresponds to $\T=\R^+$ and the purely discrete-time case corresponds to $\T=\N$. But a time scale can be much more general (see, \textit{e.g.}, \cite{gama,may} for a study of a seasonally breeding population whose generations do not overlap, and see \cite{ati} for applications to economics), and can even be a Cantor set.
Many notions of standard calculus have been extended to the time scale framework, and we refer the reader to \cite{agar2,agar3,bohn,bohn3} for details on that theory.

The theory of the calculus of variations on time scales, initiated in \cite{bohn2}, has been well studied in the existing literature (see, \textit{e.g.}, \cite{torr8,bohn4,bour12,torr9,hils,hils3}). In \cite{hils1,hils2}, the authors establish a weak version of the PMP (with a nonpositive gradient condition) for control systems defined on general time scales. In \cite{bour11}, we derived a strong version of the PMP, in a very general time scale setting, encompassing both the purely continuous-time PMP (with a maximization condition) and the purely discrete-time PMP (with a nonpositive gradient condition).

All these works are concerned with control systems defined on general time scales with permanent control. The main objective of the present paper is to handle control systems defined on general time scales with nonpermanent control, that we refer to as {\em sampled-data control systems on time scales}, and for which we assume that the state and the control are allowed to evolve on different time scales (the time scale of the control being a subset of the time scale of the state). This framework is the natural extension of the classical sampled-data setting, and allows to treat simultaneously many sampling-data control situations.

Our main result is a {\em PMP for general finite-dimensional nonlinear optimal sampled-data control problems on time scales}. Note that our result will be derived without any convexity assumption on the set $\Omega$ of pointwise constraints on the controls. Our proof is based on the construction of appropriate needle-like variations and on the Ekeland variational principle. In the case of a permanent control, our statement encompasses the time scale version of the PMP obtained in \cite{bour11}, and \textit{a fortiori} it also encompasses the classical continuous and discrete versions of the PMP.

\paragraph{Organization of the paper.}
In Section~\ref{section1}, after having recalled several basic facts in time scale calculus, we define a general nonlinear optimal sampled-data control problem defined on time scales, and we state a Pontryagin maximum principle (Theorem~\ref{thmmain}) for such problems.
Section~\ref{sectionapplication} is devoted to some applications of Theorem~\ref{thmmain} and further comments. Section~\ref{annexe} is devoted to the proof of Theorem~\ref{thmmain}.

\section{Main result}\label{section1}
Let $\T$ be a time scale, that is, an arbitrary non-empty closed subset of $\R$. Without loss of generality, we assume that $\T$ is bounded below, denoting by $a = \min \T $, and unbounded above.\footnote{In this paper we only work on a bounded subinterval of type $[a,b] \cap \T$ with $a$, $b \in \T$. It is not restrictive to assume that $a = \min \T $ and that $\T$ is unbounded above. On the other hand, these two assumptions widely simplify the notations introduced in Section~\ref{secprelimtimescale} (otherwise we would have, in all further statements, to distinguish between points of $\T \bs \{ \max\T \}$ and $\max\T$).}
Throughout the paper, $\T$ will be the time scale on which the state of the control system evolves.

We start the section by recalling some useful notations and basic results of time scale calculus, in particular the notion of Lebesgue $\DD$-measure and of absolutely continuous function within the time scale setting. The reader already acquainted with time scale calculus may jump directly to Section~\ref{subsectionOSCP}.

\subsection{Preliminaries on time scale calculus}\label{secprelimtimescale}
The forward jump operator $\sigma : \T \rightarrow \T$ is defined by $\sigma (t) = \inf \{ s \in \T\ \vert\ s > t \}$ for every $t \in \T$. A point $t \in \T$ is said to be \textit{right-scattered} whenever $\sigma (t) > t$. A point $t \in \T$ is said to be \textit{right-dense} whenever $\sigma (t) = t$. We denote by $\RS$ the set of all right-scattered points of $\T$, and by $\RD$ the set of all right-dense points of $\T$.
Note that $\RS$ is at most countable (see \cite[Lemma 3.1]{caba}) and that $\RD$ is the complement of $\RS$ in $\T$. The graininess function $\mu : \T \rightarrow \R^+$ is defined by $\mu(t) = \sigma (t) -t$ for every $t \in \T$.

For every subset $A$ of $\R$, we denote by $A_\T = A \cap \T$. An interval of $\T$ is defined by $I_\T$ where $I$ is an interval of $\R$. For every $b \in \T \backslash \{ a \}$ and every $s \in [a,b)_\T \cap \RD$, we set
\begin{equation}\label{defVbs}
\V^{s,b} = \{ \beta \geq 0 \mid s+\beta \in [s,b]_\T \}.
\end{equation}
Note that $0$ is not isolated in $\V^{s,b}$.

\paragraph{$\DD$-differentiability.}
Let $n \in \N^*$. The notations $\Vert \cdot \Vert_{\R^n}$ and $\langle \cdot, \cdot \rangle_{\R^n}$ respectively stand for the usual Euclidean norm and scalar product of $\R^n$. A function $q : \T \rightarrow \R^n$ is said to be $\DD$-differentiable at $t \in \T$ if the limit
$$q^\DD (t) = \lim\limits_{\substack{s \to t \\ s \in \T}} \frac{q^\sigma (t) -q(s)}{\sigma (t) -s} $$
exists in $\R^n$, where $q^\sigma = q \circ \sigma $. Recall that, if $t \in \RD$, then $q$ is $\DD$-differentiable at $t$ if and only if the limit of $\frac{q(t)-q(s)}{t-s}$ as $s \to t$, $s \in \T$, exists; in that case it is equal to $q^\DD (t)$.
If $t \in \RS$ and if $q$ is continuous at $t$, then $q$ is $\DD$-differentiable at $t$, and $q^\DD (t) = \frac{q^\sigma(t) - q(t)}{\mu (t)}$ (see \cite{bohn}).

If $q$, $q': \T \rightarrow \R^n$ are both $\DD$-differentiable at $t \in \T$, then the scalar product $\langle q  , q' \rangle_{\R^n}$ is $\DD$-differentiable at $t$ and
\begin{equation}\label{eqleibniz}
\langle q, q' \rangle_{\R^n}^\DD (t) = \langle q^\DD (t), q'^\sigma(t) \rangle_{\R^n}  + \langle q (t), q'^\DD(t) \rangle_{\R^n} = \langle q^\DD (t), q'(t) \rangle_{\R^n}  + \langle q^\sigma (t), q'^\DD(t) \rangle_{\R^n}.
\end{equation}
These equalities are usually called Leibniz formulas (see \cite[Theorem 1.20]{bohn}).

\paragraph{Lebesgue $\DD$-measure and Lebesgue $\DD$-integrability.}
Let $\mu_\DD$ be the Lebesgue $\DD$-measure on $\T$ defined in terms of Carath\'eodory extension in \cite[Chapter 5]{bohn3}. We also refer the reader to \cite{agar,aulb,caba,guse} for more details on the $\mu_\DD$-measure theory. For all $(c,d)\in\T^2$ such that $c \leq d $, one has $\mu_\DD ([c,d)_\T) = d-c$. Recall that $A \subset \T$ is a $\mu_\DD$-measurable set of $\T$ if and only if $A$ is an usual $\mu_L$-measurable set of $\R$, where $\mu_L$ denotes the usual Lebesgue measure (see \cite[Proposition 3.1]{caba}). Moreover, if $A \subset \T $, then
$$\mu_\DD ( A ) = \mu_L (A) +  \sum_{r \in A \cap \RS} \mu (r).$$
Let $A \subset \T$. A property is said to hold $\DD$-almost everywhere (in short, $\DD$-a.e.) on $A$ if it holds for every $t \in A \backslash A'$, where $A' \subset A$ is some $\mu_\DD$-measurable subset of $\T$ satisfying $\mu_\DD (A') = 0$. In particular, since $\mu_\DD (\{ r \}) = \mu (r) > 0$ for every $r \in \RS$, we conclude that, if a property holds $\DD$-a.e. on $A$, then it holds for every $r \in A \cap \RS$.

Similarly, if $A \subset \T$ is such that $\mu_\Delta (A) = 0$, then $A \subset \RD$.

Let $n \in \N^*$ and let $A \subset \T$ be a $\mu_\DD$-measurable subset of $\T$. Consider a function $q$ defined $\DD$-a.e.\ on $A$ with values in $\R^n$. Let $\tilde{A}=A \cup (r,\sigma(r))_{r \in A \cap \RS}$, and let $\tilde{q}$ be the extension of $q$ defined $\mu_L$-a.e.\ on $\tilde{A}$ by $\tilde{q} (t)=q(t)$ whenever $t \in A$, and by $\tilde{q}(t)=q(r)$ whenever $t \in (r,\sigma(r))$, for every $r \in A \cap \RS$.
Recall that $q$ is $\mu_\DD$-measurable on $A$ if and only if $\tilde{q}$ is $\mu_L$-measurable on $\tilde{A}$ (see \cite[Proposition 4.1]{caba}).

Let $n \in \N^*$ and let $A \subset \T$ be a $\mu_\DD$-measurable subset of $\T$. The functional space $\L^\infty_\T (A,\R^n)$ is the set of all functions $q$ defined $\DD$-a.e.\ on $A$, with values in $\R^n$, that are $\mu_\DD$-measurable on $A$ and bounded $\DD$-almost everywhere. Endowed with the norm $\Vert q \Vert_{\L^\infty_\T (A,\R^n)} = \supess_{\tau \in A} \Vert q(\tau) \Vert_{\R^n}$, it is a Banach space (see \cite[Theorem 2.5]{agar}). The functional space $\L^1_\T (A,\R^n)$ is the set of all functions $q$ defined $\DD$-a.e.\ on $A$, with values in $\R^n$, that are $\mu_\DD$-measurable on $A$ and such that
$\int_{A} \Vert q(\tau) \Vert_{\R^n} \, \DD \tau < + \infty$.
Endowed with the norm $\Vert q \Vert_{\L^1_\T (A,\R^n)} = \int_{A} \Vert q(\tau) \Vert_{\R^n} \, \DD \tau$, it is a Banach space (see \cite[Theorem 2.5]{agar}).
We recall here that if $q \in \L^1_\T (A,\R^n)$ then
$$ \int_{A} q(\tau) \,\DD \tau = \int_{\tilde{A}} \tilde{q}(\tau) \,d\tau =  \int_{A} q(\tau) \,d\tau +  \sum_{r \in A \cap \RS} \mu (r) q(r), $$
see \cite[Theorems 5.1 and 5.2]{caba}. Note that if $A$ is bounded then $\L^\infty_\T (A,\R^n) \subset \L^1_\T (A,\R^n)$.

\paragraph{Absolutely continuous functions.}
Let $n \in \N^*$ and let $(c,d) \in \T^2$ such that $c < d$. Let $\CC ([c,d]_\T,\R^n)$ denote the space of continuous functions defined on $[c,d]_\T$ with values in $\R^n$. Endowed with its usual uniform norm $\Vert \cdot \Vert_\infty$, it is a Banach space. Let $\AC([c,d]_\T,\R^n)$ denote the subspace of absolutely continuous functions.

Let $t_0 \in [c,d]_\T$ and $q:[c,d]_\T \rightarrow \R^n$.
It is easily derived from \cite[Theorem 4.1]{caba2} that $q \in \AC([c,d]_\T,\R^n)$ if and only if $q$ is $\DD$-differentiable $\DD$-a.e.\ on $[c,d)_\T$ and satisfies $q^\DD \in \L^1_\T ([c,d)_\T,\R^n)$, and for every $t \in [c,d]_\T$ one has
$ q(t) = q(t_0) + \int_{[t_0,t)_\T} q^\DD (\tau) \, \DD \tau$
whenever $t \geq t_0$, and
$ q(t) = q(t_0) - \int_{[t,t_0)_\T} q^\DD (\tau) \, \DD \tau $
whenever $ t \leq t_0$.

Assume that $q \in \L^1_\T ([c,d)_\T,\R^n)$, and let $Q$ be the function defined on $[c,d]_\T$ by $Q(t) = \int_{[t_0,t)_\T} q (\tau) \,\DD \tau$ whenever $ t \geq t_0$, and by $ Q(t) = - \int_{[t,t_0)_\T} q (\tau) \,\DD \tau $ whenever $t \leq t_0$.
Then $Q \in \AC([c,d]_\T)$ and $Q^\DD = q$ $\DD$-a.e.\ on $[c,d)_\T$.

Note that, if $q \in \AC([c,d]_\T, \R^n )$ is such that $q^\DD = 0$ $\DD$-a.e.\ on $[c,d)_\T$, then $q$ is constant on $[c,d]_\T$, and that, if $q$, $q' \in \AC([c,d]_\T, \R^n )$, then $\langle q,q' \rangle_{\R^n} \in \AC ([c,d]_\T, \R )$ and the Leibniz formula~\eqref{eqleibniz} is available $\DD$-a.e.\ on $[c,d)_\T$.

For every $q \in \L^1_\T ([c,d)_\T,\R^n)$, we denote by $\LL_{[c,d)_\T} (q)$ the set of points $t \in [c,d)_\T$ that are $\DD$-Lebesgue points of $q$. It holds $\mu_\DD (\LL_{[c,d)_\T}(q)) = \mu_\DD ([c,d)_\T) = d-c$, and
\begin{equation*}
\lim\limits_{\substack{\beta \to 0 \\ \beta \in \V^{s,d}}} \frac{1}{\beta}  \int_{[s,s+\beta)_\T} q(\tau) \, \DD \tau = q(s),
\end{equation*}
for every $s \in \LL_{[c,d)_\T}(q) \cap \RD$, where $\V^{s,d}$ is defined by~\eqref{defVbs}.

\subsection{Optimal sampled-data control problems on time scales}\label{subsectionOSCP}
Let $\TU$ be another time scale. Throughout the paper, $\TU$ will be the time scale on which the control evolves. We assume that $\TU \subset \T$.\footnote{Indeed, it is not natural to consider controlling times $t \in \TU$ at which the dynamics does not evolve, that is, at which $t \notin \T$. The value of the control at such times $t \in \TU \bs \T$ would not influence the dynamics, or, maybe, only on $[t^* ,+\infty[_\T$ where $t^*= \inf \{ s \in \T \,| \,s \geq t \}$. In this last case, note that $t^* \in \T$ and we can replace $\TU$ by $(\TU \cup \{ t^* \} ) \bs \{ t \}$ without loss of generality.}

Similarly to $\T$, we assume that $\min \TU = a$ and that $\TU$ is unbounded above. As in the previous paragraph, we introduce the notations $\sigma_1$, $\RS_1$, $\RD_1$, $\V^{s,b}_1$, $\DD_1$, etc., associated with the time scale $\TU$. Since $\TU \subset \T$, note that $\RS \subset \RS_1$ and $\RD_1 \subset \RD$. We define the map
$$ \fonction{\Phi}{\T}{\TU}{t}{\Phi(t) = \sup \,\{ s \in \TU \,| \,s \leq t \} .} $$
For every $t \in \TU$, we have $\Phi(t)=t$. For every $t \in \T \bs \TU$, we have $\Phi(t) \in \RS_1$ and $\Phi(t) < t < \sigma_1(\Phi(t))$. Note that, if $t \in \T$ is such that $\Phi(t) \in \RD_1$, then $t \in \TU$.

In what follows, given a function $u:\TU\rightarrow\R$, we denote by $u^\Phi$ the composition $u \circ \Phi : \T \rightarrow \R$. Of course, when dealing with functions having multiple components, this composition is applied to each component. Let us mention, at this step, that if $u \in \L^\infty_{\TU} (\TU,\R)$ then $u^\Phi \in \L^\infty_{\T} (\T,\R)$ (see Proposition~\ref{propnumber1} and more properties in Section~\ref{annexe1uetuphi}).

\medskip

Let $n$, $m$ and $j$ be nonzero integers. We consider the general nonlinear optimal sampled-data control problem on time scales
\begin{align}[left=\OSCP\qquad\empheqlbrace]
& \min  \int_{[a,b)_\T} f^0 (\t,q (\t), u^\Phi(\t) ) \, \DD \t , \notag  \\
& q^\DD(t) = f (t,q(t),u^\Phi(t)),  \label{DD-CS}  \\
& u\in \L^\infty_{\TU} (\TU,\Omega) , \notag \\
& g(q(a),q(b)) \in \S .  \notag
\end{align}
Here, the trajectory of the system is $q : \T \rightarrow \R^n$, the mappings $f:\T\times\R^{n} \times \R^{m} \rightarrow \R^{n}$ and $f^0: \T\times \R^{n} \times \R^{m}  \rightarrow \R$ are continuous, of class $\CC^1$ in $(q,u)$, the mapping $g: \R^n \times \R^n \rightarrow \R^j$ is of class $\CC^1$, $\Omega$ is a non-empty closed subset of $\R^m$, and $\S$ is a non-empty closed convex subset of $\R^j$.
The final time $b \in \T$ can be fixed or not.

\begin{remark}
We recall that, given $u \in \L^\infty_{\TU} (\TU,\R^m)$, we say that $q$ is a solution of~\eqref{DD-CS} on $I_\T$ if:
\begin{enumerate}
\item $I_\T$ is an interval of $\T$ satisfying $a \in I_\T$ and $I_\T \backslash \{ a \} \neq \emptyset$;
\item For every $c \in I_\T \backslash \{ a \}$, $q \in \AC ([a,c]_\T,\R^n)$ and~\eqref{DD-CS} holds for $\DD$-a.e. $t \in [a,c)_\T$.
\end{enumerate}
Existence and uniqueness of solutions (Cauchy-Lipschitz theorem on time scales) have been established in \cite{bour10}, and useful results are recalled in Section~\ref{recallCL}.
\end{remark}

\begin{remark}
The time scale $\TU$ stands for the set of controlling times of the control system~\eqref{DD-CS}. If $\T = \T_1$, then the control is permanent. The case $\T = \T_1 = \R^+$ corresponds to the classical continuous case, whereas $\T = \T_1 = \N$ coincides with the classical discrete case. If $\T_1 \varsubsetneq \T$, the control is nonpermanent and sampled. In that case, the sampling times are given by $t \in \RS_1$ such that $\sigma(t) < \sigma_1 (t)$ and the corresponding sampling time intervals are given by $[t,\sigma_1(t))_\T$. The classical optimal sampled-data control problem $\OSCPI$ investigated in Theorem~\ref{thmmainintro} corresponds to $\T=\R^+$ and $\TU=\N T$, with $T >0$.
\end{remark}

\begin{remark}\label{remarkinclusiondeuxT1}
Let us consider two optimal control problems $\OSCP$ and $\OSCPDEUX$, posed on the same general time scale $\T$ for the trajectories, but with two different sets of controlling times $\TU$ and $\T_2$, and let us assume that $\T_2 \subset \TU$. We denote by $\Phi_1$ and $\Phi_2$ the corresponding mappings from $\T$ to $\TU$ and $\T_2$ respectively. If $u_1 \in \L^\infty_{\TU} (\TU,\Omega)$ is an optimal control for $\OSCP$ and if there exists $u_2 \in \L^\infty_{\T_2} (\T_2,\Omega)$ such that $u_2^{\Phi_2} (t) = u_1^{\Phi_1} (t)$ for $\DD$-a.e. $t \in \T$, it is clear that $u_2$ is an optimal control for $\OSCPDEUX$. We refer to Section~\ref{sectionexemple} for examples.
\end{remark}

\begin{remark}\label{remarkparam}
The framework of $\OSCP$ encompasses optimal parameter problems. Indeed, let us consider the parametrized dynamical system
\begin{equation}\label{eqparam}
q^\DD(t) = f (t,q(t),\lambda), \quad \DD\text{-a.e. } t \in \T,
\end{equation}
with $\lambda \in \Omega$. Then, considering $\T_1 = \{ a \} \cup [b,+\infty[_{\T}$, \eqref{DD-CS} coincides with \eqref{eqparam} where $u(a)$ plays the role of $\lambda$. In this situation, Theorem~\ref{thmmain} (stated in Section~\ref{secPMP}) provides necessary conditions for optimal parameters $\lambda$. We refer to Section~\ref{sectionexemple} for examples.
\end{remark}

\begin{remark}\label{remmultiscale}
A possible extension is to study dynamical systems on time scales with several sampled-data controls but with different sets of controlling times:
$$  q^\DD(t) = f ( t ,q(t), u_1^{\Phi_1}(t),u_2^{\Phi_2} (t) ), \quad \DD\text{-a.e. } t \in \T, $$
where $\T_1$ and $\T_2$ are general time scales contained in $\T$, and $\Phi_1$ and $\Phi_2$ are the corresponding mappings from $\T$ to $\TU$ and $\T_2$. Our main result (Theorem~\ref{thmmain}) can be easily extended to this framework. Actually, this \textit{multiscale version} will be useful in order to derive the transversality condition on the final time (see Remark~\ref{remfindepreuvefreefinaltime}).
\end{remark}

\begin{remark}\label{rem8}
Another possible extension is to study dynamical systems on time scales with sampled-data control where the state $q$ and the constraint function $f$ are also sampled:
$$  q^\DD(t) = f ( \Phi_1(t) ,q^{\Phi_2} (t), u^{\Phi_3}(t) ), \quad \DD\text{-a.e. } t \in \T, $$
where $\T_1$, $\T_2$ and $\T_3$ are general time scales contained in $\T$, and $\Phi_1$, $\Phi_2$ and $\Phi_3$ are the corresponding mappings from $\T$ to $\TU$, $\T_2$ and $\T_3$ respectively. In particular, the setting of \cite{bour11} corresponds to the above framework with $\T = \R^+$ and $\T_1 = \T_2 = \T_3$ a general time scale. 
\end{remark}

Although this is not the main objective of our paper, we provide hereafter a result stating the existence of optimal solutions for $\OSCP$, under some appropriate compactness and convexity assumptions. Actually, if the existence of solutions is stated, the necessary conditions provided in Theorem~\ref{thmmain}, allowing to compute explicitly optimal sampled-data controls, may prove the uniqueness of the optimal solution. We refer to Section~\ref{sectionexemple} for examples.

Let $\mathcal{M}$ stand for the set of trajectories $q$, associated with $b \in \T$ and with a sampled-data control $u\in \L^\infty_{\TU}(\TU,\Omega)$, satisfying \eqref{DD-CS} $\DD$-a.e. on $[a,b)_\T$ and $g(q(a),q(b)) \in \S$. We define the set of extended velocities $ \mathcal{W}(t,q) = \{ ( f(t,q,u) , f^0(t,q,u) )^\top \mid u \in \Omega \} $ for every $(t,q) \in \T \times \R^n$.

\begin{theorem}\label{propexistence}
If $\Omega$ is compact, $\mathcal{M}$ is non-empty, $\Vert q \Vert_{\infty}+b \leq M$ for every $q \in \mathcal{M}$ and for some $M \geq 0$, and if $ \mathcal{W}(t,q)$ is convex for every $(t,q) \in \T \times \R^n$, then $\OSCP$ has at least one optimal solution.
\end{theorem}

The proof of Theorem~\ref{propexistence} is done in Section~\ref{proofexistence}. Note that, in this theorem, it suffices to assume that $g$ is continuous. Besides, the assumption on the boundedness of trajectories can be weakened, by assuming, for instance, that the extended dynamics have a sublinear growth at infinity (see, \textit{e.g.}, \cite{Cesari}; many other easy and standard extensions are possible).

\subsection{Pontryagin maximum principle for $\OSCP$}\label{secPMP}


\subsubsection{Preliminaries on convexity and stable $\Omega$-dense directions}
The orthogonal of the closed convex set $\S$ at a point $x \in \S$ is defined by
$$ \O_\S [x] = \{ x' \in \R^j \mid \forall x'' \in \S, \ \langle x',x''-x \rangle_{\R^j} \leq 0 \}. $$
It is a closed convex cone containing $0$.

We denote by $d_\S$ the distance function to $\S$ defined by $d_\S (x) = \inf_{x' \in \S} \Vert x-x' \Vert_{\R^j}$, for every $x\in\R^j$.
Recall that, for every $x \in \R^j$, there exists a unique element $\P_\S (x)\in\S$ (projection of $x$ onto $\S$) such that $d_\S (x) = \Vert x - \P_\S(x) \Vert_{\R^j}$. It is characterized by the property
$\langle x-\P_\S (x) , x'-\P_\S (x) \rangle_{\R^j} \leq 0$ for every $x'\in \S$. In particular, $x-\P_\S (x) \in \O_\S [\P_\S (x)]$. The function $\P_\S$ is $1$-Lipschitz continuous. We recall the following obvious lemmas.

\begin{lemma}\label{lemconvex}
Let $(x_k)_{k \in \N}$ be a sequence of points of $\R^j$ and $(\zeta_k)_{k \in \N}$ be a sequence of nonnegative real numbers such that $x_k \to x \in \S$ and $\zeta_k (x_k - \P_\S (x_k) ) \to x' \in \R^j$ as $k \to +\infty$.
Then $x' \in \O_\S [x]$.
\end{lemma}


\begin{lemma}
The function $d_\S^2 : x \mapsto d_\S (x)^2$ is differentiable on $\R^j$, with $\mathrm{d} d_\S^2 (x) ( x')= 2\langle x-\P_\S (x), x' \rangle_{\R^j}$.
\end{lemma}

Hereafter we recall the notion of stable $\Omega$-dense directions and we state an obvious lemma. We refer to \cite[Section~2.2]{bour11} for more details.

\begin{definition}\label{defstable}
Let $v \in \Omega$. A direction $y \in \Omega$ is said to be a stable $\Omega$-dense direction from $v$ if there exists $\varepsilon>0$ such that $0$ is not isolated in $\{ \alpha \in [0,1], \; v' + \alpha (y-v') \in \Omega \}$ for every $v' \in \overline{\mathrm{B}}_{\R^m} (v,\varepsilon) \cap \Omega$. The set of all stable $\Omega$-dense directions from $v$ is denoted by $\mathrm{D}^{\Omega}_{\mathrm{stab}} (v)$.
\end{definition}

\begin{lemma}\label{lemstable}
If $\Omega$ is convex, then $\mathrm{D}^{\Omega}_{\mathrm{stab}} (v) = \Omega$ for every $v \in \Omega$.
\end{lemma}

\subsubsection{Main result}
Recall that $g$ is said to be submersive at a point $(q_1,q_2) \in \R^n \times \R^n$ if the differential of $g$ at this point is surjective. We define the Hamiltonian $H:\T\times \R^n \times \R^n \times \R \times \R^m \rightarrow \R$ of $\OSCP$ by $H(t,q,p,p^0,u) = \langle p, f(t,q,u) \rangle_{\R^n} + p^0 f^0(t,q,u)$.

\begin{theorem}[Pontryagin maximum principle for $\OSCP$]\label{thmmain}
If a trajectory $q^*$, defined on $[a,b^*]_\T$ and associated with a sampled-data control $u^* \in \L^\infty_{\TU}(\TU,\Omega)$, is an optimal solution of $\OSCP$,
then there exists a nontrivial couple $(p,p^0)$, where $p \in \AC ([a,b^*]_\T,\R^n)$ (called adjoint vector) and $p^0 \leq 0$, such that the following conditions hold:
\begin{itemize}
\item Extremal equations:
\begin{equation}\label{extremal_equations}
q^{*\DD}(t) = \frac{\partial H}{\partial p} (t,q^*(t),p^\sigma(t),p^0,u^{*\Phi}(t)), \qquad
p^\DD(t) = - \frac{\partial H}{\partial q} (t,q^*(t),p^\sigma(t),p^0,u^{*\Phi}(t)) ,
\end{equation}
for $\DD$-a.e. $t\in[a,b^*)_\T$.
\item Maximization condition:
\begin{itemize}
\item For $\DD_1$-a.e. $s \in [a,b^*)_\TU \cap \RD_1$, we have
\begin{equation*}\label{maxcondition}
u^*(s) \in \argmax_{z \in \Omega} H(s,q^*(s),p(s),p^0,z) .
\end{equation*}
\item For every $r \in [a,b^*)_\TU \cap \RS_1$ such that $\sigma_1(r) \leq b^*$, we have
\begin{equation}\label{secondcondition}
\left\langle \int_{[r,\sigma_1(r))_\T} \frac{\partial H}{\partial u} (\t,q^*(\t),p^\sigma(\t),p^0,u^*(r)) \; \DD \t \; , \; y-u^*(r) \right\rangle_{\R^m}  \leq 0,
\end{equation}
for every $y \in \mathrm{D}^\Omega_{\mathrm{stab}} (u^*(r))$. In the case where $r \in [a,b^*)_\TU \cap \RS_1$ with $\sigma_1(r) > b^*$, the above maximization condition is still valid provided $\sigma_1(r)$ is replaced with $b^*$.
\end{itemize}
\item Transversality conditions on the adjoint vector: \\
If $g$ is submersive at $(q^*(a),q^*(b^*))$, then the nontrivial couple $(p,p^0)$ can be selected to satisfy
\begin{equation}\label{transv_cond}
p(a) = - \left( \frac{\partial g}{\partial q_1} (q^*(a),q^*(b^*)) \right)^\top  \psi,\qquad
p(b^*) =  \left( \frac{\partial g}{\partial q_2} (q^*(a),q^*(b^*)) \right)^\top  \psi,
\end{equation}
where $-\psi \in \O_\S [g (q^*(a),q^*(b^*))]$.
\item Transversality condition on the final time: \\
If the final time is left free in the optimal control problem $\OSCP$, if $b^*$ belongs to the interior of $\T$ (for the topology of $\R$), and if $f$ and $f^0$ are of class $\CC^1$ with respect to $t$ in a neighborhood of $b^*$, then the nontrivial couple $(p,p^0)$ can be moreover selected such that the Hamiltonian function $t \mapsto H(t,q^*(t),p(t),p^0,u^{*\Phi} (t)) $ coincides almost everywhere, in some neighborhood of $b^*$, with a continuous function vanishing at $t=b^*$. 

In particular, if $u^{*\Phi} (t)$ has a left-limit at $t=b^*$ (denoted by $u^{*\Phi} (b^*_-)$), then the transversality condition can be written as
$$ H(b^*,q^*(b^*),p(b^*),p^0,u^{*\Phi} (b^*_-)) = 0 .$$
\end{itemize}
\end{theorem}

Theorem~\ref{thmmain} is proved in Section~\ref{annexe}. Several remarks are in order.

\begin{remark}\label{remarknormaliser}
As is well known, the nontrivial couple $(p,p^0)$ of Theorem~\ref{thmmain}, which is a Lagrange multiplier, is defined up to a multiplicative scalar. Defining as usual an \textit{extremal} as a quadruple $(q,p,p^0,u)$ solution of the extremal equations~\eqref{extremal_equations}, an extremal is said to be \textit{normal} whenever $p^0\neq 0$ and \textit{abnormal} whenever $p^0=0$. In the normal case $p^0\neq 0$, it is usual to normalize the Lagrange multiplier so that $p^0=-1$.
\end{remark}
%

\begin{remark}\label{remencompass}
Theorem~\ref{thmmain} encompasses the time scale version of the PMP derived in \cite{bour11} when the control is permanent, that is, when $\TU = \T$. Indeed, in that case, for every $r \in [a,b^*)_\T \cap \RS$, $r \in [a,b^*)_\TU \cap \RS_1$ and $ \sigma_1(r) = \sigma (r) \leq b^*$. Then the condition~\eqref{secondcondition} can be written as the nonpositive gradient condition
\begin{equation*}
\left\langle \frac{\partial H}{\partial u} (r,q^*(r),p(\sigma(r)),p^0,u^*(r)) \; , \; y-u^*(r) \right\rangle_{\R^m} \leq 0,
\end{equation*}
for every $y \in \mathrm{D}^\Omega_{\mathrm{stab}} (u^*(r))$. Moreover, in the case of a free final time, under the assumptions made in the fourth item of Theorem~\ref{thmmain}, $b^*$ also belongs to the interior of $\TU = \T$, and then in that case we recover the classical condition
$$ \max_{z \in \Omega} H(b^*,q^*(b^*),p(b^*),p^0,z) = 0. $$
\textit{A fortiori}, Theorem~\ref{thmmain} encompasses both the classical continuous-time and discrete-time versions of the PMP, that is respectively, when $\T = \TU = \R^+$ and $\T = \TU = \N$.
\end{remark}

\begin{remark}
If the final time is free, under the assumptions made in the fourth item of Theorem~\ref{thmmain}, and if moreover $b^* \notin \TU$ (or if $b^*$ is a left-scattered point of $\TU$), $u^{*\Phi}$ is constant on $[\rho_1(b^*),b^*)_\T$ where $\rho_1 (b^*) = \max \{ s \in \TU \mid s < b^* \} < b^*$ (and thus, in particular, $u^{*\Phi}(t)$ has a left-limit at $t=b^*$), and therefore $ H(b^*,q^*(b^*),p(b^*),p^0,u^* ( \rho_1(b^*) )) = 0 $. This is similar to the situation of Theorem~\ref{thmmainintro}.
\end{remark}

\begin{remark}\label{remarkconditionsterminales}
Let us describe some typical situations of terminal conditions $g(q(a),q(b)) \in \S$ in $\OSCP$, and of the corresponding transversality conditions on the adjoint vector.
\begin{itemize}
\item If the initial and final points are fixed in $\OSCP$, that is, if we impose $q(a) = q_a$ and $q(b) = q_b$, then $j=2n$, $g(q_1,q_2) = (q_1,q_2)$ and $\S = \{ q_a \} \times \{ q_b \}$. In that case, the transversality conditions on the adjoint vector give no additional information.
\item If the initial point is fixed, that is, if we impose $q(a) = q_a$, and if the final point is left free in $\OSCP$, then $j=n$, $g(q_1,q_2) = q_1$ and $\S = \{ q_a \} $. In that case, the transversality conditions on the adjoint vector imply that $p(b^*) = 0$. Moreover, we have $p^0 \neq 0$\footnote{Indeed, if $p^0 =0$, then the adjoint vector $p$ is trivial from the extremal equation and from the final condition $p(b^*)=0$. This leads to a contradiction since the couple $(p,p^0)$ has to be nontrivial.} and we can normalize the Lagrange multiplier so that $p^0=-1$ (see Remark~\ref{remarknormaliser}).
\item If the initial point is fixed, that is, if we impose $q(a) = q_a$, and if the final point is subject to the constraint $G(q(b)) \in (\R^+)^k$ in $\OSCP$, where $G = (G^1,\ldots,G^k) : \R^n \rightarrow \R^k$ is of class $\CC^1$ and is submersive at any point $q_2 \in G^{-1} ((\R^+)^k)$, then $j=n+k$, $g(q_1,q_2) = (q_1,G(q_2))$ and $\S = \{ q_a \} \times (\R^+)^k $. The transversality conditions on the adjoint vector imply that
$$ p(b^*) = \di \sum_{i=1}^k \lambda_i \partial_{q_2} G^i (q^*(b^*)), $$
with $\lambda_i \geq 0$, $i=1,\ldots,k$.
\item If the periodic condition $q(a)=q(b)$ is imposed in $\OSCP$, then $j=n$, $g(q_1,q_2) = q_1- q_2$ and $\S = \{ 0 \}$. In that case, the transversality conditions on the adjoint vector yield that $p(a) = p(b^*)$.
\end{itemize}
We stress that, in all examples above, the function $g$ is indeed a submersion.
\end{remark}

\begin{remark}\label{remarknonsubmersive}
In the case where $g$ is not submersive at $(q^*(a),q^*(b^*))$, to obtain transversality conditions on the adjoint vector, Theorem~\ref{thmmain} can be reformulated as follows:
\begin{quote}
\textit{If a trajectory $q^*$, defined on $[a,b^*]_\T$ and associated with a sampled-data control $u^* \in \L^\infty_{\TU}(\TU,\Omega)$, is an optimal solution of $\OSCP$,
then there exists a nontrivial couple $(\psi,p^0) \in \R^j \times \R$, with $-\psi \in \O_\S [g (q^*(a),q^*(b^*))]$ and $p^0 \leq 0$, and there exists $p \in \AC ([a,b^*]_\T,\R^n)$ such that the extremal equations, the maximization conditions and the transversality conditions are satisfied.}
\end{quote}
However, with this formulation, the couple $(p,p^0)$ may be trivial\footnote{Indeed, if $p^0 =0$ and $\psi$ belongs to the kernel of $(\partial_{q_2} g  (q^*(a),q^*(b^*)) )^\top$, then the couple $(p,p^0)$ is trivial. This situation leads to a contradiction if $g$ is submersive at $(q^*(a),q^*(b^*))$.} and, as a consequence, the result may not provide any information. We refer to Sections~\ref{section46} and \ref{section47} for more details.
\end{remark}

\begin{remark}
In this paper, the closedness of $\Omega$ is used in a crucial way in our proof of the PMP. Indeed, the closure of $\Omega$ allows us to define the Ekeland functional on a complete metric space (see Section~\ref{section41}). However, if the initial point is fixed, that is, if we impose $q(a) = q_a$, and if the final point is left free in $\OSCP$, then Theorem~\ref{thmmain} can be proved with a simple calculus of variations, without using the Ekeland variational principle. In this particular case, the closedness assumption can be removed.
\end{remark}

\begin{remark}\label{remarkbolza}
If the cost functional to be minimized in $\OSCP$ is $ \int_{[a,b)_\T} f^0 (\t,q (\t), u^\Phi(\t) ) \, \DD \t + \ell (b,q(a),q(b))$, where $\ell : \T \times \R^n \times \R^n \rightarrow \R$ is continuous, of class $\CC^1$ in $(q_1,q_2)$, then the transversality conditions on the adjoint vector become 
\begin{multline*}
p(a) =  - \left( \partial_{q_1} g (q^*(a),q^*(b^*)) \right)^\top  \psi - p^0 \partial_{q_1} \ell  (b^*,q^*(a),q^*(b^*)), \\ \text{and } p(b^*) =  \left( \partial_{q_2} g (q^*(a),q^*(b^*)) \right)^\top  \psi + p^0 \partial_{q_2} \ell (b^*,q^*(a),q^*(b^*)) . $$
\end{multline*}
Moreover, in the fourth item of Theorem~\ref{thmmain}, if $\ell$ is of class $\CC^1$ in a neighborhood of $b^*$, the transversality condition on the final time must be replaced by:
\begin{quote}
\textit{The nontrivial couple $(p,p^0)$ can be selected such that the Hamiltonian function $ t \mapsto H(t,q^*(t),p(t),p^0,u^{*\Phi} (t)) $ coincides almost everywhere, in some neighborhood of $b^*$, with a continuous function that is equal to $-p^0 \partial_b \ell (b^*,q^*(a),q^*(b^*))$ at $t=b^*$.}
\end{quote}
To prove this claim, it suffices to modify accordingly the Ekeland functional in the proof of Theorem~\ref{thmmain} (see Section~\ref{section41}).
\end{remark}

\section{Applications and further comments}\label{sectionapplication}

In this section, we first give, in Section~\ref{sectionexemple}, a very simple example of an optimal control problem on time scales with sampled-data control, that we treat in details and on which all computations are explicit. The interest is that this example provides as well a simple situation where it is evident that some of the properties that are valid in the classical continous-time PMP do not hold anymore in the time-scale context. We gather these remarks in Section~\ref{nonextension}.

\subsection{A model for optimal consumption with sampled-data control}\label{sectionexemple}

Throughout this subsection, $\T$ and $\TU$ are two time scales, unbounded above, satisfying $\TU \subset \T$, $\min \T = \min \TU = 0$ and $12 \in \T$. In the sequel, we study the following one-dimensional dynamical system with sampled-data control on time scales:
\begin{equation}\label{exequationq}
q^\DD (t) = u^\Phi(t)q(t), \qquad \DD\text{-a.e. } t \in [0,12)_\T,
\end{equation}
with the initial condition $q(0)=1$, and subject to the constraint $u(t) \in [0,1]$ for $\DD_1$-a.e. $t \in [0,12)_\TU$. Since the final time $b=12$ is fixed, we can assume that $12 \in \TU$ without loss of generality.

The above model is a classical model for the evolution of a controlled output of a factory during the time interval $[0,12]_\T$ (corresponding to the twelve months of a year). Precisely, $q(t) \in \R$ stands for the output at time $t \in \T$ and $u(t) \in [0,1]$ stands for the fraction of the output reinvested at each controlling time $t \in \TU$. We assume that this fraction is sampled at each sampling time $t \in \TU$ such that $t \in \RS_1$ and $\sigma (t) < \sigma_1(t)$, on the corresponding sampling interval $[t,\sigma_1(t))_\T$. In the sequel, our goal is to maximize the total consumption
\begin{equation}\label{exC}
C(u) = \int_{[0,12)_\T} (1-u^\Phi(\t))q(\t) \,\DD \t.
\end{equation}
In other words, our aim is to maximize the quantity of the output that we do not reinvest.

\begin{remark}
In the continuous case and with a permanent control (that is, with $\T = \TU = \R^+$), the above optimal control problem is a very well-known application of the classical Pontryagin maximum principle. We refer for example to~\cite[Exercice~2.3.3. p.82]{seie} or \cite[p.92]{JBHU}. In this section, our aim is to solve this optimal control problem in cases where the control is nonpermanent and sampled. We first treat some examples in the continuous-time setting $\T = \R^+$ and then in the discrete-time setting $\T = \N$.
\end{remark}

The above optimal sampled-data control problem corresponds to $\OSCP$ with $n=m=j=1$, $a=0$, $b=12$ (fixed final time), $\Omega = [0,1]$ (convex), $ g(q_1,q_2)=q_1$ and $\S =  \{ 1 \}$ (fixed initial value $q(0)=1$ and free final value $q(12)$), $f(t,q,u) = uq$ and $f^0(t,q,u) = (u-1)q$ (our aim is to minimize $-C(u)$).

Since $f$ and $f^0$ are affine in $u$ and since $\Omega$ is compact, Theorem~\ref{propexistence} asserts that $\OSCP$ admits an optimal solution $q^*$, defined on $[0,12]_\T$ and associated with a sampled-data control $u^* \in \L^\infty_\TU( \TU,[0,1])$. We now apply Theorem~\ref{thmmain} in order to compute explicitly the values of $u^*$ at each controlling time $t \in [0,12)_\TU$. The nontrivial couple $(p,p^0)$ satisfies $p^0 = -1$ and $p(12)=0$ (see Remark~\ref{remarkconditionsterminales}). The adjoint vector $p \in \AC ([0,12]_\T,\R)$ is a solution of
\begin{equation}\label{exequationp}
p^\DD(t) = - u^{*\Phi}(t) p^\sigma(t) + u^{*\Phi}(t) -1 , \qquad \DD\text{-a.e. } t \in [0,12)_\T.
\end{equation}
Moreover, one has the following maximization conditions:
\begin{enumerate}
\item for $\DD_1$-a.e. $s \in [0,12)_\TU \cap \RD_1$,
\begin{equation}\label{exequationmax1}
u^*(s) \in \argmax_{z \in [0,1]} \,( z (p(s)-1)+1 ) q^*(s) .
\end{equation}
\item for every $r \in [0,12)_\TU \cap \RS_1$,
\begin{equation}\label{exequationmax2}
\int_{[r,\sigma_1(r))_\T} q^*(\t) (p^\sigma (\t) - 1) (y-u^*(r)) \; \DD \t \leq 0,
\end{equation}
for every $y \in \mathrm{D}^\Omega_{\mathrm{stab}}(u^*(r)) = \Omega = [0,1]$ (since $\Omega$ is convex, see Lemma~\ref{lemstable}).
\end{enumerate}
Since $q^*$ is a solution of~\eqref{exequationq} and satisfies $q^*(0)=1$, one can easily see that $q^*$ is monotonically increasing on $[0,12]_\T$ and then $q^*$ has positive values. From~\eqref{exequationp} and since $p(12)=0$, one can easily obtain that $p$ is monotonically decreasing on $[0,12]_\T$ and then $p$ has nonnegative values.

\subsubsection{Continuous-time setting $\T = \R^+$}\label{sectionexcontinu}

\paragraph*{Case $\TU = \R^+$ (permanent control).} Since $q^*$ has positive values, one can conclude from~\eqref{exequationp} and~\eqref{exequationmax1} that:
$$ u^*(t) = \left\lbrace \begin{array}{ll}
1 & \text{ for a.e. } t \text{ such that } p(t)-1 > 0, \\
0 & \text{ for a.e. } t \text{ such that } p(t)-1 < 0,
\end{array} \right. \; 
\dot{p}(t) = \left\lbrace \begin{array}{ll}
-p(t) & \text{ for a.e. } t \text{ such that } u^*(t)=1, \\
-1 & \text{ for a.e. } t \text{ such that } u^*(t)=0.
\end{array} \right. $$
Since $p(12)=0$, one can easily prove that the optimal (permanent) control $u$ is unique and given by $ u^*(t) = 1 $ for a.e. $t \in [0,11)$ and $ u^*(t) = 0$ for a.e. $t \in [11,12)$. The associated optimal consumption is $ C(u^*) = e^{11} \simeq 59874.142  $. We refer to~\cite[Exercice~2.3.3. p.82]{seie} or \cite[p.92]{JBHU} for more details.

\paragraph*{Case $\TU $ discrete (sampled-data control).} Solving the differential equations~\eqref{exequationq} and \eqref{exequationp} leads to $ q^*(t) = q^*(\sigma_1(r)) e^{u^*(r)(t-\sigma_1(r))}$ and
$$ p^\sigma (t) = p(t) = \left\lbrace \begin{array}{ll} 
p(\sigma_1(r)) e^{-u^*(r)(t-\sigma_1(r))} + \frac{u^*(r)-1}{u^*(r)} ( 1 - e^{-u^*(r)(t-\sigma_1(r))} ) & \text{if } u^*(r) \neq 0, \\
p(\sigma_1(r)) + \sigma_1 (r) - t & \text{if } u^*(r) = 0,
\end{array} \right. $$
for every $r \in [0,12)_{\TU}$ and every $t \in [r,\sigma_1(r)]$. Then, \eqref{exequationmax2} can be written as $ q^*( \sigma_1(r) ) (y-u^*(r)) \Gamma_r (u^*(r)) \leq 0$, where $\Gamma_r : [0,1] \rightarrow \R$ is the continuous function given by:
$$ \Gamma_r (x) = \left\lbrace \begin{array}{ll}
\dfrac{1}{x^2} \Big[e^{-\mu_1(r)x} -  \Big(  1+\mu_1(r)x(x(1-p(\sigma_1(r)))-1) \Big) \Big] & \text{if } x \neq 0, \\
\mu_1(r) \left( p(\sigma_1(r)) + \dfrac{\mu_1(r)}{2} -1 \right) & \text{if } x = 0.
\end{array} \right. $$
Since~\eqref{exequationmax2} holds true for every $y \in [0,1]$ and since $q^*$ has positive values, the following properties are satisfied for every $r \in [0,12)_{\TU}$:
\begin{itemize}
\item if $\Gamma_r$ has negative values on $(0,1]$, then $u^*(r) =0$;
\item if $\Gamma_r$ has positive values on $[0,1)$, then $u^*(r) =1$;
\item if $\Gamma_r (0) > 0$ and $\Gamma_r (1) < 0 $, then $u^*(r) \in (0,1)$ is a solution of the nonlinear equation $\Gamma_r (x) = 0$.
\end{itemize}
Note that $\Gamma_r$ depends only on $\mu_1(r)$ and $p(\sigma_1(r))$. As a consequence, the knowledge of the value $p(12)=0$ and the above properties allow to compute $u^*(r_0)$ where $r_0$ is the element of $[0,12)_{\TU}$ such that $\sigma_1 (r_0) =12$ (and $\mu_1 (r_0) = 12-r_0$). Then, the knowledge of $u^*(r_0)$ allows to compute $p(r_0)$ from~\eqref{exequationp}. Then, the knowledge of $p(r_0)$ and the above properties allow to compute $u^*(r_1)$ where $r_1$ is the element of $[0,12)_{\TU}$ such that $\sigma_1 (r_1) =r_0$ (and $\mu_1 (r_1) = r_0-r_1$), etc. Actually, this recursive procedure allows to compute $u^*(r)$ for every $r \in [0,12)_{\TU}$. Numerically, we obtain the following results:
\begin{center}
\begin{tabular}{|c|c|c|}
\hline
$\TU = \N$ & $u^*(t) = \left\lbrace \begin{array}{ll}
1 & \text{ if }  t \in \{ 0,1,\ldots, 10 \} \\
0 & \text{ if }  t =11
\end{array} \right.$ & $C(u^*)=e^{11} \simeq 59874.142$ \\ \hline
$\TU = 2\N$ & $u^*(t) = \left\lbrace \begin{array}{ll}
1 & \text{ if }  t \in \{ 0,2,4,6,8 \} \\
0 & \text{ if }  t =10
\end{array} \right.$ & $C(u^*)= 2 e^{10} \simeq 44052.932 $ \\ \hline
$\TU = 3\N$ & $u^*(t) = \left\lbrace \begin{array}{ll}
1 & \text{ if }  t \in \{ 0,3,6 \} \\
0.4536 & \text{ if }  t =9
\end{array} \right.$ & $C(u^*)  \simeq 28299.767 $ \\ \hline
$\TU = 4\N$ & $u^*(t) = \left\lbrace \begin{array}{ll}
1 & \text{ if }  t \in \{ 0,4 \} \\
0.6392 & \text{ if }  t =8
\end{array} \right.$ & $C(u^*)  \simeq 20013.885$ \\ \hline
$\TU = 9\N$ & $u^*(t) = \left\lbrace \begin{array}{ll}
1 & \text{ if }  t = 0 \\
0.4536 & \text{ if }  t =9
\end{array} \right.$ & $C(u^*) \simeq 28299.767$ \\ \hline
$\TU = 12\N$ & $u^*(0) \simeq 0.9083$ & $C(u^*) \simeq 5467.24$ \\ \hline
$\TU = 12\N \cup \{ 10 , 11.5 \} $ & $u^*(t) = \left\lbrace \begin{array}{ll}
1 & \text{ if }  t = 0 \\
0.9072 & \text{ if } t = 10 \\
0 & \text{ if }  t =11.5
\end{array} \right.$ & $C(u^*)  \simeq 49476.604 $ \\ \hline
\end{tabular}
\end{center}

\begin{remark}
In this example, Theorem~\ref{propexistence} states the existence of an optimal solution. In all cases above studied, the Pontryagin maximum principle proves that the optimal solution is moreover unique.
\end{remark}

\begin{remark}
The case $\TU = \N$ can be easily deduced from the permanent case $\TU = \R^+$ (see Remark~\ref{remarkinclusiondeuxT1}). Similarly, the case $\TU = 9\N$ can be deduced from the case $\TU = 3\N$.
\end{remark}

\begin{remark}
The case $\TU = 12\N$ corresponds to an optimal parameter problem (see Remark~\ref{remarkparam}).
\end{remark}

\begin{remark}\label{remarkcontinu2N}
For the needs of Section~\ref{nonextension}, let us give some details on the case $\TU = 2\N$. In that case, $q^*(t)=e^t$ on $[0,10]$ and $q^*(t)=e^{10}$ on $[10,12]$. Moreover, $p(t) = 2e^{10-t}$ on $[0,10]$ and $p(t)=12-t$ on $[10,12]$. Hence, for every $t \in [10,11)$, $\argmax_{z \in \Omega} H(t,q^*(t),p(t),p^0,z) = \{ 1 \}$. The Hamiltonian $ t \mapsto H(t,q^*(t),p(t),p^0,u^{*\Phi}(t)) $ is equal to $t \mapsto 2e^{10}$ a.e. on $[0,10)$ and is equal to $t \mapsto e^{10}$ a.e. on $[10,12)$. Finally, the maximized Hamiltonian $ t \mapsto \max_{z \in \Omega} H(t,q^*(t),p^\sigma(t),p^0,z)$ is equal to $t \mapsto 2e^{10}$ on $[0,10)$, to $t \mapsto (12-t) e^{10}$ on $[10,11)$ and to $t \mapsto e^{10}$ on $[11,12)$.
\end{remark}

\paragraph*{Case $\TU$ hybrid (sampled-data control).} In this paragraph, we study the hybrid case $\TU = [0,6] \cup \{ 10 \} \cup [11.5,+\infty)$. Similarly to the permanent case $\TU = \R^+$, one can easily conclude from~\eqref{exequationmax1} that:
$$ u^*(t) = \left\lbrace \begin{array}{ll}
1 & \text{ for a.e. } t \in [0,6) \cup [11.5,12) \text{ such that } p(t)-1 > 0, \\
0 & \text{ for a.e. } t \in [0,6) \cup [11.5,12) \text{ such that } p(t)-1 < 0.
\end{array} \right. $$
Since $p(12)=0$, one can easily prove that $u^*(t) = 0$ for a.e. $t \in [11.5,12)$, and then $p(11.5)=0.5$. From the knowledge of $p(11.5)=0.5$ and using similar arguments than in the previous paragraph, one can compute $u^*(10) \simeq 0.9072$. From~\eqref{exequationp}, it gives $p(10) \simeq 2.2462$. From the knowledge of $p(10) \simeq 2.2462 $, one can compute $u^*(6)=1$. From~\eqref{exequationp}, it gives $p(6) \simeq 122.6402$. Since $p$ is monotonically decreasing, we conclude that $p(t)-1 > 0$ for every $t \in [0,6)$. We finally conclude that the optimal sampled-data control $u^*$ is unique and is given by $u^*(t) = 1$ for a.e. $t \in [0,6)$, $u^*(6)=1$, $u^*(10) \simeq 0.9072$ and $u^*(t) = 0$ for a.e. $t \in [11.5,12)$. The associated optimal consumption is $ C(u^*) \simeq 49476.604 $.

\begin{remark}
The discrete case $\TU = 12\N \cup \{ 10 , 11.5 \}$ can now be seen as a consequence of the hybrid case $\TU = [0,6] \cup \{ 10 \} \cup [11.5,+\infty)$ (see Remark~\ref{remarkinclusiondeuxT1}).
\end{remark}

\subsubsection{Discrete-time setting $\T = \N$}\label{sectionexdiscret}
In the discrete-time setting $\T = \N$, similarly to the continuous-time one, we can prove that \eqref{exequationmax2} can be written as $ q^*(r)(y-u^*(r)) \Lambda_r (u^*(r)) \leq 0$, where $\Lambda_r : [0,1] \rightarrow \R$ is the continuous function given by:
$$ \Lambda_r (x) = \left\lbrace \begin{array}{ll}
\dfrac{1}{x^2}  \Big[ 1- (1+x)^{\mu_1(r)-1} \Big(  1+\mu_1(r)x(x(1-p(\sigma_1(r)))-1)+x \Big)  \Big] & \text{if } x \neq 0, \\
\mu_1(r) \left( p(\sigma_1(r)) + \dfrac{\mu_1(r)-3}{2} \right) & \text{if } x = 0.
\end{array} \right. $$
The same recursive procedure (than in the previous section) allows to compute $u^*(r)$ for every $r \in [0,12)_{\TU}$. Numerically, we obtain the following results:
\begin{center}
\begin{tabular}{|c|c|c|}
\hline
$\TU = 2\N$ & $u^*(t) = \left\lbrace \begin{array}{ll}
1 & \text{ if }  t \in \{ 0,2,4,6,8 \} \\
0 & \text{ if }  t =10
\end{array} \right.$ & $C(u^*)= 2^{11} = 2048 $ \\ \hline
$\TU = 3\N$ & $u^*(t) = \left\lbrace \begin{array}{ll}
1 & \text{ if }  t \in \{ 0,3,6 \} \\
0 & \text{ if }  t =9
\end{array} \right.$ & $C(u^*) = 3 \cdot 2^9 = 1536 $ \\ \hline
$\TU = 4\N$ & $u^*(t) = \left\lbrace \begin{array}{ll}
1 & \text{ if }  t \in \{ 0,4 \} \\
0.2886 & \text{ if }  t =8
\end{array} \right.$ & $C(u^*)  \simeq 1108.882$ \\ \hline
$\TU = 6\N$ & $u^*(t) = \left\lbrace \begin{array}{ll}
1 & \text{ if }  t =0 \\
0.5725 & \text{ if }  t =6
\end{array} \right.$ & $C(u^*) \simeq 674.787 $ \\ \hline
$\TU = 12\N$ & $u^*(0) \simeq 0.8145$ & $C(u^*) \simeq 159.647$ \\ \hline
\end{tabular}
\end{center}

\begin{remark}\label{remarknonunique}
In the permanent discrete case $\T = \TU = \N$, the optimal control is not unique. Indeed, every control $u^*$ satisfying $u^*(t)=1$ for every $t \in \{ 0,1,\ldots,9 \}$, $u^*(10) \in [0,1]$ and $u^*(11)=0$ is optimal. Let us give some details on the proof and let us note that the PMP does not provide any constraint on the value $u^*(10)$ in that case.

With $r=11$ and $p(12)=0$, $\Lambda_r$ has negative values on $(0,1]$ (constantly equal to $-1$), then $u^*(11) = 0$ from \eqref{exequationmax2}. From~\eqref{exequationp}, we compute $p(11)=1$. With $r=10$ and $p(11)=1$, $\Lambda_r$ is constantly equal to $0$ and then \eqref{exequationmax2} does not provide any constraint on the value $u^*(10)$. Actually, it does not matter since $p(10)$ can still be computed from~\eqref{exequationp} ($p(10)=2$) and the recursive procedure can be pursued. We obtain $u^*(t)=1$ for every $t \in \{ 0,1,\ldots,9 \}$. From~\eqref{exC}, we have
$$ C(u^*) = (1-u^*(10)) q^*(10) + q^*(11) = (1-u^*(10)) \cdot 2^{10} + (1+u^*(10)) q^*(10) = 2 \cdot 2^{10} = 2^{11}. $$
Finally, if $u^*(11)=0$ and $u^*(t)=1$ for every $t \in \{ 0,1,\ldots,9 \}$, the value of $u^*(10) \in [0,1]$ does not influence $C(u^*)$. This concludes the proof and the remark.
\end{remark}

\begin{remark}
The case $\TU = 2\N$ can be seen as a consequence of the permanent case $\TU = \N$ (see Remarks~\ref{remarkinclusiondeuxT1} and \ref{remarknonunique}).
\end{remark}

\begin{remark}
The case $\TU = 12\N$ corresponds to an optimal parameter problem (see Remark~\ref{remarkparam}).
\end{remark}

\subsection{Non-extension of several classical properties}\label{nonextension}
In this section, we recall some basic properties that occur in classical optimal control theory in the continuous-time setting and with a permament control, that is, with $\T = \TU = \R^+$. Our aim is to discuss their extension (or their failure) to the general time scale setting and to the nonpermament control case. We will provide several counterexamples in the discrete-time setting with a permanent control ($\T = \TU = \N$) and in the continuous-time setting with a nonpermanent control ($\TU  \varsubsetneq \T = \R^+$). In the following paragraphs, except the last one, the final time $b \in \T$ can be fixed or not.

\paragraph*{Pointwise maximization condition of the Hamiltonian.}
In the case $\T = \TU = \R^+$, it is well known that an optimal (permanent) control $u^*$ satisfies the maximization condition $ u^*(t) \in \argmax_{z \in \Omega} H(t,q^*(t),p(t),p^0,z)$ for a.e. $t \in [0,b^*)$. We refer to \cite[Example~7]{bour11} for a counterexample showing the failure of this maximization condition in the case $\T = \TU = \N$. We refer to Remark~\ref{remarkcontinu2N} for a counterexample in the case $\TU  \varsubsetneq \T = \R^+$.

%

\paragraph*{Continuity of the Hamiltonian.}
In the case $\T = \TU = \R^+$, it is well known that the Hamiltonian function $ t \mapsto H(t,q^*(t),p(t),p^0,u^*(t)) $ coincides almost everywhere on $[0,b^*]$ with the continuous function $ t \mapsto \max_{z \in \Omega} H(t,q^*(t),p(t),p^0,z)$. Remark~\ref{remarkcontinu2N} provides a counterexample showing the failure of this regularity property in the case $\TU  \varsubsetneq \T = \R^+$.

\begin{remark}
Nevertheless, in the case of a free final time, under the assumptions of the fourth item of Theorem~\ref{thmmain}, the Hamiltonian function $ t \mapsto H(t,q^*(t),p(t),p^0,u^{*\Phi}(t)) $ coincides almost everywhere, in some neighborhood of $b^*$, with a continuous function.
\end{remark}


%
%

\paragraph*{The autonomous case.}
In the case $\T = \TU = \R^+$, if the Hamiltonian $H$ is autonomous (that is, does not depend on $t$), it is well known that the function $t \mapsto H(q^*(t),p(t),p^0,u^*(t))$ is almost everywhere constant on $[0,b^*]$, this constant being equal to the maximized Hamiltonian. We refer to \cite[Example~8]{bour11} for a counterexample showing the failure of this constantness property in the case $\T = \TU = \N$, and we refer to Remark~\ref{remarkcontinu2N} for a counterexample in the case $\TU  \varsubsetneq \T = \R^+$ (and there, even the maximized autonomous Hamiltonian is not constant).

\paragraph*{Saturated constraint set $\Omega$ for Hamiltonian affine in $u$.}
In this paragraph, we assume that $\Omega$ is convex. In the case $\T = \TU = \R^+$, if the Hamiltonian is \textit{affine in $u$}, that is, if it can be written as
$$ H(t,q,p,p^0,u) = \langle H_1 (t,q,p,p^0) , u \rangle_{\R^m} + H_2 (t,q,p,p^0), $$
one can easily prove that $H_1 (t,q^*(t),p(t),p^0) \in \O_\Omega [u^*(t)]$ for almost every $t \in [0,b^*)$. It follows that an optimal (permament) control $u^*$ must take its values at the boundary of $\Omega$ (saturation of the constraints) for almost every $t \in [0,b^*)$ such that $H_1 (t,q^*(t),p(t),p^0) \neq 0_{\R^m}$.

\begin{remark}
This classical property can be extended to the case $\T = \TU =\N$. Indeed, in that case, the nonpositive gradient condition is given by
$$ \left\langle \frac{\partial H}{\partial u} (t,q^*(t),p(t+1),p^0,u^*(t)) , y -u^*(t) \right\rangle_{\R^m} = \langle H_1 (t,q^*(t),p(t+1),p^0) , y - u^*(t) \rangle_{\R^m} \leq 0, $$
for every $y \in \Omega$, that is, $H_1 (t,q^*(t),p(t+1),p^0) \in \O_\Omega [u^*(t)]$.
\end{remark}

\begin{remark}
Remark~\ref{remarknonunique} provides an interesting example in the case $\T = \TU = \N$. Indeed, in that case, the control defined by $u^*(t) =0$ for every $t \in \{ 0, \ldots ,9 \}$, $u^*(10) = 1/2$ and $u^*(11)=1$ is an optimal (permanent) control. However, it does not saturate the constraint set $\Omega$ at $t=10$. It is not a surprise since, in that case, $H_1(t,q^*(t),p(t+1),p^0) = 0$ at $t=10$.

Note that, in the case $\TU = 4\N$, Section~\ref{sectionexdiscret} provides a counterexample showing the failure of this classical property in the case $\TU \varsubsetneq \T = \N$. Similarly, Section~\ref{sectionexcontinu} in the case $\TU = 3\N$ provides a counterexample in the case $\TU \varsubsetneq \T = \R^+$.
\end{remark}

\begin{remark}
Figure \ref{reffig} represents the values of the optimal sampled-data control $u^*$ of Section~\ref{sectionexcontinu}, in the case $\TU = 12\N \cup \{ \lambda \}$, where $\lambda$ is a parameter evolving in $(0,12)$. In that case, $u^*(0)$ (resp., $u^*(\lambda)$) saturates the constraint set $\Omega = [0,1]$ approximately for $\lambda \in (0,11.9245)$ (resp., for $\lambda \in (9.9866,12)$).
\end{remark}

\begin{figure}[h]
\begin{center}
\includegraphics[scale=0.25]{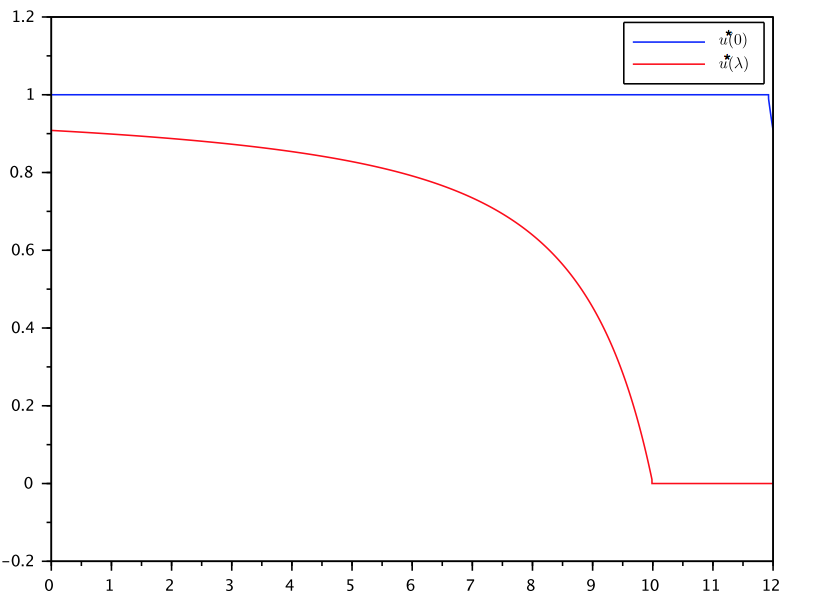}
\end{center}
\caption{Optimal sampled-data control $u^*$.}\label{reffig}
\end{figure}

\paragraph*{Vanishing of the maximized Hamiltonian at the final time.}
In this paragraph, we assume that the final time is left free.
In the case $\T = \TU = \R^+$, under the assumptions of the fourth item of Theorem~\ref{thmmain}, it is well known that the maximized Hamiltonian vanishes at $t=b^*$ (see Remark~\ref{remencompass}). In Theorem~\ref{thmmain}, we have established that this property is still valid in the time scale setting under some appropriate conditions, the main one being that $b^*$ must belong to the interior of $\T$. In the discrete case $\T = \N$, the interior of $\T$ is empty and then the latter assumption is never satisfied. The Hamiltonian at the final time may then not vanish, and we refer to \cite[Example~8]{bour11} for a counterexample with $\T = \TU = \N$.


\section{Proofs}\label{annexe}

The section is structured as follows. Subsections~\ref{annexe1}, \ref{sectionvariations} and \ref{annexe3} are devoted to the proof of Theorem~\ref{thmmain}. In Subsection~\ref{annexe1}, we recall some known Cauchy-Lipschitz results on time scales and we establish some preliminary results on the relations between $u$ and $u^\Phi$. In Subsection~\ref{sectionvariations}, we introduce appropriate needle-like variations of the control. Finally, in Subsection~\ref{annexe3}, we apply the Ekeland variational principle to an adequate functional in an appropriate complete metric space, and then we prove the PMP. In Subsection~\ref{proofexistence}  (that the reader can read independently of the rest of Section~\ref{annexe}), we detail the proof of Theorem~\ref{propexistence}.

\subsection{Preliminaries}\label{annexe1}
\subsubsection{Relations between $u$ and $u^\Phi$}\label{annexe1uetuphi}
We start with a lemma whose arguments of proof will be used several times.

\begin{lemma}\label{lemnumber1}
Let $m \in \N^*$ and let $c < d$ be two elements of $\T$ with $c \in \TU$. Let $u$, $v : [c,d)_{\TU} \rightarrow \R^m$ be two functions. Then, $u = v$ $\DD_1$-a.e. on $[c,d)_{\TU}$ if and only if $u^\Phi = v^\Phi$ $\DD$-a.e. on $[c,d)_{\T}$.
\end{lemma}

\begin{proof}
Without loss of generality, we can assume that $v$ is constant equal to $0_{\R^m}$. Let us define $ A = \{ t \in [c,d)_{\TU} \,| \,u(t) \neq 0_{\R^m} \} $ and $ B = \{ t \in [c,d)_{\T} \,| \,u^\Phi(t) \neq 0_{\R^m} \} $. Since $\Phi (t) = t$ for every $t \in \TU$, the inclusion $A \subset B$ holds. Firstly, let us assume that $\mu_{\DD_1}(A) = 0$. Hence, $A \subset \RD_1$ and $\mu_{\DD_1}(A) = \mu_L (A) = 0$. Since $A \subset \RD_1 \subset \RD$, we deduce that $\mu_\DD (A) = \mu_L (A) = 0$. On the other hand, for every $t \in B$, $\Phi(t) \in A \subset \RD_1$, then $t \in \TU$ and $t = \Phi(t) \in A$. We conclude that $A = B$ and $\mu_\Delta (B) = 0$. Secondly, let us assume that $\mu_\Delta (B) = 0$ and $\mu_{\DD_1}(A) > 0$. Since $A \subset B$, we deduce that $\mu_\Delta (A) = 0$, $A \subset \RD$ and $\mu_\Delta (A) = \mu_L (A) = 0$. Since $\mu_L (A) = 0$ and $\mu_{\DD_1}(A) > 0$, we conclude that there exists $t_0 \in \RS_1 \cap A \subset \RS_1 \cap \RD$. Consequently, $u^\Phi$ is constant (different of $0_{\R^m}$) on $[t_0, \min ( \sigma_1(t_0),d))_\T$. As a consequence, $\mu_\Delta (B) \geq \min ( \sigma_1(t_0),d) - t_0 > 0$. This raises a contradiction.
\end{proof}

\begin{proposition}\label{propnumber1}
Let $m \in \N^*$ and let $c < d$ be two elements of $\T$ with $c \in \TU$.
\begin{enumerate}
\item For every $u \in \L^1_{\TU} ([c,d)_{\TU},\R^m)$, we have $u^\Phi \in \L^1_{\T}( [c,d)_{\T},\R^m)$ and
\begin{equation}\label{inequalitynotequality}
\Vert u^\Phi \Vert_{\L^1_{\T}( [c,d)_{\T},\R^m)} \leq \Vert u \Vert_{\L^1_{\TU} ([c,d)_{\TU},\R^m)}.
\end{equation}
\item For every $u \in \L^\infty_{\TU} ([c,d)_{\TU},\R^m)$, we have $u^\Phi \in \L^\infty_{\T}( [c,d)_{\T},\R^m)$ and
$$ \Vert u^\Phi \Vert_{\L^\infty_{\T}( [c,d)_{\T},\R^m)} = \Vert u \Vert_{\L^\infty_{\TU} ([c,d)_{\TU},\R^m)}. $$
\end{enumerate}
\end{proposition}

\begin{proof}
Let $u \in \L^1_{\TU} ([c,d)_{\TU},\R^m)$. We first treat the $\mu_{\DD}$-measurability of $u^\Phi$. From Lemma~\ref{lemnumber1}, we can consider that $u$ is defined everywhere on $[c,d)_{\TU}$ and is $\mu_{\DD_1}$-measurable on $[c,d)_{\TU}$. Let us prove that $u^\Phi$ is $\mu_{\DD}$-measurable on $[c,d)_{\T}$. We introduce $d' = \inf \{ s \in \TU \,| \,s \geq d \} \in \TU$. Note that $d' \geq d$, with equality if and only if $d \in \TU$. From~\cite{caba}, $u$ is $\mu_{\DD_1}$-measurable on $[c,d)_{\TU}$ if and only if the extension $\tilde{u}$ defined on $[c,d')$ by
$$ \tilde{u}(t) = \left\lbrace \begin{array}{lcl}
u(t) & \text{if} & t \in \TU, \\
u(r) & \text{if} & t \in (r,\sigma_1(r)), \,\text{where} \,r \in \RS_1,
\end{array} \right. $$
is $\mu_L$-measurable on $[c,d')$. By hypothesis, $\tilde{u}$ is $\mu_L$-measurable on $[c,d')$ and consequently, the restriction $\tilde{u}_{|_{[c,d)}}$ is $\mu_L$-measurable on $[c,d)$. Still from~\cite{caba}, $u^\Phi$ is $\mu_{\DD}$-measurable on $[c,d)_{\T}$ if and only if the extension $\widetilde{u^\Phi}$ defined on $[c,d)$ by
$$ \widetilde{u^\Phi}(t) = \left\lbrace \begin{array}{lcl}
u^\Phi(t) & \text{if} & t \in \T, \\
u^\Phi(r) & \text{if} & t \in (r,\sigma(r)), \,\text{where} \,r \in \RS,
\end{array} \right. $$
is $\mu_L$-measurable on $[c,d)$. Hence, it is sufficient to see that $\widetilde{u^\Phi} = \tilde{u}_{|_{[c,d)}}$. This is true since, for any $t \in [c,d)$, we have:
\begin{itemize}
\item either $t \in \TU$, and then $t \in \T$ and thus $\widetilde{u^\Phi}(t) = u^\Phi(t) = u(\Phi(t)) = u(t) = \tilde{u}(t)$.
\item either $t \in \T \bs \TU$, and then $t \in (\Phi(t),\sigma_1(\Phi(t)))$ where $\Phi(t) \in \RS_1$ and thus $\widetilde{u^\Phi}(t) = u^\Phi(t) = u(\Phi(t)) = \tilde{u}(t)$.
\item either $t \notin \T$, and then $t \in (r,\sigma(r))$ where $r \in \RS$ and then $t \in (\Phi(r),\sigma_1(\Phi(r)))$ where $\Phi(r) \in \RS_1$ and thus $\widetilde{u^\Phi}(t) = u^\Phi(r) = u(\Phi(r)) = \tilde{u}(t)$.
\end{itemize}
This establishes the $\mu_{\DD}$-measurability of $u^\Phi$ on $[c,d)_{\T}$.

Considering $\Vert u \Vert_{\R^m}$ instead of $u$, we can consider that $m = 1$ and $u \in \L^1_{\TU}([c,d)_{\TU},\R^+)$. From~\cite{caba}, we have
\begin{multline*}
\di \int_{[c,d)_\T} u^\Phi (\t) \,\DD \t  =  \int_{[c,d)} \widetilde{u^\Phi}(\t) \,d\t
 =  \int_{[c,d)} \tilde{u}_{|_{[c,d)}}(\t) \,d\t  \\
 =  \int_{[c,d')} \tilde{u}(\t) \,d\t - \int_{[d,d')} \tilde{u}(\t) \,d\t
 =  \int_{[c,d)_{\TU}} u(\t) \,\DD_1 \t - (d'-d)u(\Phi(d)).
\end{multline*}
Noting that $(d'-d)u(\Phi(d)) \geq 0$ concludes the proof of the first point.

Let us prove the second point. Let $u \in \L^\infty_{\TU} ([c,d)_{\TU},\R^m)$. The $\mu_{\DD}$-measurability of $u^\Phi$ is already proved since $\L^\infty_{\TU} ([c,d)_{\TU},\R^m) \subset \L^1_{\TU} ([c,d)_{\TU},\R^m)$. Let $M \geq 0$ be a constant. With the same arguments as in the proof of Lemma~\ref{lemnumber1}, we can prove that $ \Vert u(t) \Vert_{\R^m} \leq M$ for $\DD_1$-a.e. $t \in [c,d)_{\TU}$ if and only if $ \Vert u^\Phi(t) \Vert_{\R^m} \leq M$ for $\DD$-a.e. $t \in [c,d)_{\T}$. As a consequence, we get $u^\Phi \in \L^\infty_{\T}( [c,d)_{\T},\R^m)$ and $ \Vert u^\Phi \Vert_{\L^\infty_{\T}( [c,d)_{\T},\R^m)} = \Vert u \Vert_{\L^\infty_{\TU} ([c,d)_{\TU},\R^m)}$.
\end{proof}

\begin{remark}
From the proof above, we see that the inequality~\eqref{inequalitynotequality} is an equality if and only if $d \in \TU$ or $u(\Phi(d)) = 0$. Then, considering $\T = \N$, $\TU = 2\N$, $m=1$, $c=0$, $d=1$ and $u$ the constant function equal to $1$ provides a counterexample. Indeed, in that case, we have $u \in \L^1_{\TU} ([c,d)_{\TU},\R^m)$ and $u^\Phi \in \L^1_{\T}( [c,d)_{\T},\R^m)$ with
$ \Vert u^\Phi \Vert_{\L^1_{\T}( [c,d)_{\T},\R^m)} = 1 < 2 = \Vert u \Vert_{\L^1_{\TU} ([c,d)_{\TU},\R^m)}.  $
\end{remark}

\subsubsection{Recalls on $\DD$-Cauchy-Lipschitz results}\label{recallCL}
According to \cite[Theorem 1]{bour10}, for every control $u \in \L^\infty_{\TU} (\TU,\R^m)$ and every initial condition $q_a \in \R^n$, there exists a unique maximal solution of~\eqref{DD-CS} such that $q(a)=q_a$, denoted by $q(\cdot,u,q_a)$, and defined on a maximal interval, denoted by $I_\T(u,q_a)$. The word \textit{maximal} means that $q(\cdot,u,q_a)$ is an extension of any other solution.
Moreover, we recall that (see \cite[Lemma~1]{bour10})
$$ \forall t \in I_\T (u,q_a), \quad q(t,u,q_a) = q_a + \di \int_{[a,t)_\T} f(\t,q(\tau,u,q_a),u^\Phi(\tau)) \,\DD \tau .$$
Finally, either $I_\T (u,q_a) = \T$, that is, $q(\cdot,u,q_a)$ is a \textit{global} solution of~\eqref{DD-CS}, or $I_\T (u,q_a) = [a,c)_\T$ where $c$ is a left-dense point of $\T$, and in that case, $q(\cdot,u,q_a)$ is unbounded on $I_\T (u,q_a)$ (see \cite[Theorem~2]{bour10}).

\begin{definition}\label{defadm}
For a given $b \in \T$, a couple $(u,q_a) \in \L^\infty_{\TU} (\TU,\R^m) \times \R^n$ is said to be \textit{admissible} on $[a,b]_\T$ whenever $b \in I_\T (u,q_a)$.
\end{definition}

For a given $b \in \T$, we denote by $\UU$ the set of all admissible couples $(u,q_a)$ on $[a,b]_\T$. It is endowed with the norm
$ \Vert (u,q_a) \Vert_{\UU} = \Vert u \Vert_{\L^1_{\TU} ([a,b)_{\TU},\R^{m})} + \Vert q_a \Vert_{\R^n} . $

\subsection{Needle-like variations of the control, and variation of the initial condition}\label{sectionvariations}

Throughout this section, we consider $b \in \T$ and $(u,q_a) \in \UU$.
We are going, in particular, to define appropriate needle-like variations. As in \cite{bour11}, we have to distinguish between right-dense and right-scattered times, along the time scale $\TU$. We will also define appropriate variation of the initial condition.

In the sequel, the notation $\Vert \cdot \Vert$ stands for the usual induced norm of matrices in $\R^{n,n}$, $\R^{n,m}$ and $\R^{m,m}$.

\subsubsection{General variation of $(u,q_a)$}\label{section31prel}
In the first lemma below, we prove that $\UU$ is open. Actually we prove a stronger result, by showing that $\UU$ contains a neighborhood of any of its point in $\L^1_{\TU} \times \R^n$ topology, which will be useful in order to define needle-like variations.

\begin{lemma}\label{prop30-1}
Let $R > \Vert u \Vert_{\L^\infty_{\TU}([a,b)_{\TU},\R^m)}$.
There exist $\nu_R > 0$ and $\eta_R > 0$ such that the set
\begin{equation*}
\E_R(u,q_a) = \Big( \B_{\L^\infty_{\TU}([a,b)_{\TU},\R^m)} (0,R) \cap \B_{\L^1_{\TU}([a,b)_{\TU},\R^m)} (u,\nu_R) \Big) \times \B_{\R^n} (q_a,\eta_R)
\end{equation*}
is contained in $\UU$.
\end{lemma}

Before proving this lemma, let us recall a time scale version of Gronwall's Lemma (see \cite[Chapter 6.1]{bohn}).
The generalized exponential function is defined by
$e_L (t,c) = \exp ( \int_{[c,t)_\T} \xi_{\mu (\tau)} (L) \ \DD \tau )$,
for every $L \geq 0$, every $c \in \T$ and every $t \in [c,+\infty)_\T$,
where
$\xi_{\mu (\tau)} (L) = \log ( 1+ L \mu (\tau)) / \mu (\tau)$ whenever $\mu (\tau) > 0$, and $\xi_{\mu (\tau)} (L) = L$ whenever $\mu (\tau) =0$
(see \cite[Chapter 2.2]{bohn}). Note that, for every $L \geq 0$ and every $c \in \T$, the function $e_L (\cdot,c)$ is positive and increasing on $[c,+\infty)_\T$.

\begin{lemma}[\cite{bohn}]\label{lemgronwall}
Let $c < d$ be two elements of $\T$, let $L_1$ and $L_2$ be two nonnegative real numbers, and let $q \in \CC ([c,d]_\T, \R )$ satisfying
$0 \leq q(t) \leq L_1 + L_2 \int_{[c,t)_\T} q(\tau) \;\DD \tau,$
for every $t \in [c,d]_\T$. Then
$0 \leq q(t) \leq L_1 e_{L_2} (t,c),$
for every $t \in [c,d]_\T$.
\end{lemma}

\begin{proof}[Proof of Lemma~\ref{prop30-1}]
By continuity of $q(\cdot,u,q_a)$ on $[a,b]_\T$, the set
$$
K_R = \{ (t,x,v) \in [a,b]_\T\times\R^{n} \times \B_{\R^m}(0,R) \mid \Vert x-q(t,u,q_a) \Vert_{\R^n} \leq 1 \}
$$
is a compact subset of $\T \times \R^n \times \R^m $. Therefore $ \Vert \partial f / \partial q \Vert $ and $ \Vert \partial f / \partial u \Vert $ are bounded on $K_R$ by some $L_R \geq 0$ and then, from convexity and from the mean value inequality, it holds:
\begin{equation}\label{eqlipK}
\Vert f(t,x_2,v_2) - f(t,x_1,v_1) \Vert_{\R^n}  \leq L_R ( \Vert x_2 -x_1 \Vert_{\R^n}  + \Vert v_2 -v_1 \Vert_{\R^m}  ),
\end{equation}
for all $(t,x_1,v_1),(t,x_2,v_2) \in K_R$.
Let $\nu_R >0$ and $0 < \eta_R <1$ such that $ (\eta_R + \nu_R L_R) e_{L_R} (b,a) < 1$.

Let $(u',q'_a) \in \E_R(u,q_a)$. Our aim is to prove that $b \in I_\T(u',q'_a)$.
By contradiction, assume that the set
$$ A = \{ t \in I_\T(u',q'_a) \cap [a,b]_\T\ \vert \ \Vert q(t,u',q'_a) - q(t,u,q_a) \Vert_{\R^n}  > 1 \}$$
is not empty and let $t_0 = \inf A$. Since $\T$ is closed, $t_0 \in I_\T(u',q'_a) \cap [a,b]_\T$ and $[a,t_0]_\T \subset I_\T(u',q'_a) \cap [a,b]_\T$. If $t_0$ is a minimum then $ \Vert q(t_0,u',q'_a) - q(t_0,u,q_a) \Vert_{\R^n} > 1$. If $t_0$ is not a minimum then $t_0 \in \RD$ and by continuity we have $ \Vert q(t_0,u',q'_a) - q(t_0,u,q_a) \Vert_{\R^n}  \geq 1$. Moreover one has $t_0 > a$ since $\Vert q(a,u',q'_a) - q(a,u,q_a) \Vert_{\R^n}  = \Vert q'_a - q_a \Vert_{\R^n}  \leq \eta_R < 1$. Hence $\Vert q(\tau,u',q'_a) - q(\tau,u,q_a) \Vert_{\R^n}  \leq 1$ for every $\tau \in [a,t_0)_\T$. Therefore $(\tau,q(\tau,u',q'_a),u'^{\Phi}(\tau))$ and $(\tau,q(\tau,u,q_a),u^{\Phi}(\tau))$ are elements of $K_R$ for $\DD$-a.e. $\tau \in [a,t_0)_\T$.
Since one has
$$
q(t,u',q'_a) - q(t,u,q_a) = q'_a - q_a  + \int_{[a,t)_\T} \left( f(\tau,q(\tau,u',q'_a),u'^{\Phi}(\tau)) - f(\tau,q(\tau,u,q_a),u^{\Phi}(\tau)) \right)  \DD \tau ,
$$
for every $t \in I_\T(u',q'_a) \cap [a,b]_\T$, it follows from~\eqref{eqlipK} that, for every $t \in [a,t_0]_\T$,
\begin{multline*}
\Vert q(t,u',q'_a) - q(t,u,q_a) \Vert_{\R^n} \leq  \Vert q'_a - q_a \Vert_{\R^n}  + L_R  \int_{[a,t)_\T} \Vert  u'^{\Phi}(\tau) - u^{\Phi}(\tau) \Vert_{\R^m}  \, \DD \tau \\
 + L_R  \int_{[a,t)_\T} \Vert   q(\tau,u',q'_a) - q(\tau,u,q_a) \Vert_{\R^n}  \, \DD \tau,  
\end{multline*}
which implies from Lemma~\ref{lemgronwall} that, for every $t \in [a,t_0]_\T$,
\begin{equation*}
\Vert q(t,u',q'_a) - q(t,u,q_a) \Vert_{\R^n} \leq  (\Vert q'_a - q_a \Vert_{\R^n}  + L_R \Vert u'^{\Phi}- u^{\Phi} \Vert_{\L^1_\T([a,b)_\T,\R^m)}) e_{L_R} (b,a),
\end{equation*}
which finally implies from Proposition~\ref{propnumber1} that, for every $t \in [a,t_0]_\T$,
\begin{equation*}
\begin{split}
\Vert q(t,u',q'_a) - q(t,u,q_a) \Vert_{\R^n} & \leq  (\Vert q'_a - q_a \Vert_{\R^n}  + L_R \Vert u'- u \Vert_{\L^1_{\TU} ([a,b)_{\TU},\R^m)}) e_{L_R} (b,a) \\
& \leq  (\eta_R + \nu_R L_R) e_{L_R} (b,a)  <  1.
\end{split}
\end{equation*}
This raises a contradiction at $t=t_0$. Therefore $A$ is empty and thus $q(\cdot,u',q'_a)$ is bounded on $I_\T(u',q'_a) \cap [a,b]_\T$. It follows from \cite[Theorem 2]{bour10} that $b \in I_\T(u',q'_a)$, that is, $(u',q'_a) \in \UU$.
\end{proof}

\begin{remark}\label{rmK}
Let $R > \Vert u \Vert_{\L^\infty_{\TU}([a,b)_{\TU},\R^m)}$ and $(u',q'_a) \in \E_R(u,q_a)$. With the notations of the above proof, since $[a,b]_\T \subset I_\T(u',q'_a)$ and $A$ is empty, we infer that $\Vert q(t,u',q'_a) - q(t,u,q_a) \Vert_{\R^n} \leq 1$, for every $t \in [a,b]_\T$.
Therefore $(\t,q(\t,u',q'_a),u'^{\Phi}(\t)) \in K_R$ for every $(u',q'_a) \in \E_R(u,q_a)$ and for $\DD$-a.e. $\t \in [a,b)_\T$.
\end{remark}

\begin{lemma}\label{prop30-1-1}
Let $R > \Vert u \Vert_{\L^\infty_{\TU}([a,b)_{\TU},\R^m)}$. The mapping
\begin{equation*}
\fonction{F_R(u,q_a)}{(\E_R(u,q_a),\Vert \cdot \Vert_{\UU})}{(\CC([a,b]_\T,\R^{n}),\Vert \cdot \Vert_\infty)}{(u',q'_a)}{q(\cdot,u',q'_a)}
\end{equation*}
 is Lipschitzian.
In particular, for every $(u',q'_a) \in \E_R(u,q_a)$, $q(\cdot,u',q'_a)$ converges uniformly to $q(\cdot,u,q_a)$ on $[a,b]_\T$ when $u'$ tends to $u$ in $\L^1_{\TU} ([a,b)_{\TU},\R^m)$ and $q'_a $ tends to $q_a$ in $\R^n$.
\end{lemma}

\begin{proof}
Let $(u',q'_a)$ and $(u'',q''_a)$ be two elements of $\E_R(u,q_a) \subset \UU$. It follows from Remark~\ref{rmK} that $(\tau,q(\tau,u'',q''_a),u''^{\Phi}(\tau))$ and $(\tau,q(\tau,u',q'_a),u'^{\Phi}(\tau))$ are elements of $K_R$ for $\DD$-a.e. $\tau \in [a,b)_\T$. Following the same arguments as in the previous proof, it follows from~\eqref{eqlipK}, from Lemma~\ref{lemgronwall} and from Proposition~\ref{propnumber1} that, for every $t \in [a,b]_\T$,
\begin{equation*}
\Vert q(t,u'',q''_a) - q(t,u',q'_a) \Vert_{\R^n}  \leq  (\Vert q''_a - q'_a \Vert_{\R^n}  + L_R \Vert u'' - u' \Vert_{\L^1_{\TU}([a,b)_{\TU},\R^m)}) e_{L_R} (b,a).
\end{equation*}
The lemma follows.
\end{proof}

\subsubsection{Needle-like variation of $u$ at a point $r \in \RS_1$}\label{section31RS}
Let $r \in [a,b)_\TU \cap \RS_1$ and let $y \in \R^m$. We define the needle-like variation $\Pi = (r,y)$ of $u$ at $r$ by
\begin{equation*}
u_\Pi (t,\alpha) = \left\{ \begin{array}{lcl}
u(r) + \alpha (y-u(r)) & \textrm{if} & t=r, \\
u(t) & \textrm{if} & t \neq r,
\end{array} \right.
\end{equation*}
for $\Delta_1$-a.e. $t \in [a,b)_{\TU}$ and for every $\alpha \in [0,1]$. In the sequel, let us denote by $\sigma^*_1(r) = \min(\sigma_1(r),b)$.

\begin{lemma}\label{lem32-1}
There exists $0 < \alpha_0 \leq 1$ such that $(u_\Pi (\cdot,\alpha),q_a) \in \UU$ for every $\alpha \in [0,\alpha_0]$.
\end{lemma}

\begin{proof}
Let $R = \max ( \Vert u \Vert_{\L^\infty_{\TU} ([a,b)_{\TU},\R^m)},\Vert u(r) \Vert_{\R^m} + \Vert y \Vert_{\R^m} ) +1 > \Vert u \Vert_{\L^\infty_{\TU} ([a,b)_{\TU},\R^m)}$. We use the notations $K_R$, $L_R$, $\nu_R$ and $\eta_R$, defined in Lemma~\ref{prop30-1} and in its proof. One has $\Vert u_\Pi (\cdot,\alpha) \Vert_{\L^\infty_{\TU} ([a,b)_{\TU},\R^m)} \leq R$ for every $\alpha \in [0,1]$, and
$$
\Vert u_\Pi (\cdot,\alpha) - u \Vert_{\L^1_{\TU} ([a,b)_{\TU},\R^m)} = \mu_1 (r) \Vert u_\Pi (r,\alpha) - u(r) \Vert_{\R^m}  = \alpha \mu_1 (r) \Vert y - u(r) \Vert_{\R^m} .
$$
Hence, there exists $0 < \alpha_0 \leq 1$ such that $\Vert u_\Pi (\cdot,\alpha) - u \Vert_{\L^1_{\TU} ([a,b)_{\TU},\R^m)} \leq \nu_R$ for every $\alpha \in [0,\alpha_0]$, and hence $(u_\Pi (\cdot,\alpha),q_a) \in \E_R(u,q_a)$. The claim follows then from Lemma~\ref{prop30-1}.
\end{proof}

\begin{lemma}\label{lem32-1-1}
The mapping
\begin{equation*}
\fonction{F_\Pi(u,q_a)}{([0,\alpha_0],\vert \cdot \vert)}{(\CC([a,b]_\T,\R^{n}),\Vert \cdot \Vert_\infty)}{\alpha}{q (\cdot,u_\Pi (\cdot,\alpha),q_a)}
\end{equation*}
is Lipschitzian.
In particular, for every $\alpha \in [0,\alpha_0]$, $q (\cdot,u_\Pi (\cdot,\alpha),q_a)$ converges uniformly to $q (\cdot,u,q_a)$ on $[a,b]_\T$ as $\alpha$ tends to $0$.
\end{lemma}

\begin{proof}
We use the notations of the proof of Lemma~\ref{lem32-1}. It follows from Lemma~\ref{prop30-1-1} that there exists $C \geq 0$ (the Lipschitz constant of $F_R(u,q_a)$) such that
\begin{equation*}
\begin{split}
\Vert q (\cdot,u_\Pi (\cdot,\alpha_2),q_a) - q (\cdot,u_\Pi (\cdot,\alpha_1),q_a) \Vert_\infty
& \leq C \Vert (u_\Pi (\cdot,\alpha_2),q_a)- (u_\Pi (\cdot,\alpha_1),q_a) \Vert_{\UU}  \\
& = C \vert \alpha_2 - \alpha_1 \vert \mu_1 (r) \Vert y - u(r) \Vert_{\R^m} ,
\end{split}
\end{equation*}
for all $\alpha_1$ and $\alpha_2$ in $ [0,\alpha_0]$.
The lemma follows.
\end{proof}

We define the so-called \textit{first variation vector} $h_\Pi (\cdot,u,q_a)$ associated with the needle-like variation $\Pi = (r,y)$ as the unique solution on $[r,\sigma^*_1(r)]_\T$ of the linear $\DD$-Cauchy problem
\begin{equation}\label{varvect_scatteredh}
h^\DD(t) = \frac{\partial f}{\partial q} (t,q (t,u,q_a),u(r)) \, h(t) + \frac{\partial f}{\partial u} (t,q (t,u,q_a),u(r)) \, (y-u(r)), \quad
h(r) = 0.
\end{equation}
The existence and uniqueness of $h_\Pi (\cdot,u,q_a)$ are ensured by \cite[Theorem 3]{bour10}.

\begin{proposition}\label{prop32-10}
The mapping
\begin{equation*}
\fonction{F_\Pi(u,q_a)}{ ([0,\alpha_0],\vert \cdot \vert)}{(\CC ([r,\sigma^*_1(r)]_\T,\R^{n}),\Vert \cdot \Vert_\infty)}{\alpha}{q (\cdot,u_\Pi (\cdot,\alpha),q_a)}
\end{equation*}
is differentiable at $0$, and one has $DF_\Pi(u,q_a)(0) = h_\Pi (\cdot,u,q_a)$.
\end{proposition}

\begin{proof}
We use the notations of the proof of Lemma~\ref{lem32-1}. Recall that $(\t,q(\t,u_\Pi(\cdot,\alpha),q_a),u^\Phi_\Pi (\t,\alpha)) \in K_R$ for $\DD$-a.e. $\t \in [a,b)_\T$ and for every $\alpha \in [0,\alpha_0] $ (see Remark~\ref{rmK}). For every $\alpha \in (0,\alpha_0]$ and every $t \in [r,\sigma^*_1(r)]_\T$, we define
$$
\varepsilon_\Pi (t,\alpha) = \frac{q(t,u_\Pi(\cdot,\alpha),q_a) - q (t,u,q_a)}{\alpha }- h_{\Pi} (t,u,q_a).
$$
It suffices to prove that $\varepsilon_\Pi (\cdot,\alpha)$ converges uniformly to $0$ on $[r,\sigma^*_1(r)]_\T$ as $\alpha$ tends to $0$.
For every $\alpha \in (0,\alpha_0]$, since the function $\varepsilon_\Pi (\cdot,\alpha)$ vanishes at $t=r$ and is absolutely continuous on $[r,\sigma^*_1(r)]_\T$, $\varepsilon_\Pi (t,\alpha) = \int_{[r,t)_\T} \varepsilon^\DD_\Pi (\tau,\alpha) \,\DD \tau$
for every $t \in [r,\sigma^*_1(r)]_\T$, where
\begin{equation*}
\begin{split}
\varepsilon_\Pi^\DD (\t,\alpha) = & \frac{f(\t,q(\t,u_\Pi(\cdot,\alpha),q_a),u(r)+\alpha (y-u(r)))-f(\t,q (\t,u,q_a),u(r))}{\alpha} \\ & - \frac{\partial f}{\partial q} (\t,q(\t,u,q_a),u(r)) \, h_{\Pi} (\t,u,q_a) - \frac{\partial f}{\partial u} (\t,q(\t,u,q_a),u(r)) \, (y-u(r)) ,
\end{split}
\end{equation*}
for $\DD$-a.e. $\t \in [r,\sigma^*_1(r))_\T$.
Using the Taylor formula with integral remainder, we get
\begin{equation*}\begin{split}
\varepsilon_\Pi^\DD (\t,\alpha) = & \di \int_0^1 \dfrac{\partial f}{\partial q} (\star_{\theta \alpha \tau}) \,d\theta \ \varepsilon_\Pi (\t,\alpha)\\ &  + \left( \di \int_0^1 \dfrac{\partial f}{\partial q} (\star_{\theta \alpha \tau}) \,d\theta - \frac{\partial f}{\partial q} (\t,q(\t,u,q_a),u(r)) \right)  h_{\Pi} (\t,u,q_a) \\ & + \left( \di \int_0^1 \dfrac{\partial f}{\partial u} (\star_{\theta \alpha \tau}) \,d\theta - \frac{\partial f}{\partial u} (\t,q(\t,u,q_a),u(r)) \right)  (y-u(r)),
\end{split}\end{equation*}
where
$$ \star_{\theta \alpha \tau} = ( \t, q (\t,u,q_a) + \theta (q(\t,u_\Pi(\cdot,\alpha),q_a) - q (\t,u,q_a)),u(r)+\theta \alpha (y-u(r)) ) \in K_R.$$
It follows that
$\Vert \varepsilon_\Pi^\DD (\t,\alpha) \Vert_{\R^n} \leq \chi_\Pi (\t,\alpha) +  L_R \Vert \varepsilon_\Pi (\t,\alpha) \Vert_{\R^n}$, where
\begin{equation*}\begin{split}
\chi_\Pi (\t,\alpha) = & \left\Vert \left( \di \int_0^1 \dfrac{\partial f}{\partial q} (\star_{\theta \alpha \tau}) \,d\theta - \frac{\partial f}{\partial q} (\t,q(\t,u,q_a),u(r)) \right) \, h_{\Pi} (\t,u,q_a) \right\Vert_{\R^n} \\ & + \left\Vert \left( \di \int_0^1 \dfrac{\partial f}{\partial u} (\star_{\theta \alpha \tau}) \,d\theta - \frac{\partial f}{\partial u} (\t,q(\t,u,q_a),u(r)) \right) \, (y-u(r)) \right\Vert_{\R^n}.
\end{split}\end{equation*}
Therefore, one has
\begin{equation*}
\Vert \varepsilon_\Pi (t,\alpha) \Vert_{\R^n}  \leq \int_{[r,\sigma^*_1(r))_\T} \chi_\Pi (\tau,\alpha) \,\DD \tau  +L_R \int_{[r,t)_\T} \Vert \varepsilon_\Pi (\tau,\alpha) \Vert_{\R^n}  \, \DD \tau,
\end{equation*}
for every $t \in [r,\sigma^*_1(r)]_\T$. From Lemma~\ref{lemgronwall},
$\Vert \varepsilon_\Pi (t,\alpha) \Vert_{\R^n}  \leq \Upsilon_\Pi (\alpha) e_{L_R} (\sigma^*_1(r),r)$,
for every $t \in [r,\sigma^*_1(r)]_\T$, where
$\Upsilon_\Pi (\alpha) = \int_{[r,\sigma^*_1(r))_\T} \chi_\Pi (\tau,\alpha) \,\DD \tau $.

To conclude, it remains to prove that $\Upsilon_\Pi (\alpha)$ converges to $0$ as $\alpha$ tends to $0$. Since $q(\t,u_\Pi(\cdot,\alpha),q_a)$ converges uniformly to $q(\t,u,q_a)$ on $[r,\sigma^*_1(r)]_\T$ as $\alpha$ tends to $0$ (see Lemma~\ref{lem32-1-1}) and since $\partial f / \partial q$ and $\partial f / \partial u$ are uniformly continuous on $K_R$, the conclusion follows.
\end{proof}

Then, we define the so-called \textit{second variation vector} $w_\Pi (\cdot,u,q_a)$ associated with the needle-like variation $\Pi = (r,y)$ as the unique solution on $[\sigma^*_1(r),b]_\T$ of the linear $\DD$-Cauchy problem
\begin{equation}\label{varvect_scatteredw}
w^\DD(t) = \frac{\partial f}{\partial q} (t,q (t,u,q_a),u^\Phi(t)) \, w(t) , \quad
w(\sigma^*_1(r)) = h_\Pi(\sigma^*_1(r),u,q_a).
\end{equation}
The existence and uniqueness of $w_\Pi (\cdot,u,q_a)$ are ensured by \cite[Theorem 3]{bour10}.

\begin{proposition}\label{prop32-1}
The mapping
\begin{equation*}
\fonction{F_\Pi(u,q_a)}{ ([0,\alpha_0],\vert \cdot \vert)}{(\CC ([\sigma^*_1(r),b]_\T,\R^{n}),\Vert \cdot \Vert_\infty)}{\alpha}{q (\cdot,u_\Pi (\cdot,\alpha),q_a)}
\end{equation*}
is differentiable at $0$, and one has $DF_\Pi(u,q_a)(0) = w_\Pi (\cdot,u,q_a)$.
\end{proposition}

\begin{proof}
We use the notations of the proof of Lemma~\ref{lem32-1}. From Proposition~\ref{prop32-10}, the case $\sigma^*_1(r) = b$ is already proved. As a consequence, we only focus here on the case $\sigma^*_1(r) = \sigma_1(r) < b$.

Recall that $(\t,q(\t,u_\Pi(\cdot,\alpha),q_a),u^\Phi_\Pi (\t,\alpha)) \in K_R$ for $\DD$-a.e. $\t \in [a,b)_\T$ and for every $\alpha \in [0,\alpha_0] $ (see Remark~\ref{rmK}). For every $\alpha \in (0,\alpha_0]$ and every $t \in [\sigma_1(r),b]_\T$, we define
$$
\varepsilon_\Pi (t,\alpha) = \frac{q(t,u_\Pi(\cdot,\alpha),q_a) - q (t,u,q_a)}{\alpha }- w_{\Pi} (t,u,q_a).
$$
It suffices to prove that $\varepsilon_\Pi (\cdot,\alpha)$ converges uniformly to $0$ on $[\sigma_1(r),b]_\T$ as $\alpha$ tends to $0$.
For every $\alpha \in (0,\alpha_0]$, since the function $\varepsilon_\Pi (\cdot,\alpha)$ is absolutely continuous on $[\sigma_1(r),b]_\T$, we have
$\varepsilon_\Pi (t,\alpha) = \varepsilon_\Pi (\sigma_1(r),\alpha) + \int_{[\sigma_1(r),t)_\T} \varepsilon^\DD_\Pi (\tau,\alpha) \, \DD \tau$
for every $t \in [\sigma_1(r),b]_\T$, where
\begin{multline*}
\varepsilon_\Pi^\DD (\t,\alpha) =  \frac{f(\t,q(\t,u_\Pi(\cdot,\alpha),q_a),u^\Phi(\t))-f(\t,q (\t,u,q_a),u^\Phi(\t))}{\alpha}  \\
 - \frac{\partial f}{\partial q} (\t,q(\t,u,q_a),u^\Phi(\t)) \, w_{\Pi} (\t,u,q_a) ,
\end{multline*}
for $\DD$-a.e. $\t \in [\sigma_1(r),b)_\T$.
Using the Taylor formula with integral remainder, we get
\begin{multline*}
\varepsilon_\Pi^\DD (\t,\alpha) =  \di \int_0^1 \dfrac{\partial f}{\partial q} (\star_{\theta \alpha \tau}) \,d\theta \, \varepsilon_\Pi (\t,\alpha)\\   + \left( \di \int_0^1 \dfrac{\partial f}{\partial q} (\star_{\theta \alpha \tau}) \,d\theta - \frac{\partial f}{\partial q} (\t,q(\t,u,q_a),u^\Phi(\t)) \right) \, w_{\Pi} (\t,u,q_a),
\end{multline*}
where:
$$ \star_{\theta \alpha \tau} = ( \t, q (\t,u,q_a) + \theta (q(\t,u_\Pi(\cdot,\alpha),q_a) - q (\t,u,q_a)),u^\Phi(\t) ) \in K_R.$$
It follows that
$\Vert \varepsilon_\Pi^\DD (\t,\alpha) \Vert_{\R^n} \leq \chi_\Pi (\t,\alpha) +  L_R \Vert \varepsilon_\Pi (\t,\alpha) \Vert_{\R^n}$, where
\begin{equation*}
\chi_\Pi (\t,\alpha) = \left\Vert \left( \di \int_0^1 \dfrac{\partial f}{\partial q} (\star_{\theta \alpha \tau}) \,d\theta - \frac{\partial f}{\partial q} (\t, q(\t,u,q_a),u^\Phi(\t)) \right) \, w_{\Pi} (\t,u,q_a) \right\Vert_{\R^n}.
\end{equation*}
Therefore, one has
\begin{equation*}
\Vert \varepsilon_\Pi (t,\alpha) \Vert_{\R^n}  \leq \Vert \varepsilon_\Pi (\sigma_1(r),\alpha) \Vert_{\R^n}  +  \int_{[\sigma_1(r),b)_\T} \chi_\Pi (\tau,\alpha) \,\DD \tau  +L_R \int_{[\sigma_1(r),t)_\T} \Vert \varepsilon_\Pi (\tau,\alpha) \Vert_{\R^n}  \, \DD \tau,
\end{equation*}
for every $t \in [\sigma_1(r),b]_\T$. It follows from Lemma~\ref{lemgronwall} that
$\Vert \varepsilon_\Pi (t,\alpha) \Vert_{\R^n}  \leq \Upsilon_\Pi (\alpha) e_{L_R} (b,\sigma_1(r))$,
for every $t \in [\sigma_1(r),b]_\T$, where
$\Upsilon_\Pi (\alpha) = \Vert \varepsilon_\Pi (\sigma_1(r),\alpha) \Vert_{\R^n}  +  \int_{[\sigma_1(r),b)_\T} \chi_\Pi (\tau,\alpha) \, \DD \tau$.

To conclude, it remains to prove that $\Upsilon_\Pi (\alpha)$ converges to $0$ as $\alpha$ tends to $0$.
First, from Proposition~\ref{prop32-10}, it is easy to see that $\Vert \varepsilon_\Pi (\sigma_1(r),\alpha) \Vert_{\R^n} $ converges to $0$ as $\alpha$ tends to $0$.
Second, since $q(\t,u_\Pi(\cdot,\alpha),q_a)$ converges uniformly to $q(\t,u,q_a)$ on $[\sigma_1(r),b]_\T$ as $\alpha$ tends to $0$ (see Lemma~\ref{lem32-1-1}) and since $\partial f / \partial q$ is uniformly continuous on $K_R$, we infer that $\int_{[\sigma_1(r),b)_\T} \chi_\Pi (\tau,\alpha) \, \DD \tau$ converges to $0$ as $\alpha$ tends to $0$. The conclusion follows.
\end{proof}

\begin{lemma}\label{lem32-2h}
Let $R > \Vert u \Vert_{\L^\infty_\TU ([a,b)_\TU,\R^m)}$ and let $(u_k,q_{a,k})_{k \in \N}$ be a sequence of elements of $\E_R(u,q_a)$. If $u_k$ converges to $u$ $\DD_1$-a.e. on $[a,b)_\TU$ and $q_{a,k}$ converges to $q_a$ in $\R^n$ as $k$ tends to $+\infty$, then
$h_{\Pi}(\cdot,u_k,q_{a,k})$ converges uniformly to $h_\Pi(\cdot,u,q_a)$ on $[r,\sigma^*_1(r)]_\T$ as $k$ tends to $+\infty$.
\end{lemma}

\begin{proof}
We use the notations $K_R$, $L_R$, $\nu_R$ and $\eta_R$, defined in Lemma~\ref{prop30-1} and in its proof.

Let us consider the absolutely continuous function defined by $\varepsilon_k (\cdot) = h_{\Pi}(\cdot,u_k,q_{a,k}) - h_{\Pi}(\cdot,u,q_a)$ on $[r,\sigma^*_1(r)]_\T$. Let us prove that $\varepsilon_k$ converges uniformly to $0$ on $[r,\sigma^*_1(r)]_\T$ as $k$ tends to $+\infty$. Since $\varepsilon_k (r) = 0$, one has
\begin{equation*}\begin{split}
\varepsilon_k (t) = & \int_{[r,t)_\T} \frac{\partial f}{\partial q} (\tau,q(\tau,u_k,q_{a,k}),u_k (r)) \, \varepsilon_k (\tau) \,\DD \tau \\
& + \int_{[r,t)_\T} \left(\frac{ \partial f}{\partial q} (\tau,q(\tau,u_k,q_{a,k}),u_k (r)) - \frac{\partial f}{\partial q} (\tau,q(\tau,u,q_a),u (r)) \right) \, h_{\Pi} (\tau,u,q_a) \,\DD \tau \\
& + \int_{[r,t)_\T} \left( \frac{ \partial f}{\partial u} (\tau,q(\tau,u_k,q_{a,k}),u_k (r)) - \frac{\partial f}{\partial u} (\tau,q(\tau,u,q_a),u (r)) \right)  \, (y-u(r)) \,\DD \tau \\
& + \int_{[r,t)_\T} \frac{ \partial f}{\partial u} (\tau,q(\tau,u_k,q_{a,k}),u_k (r)) \, (u(r)-u_k(r)) \,\DD \tau
\end{split}\end{equation*}
for every $t \in [r,\sigma^*_1(r)]_\T$ and every $k \in \N$. Since $(u_k,q_{a,k}) \in \E_R(u,q_a)$ for every $k \in \N$, it follows from Remark~\ref{rmK} that $(\t,q(\t,u_k,q_{a,k}),u^\Phi_k(\t)) \in K_R$ for $\DD$-a.e. $\t \in [a,b)_\T$. Hence it follows from Lemma~\ref{lemgronwall} that
$ \Vert \varepsilon_k (t) \Vert_{\R^n} \leq \Upsilon_k e_{L_R} (b,r),$
for every $t \in [r,\sigma^*_1(r)]_\T$, where $\Upsilon_k$ is given by
\begin{equation*}\begin{split}
\Upsilon_k & = L_R \mu_1(r) \Vert u(r)-u_k(r) \Vert_{\R^m} \\
 & + \int_{[r,\sigma^*_1(r))_\T} \left\Vert \frac{ \partial f}{\partial q} (\tau,q(\tau,u_k,q_{a,k}),u_k (r)) - \frac{\partial f}{\partial q} (\tau,q(\tau,u,q_a),u (r)) \right\Vert \Vert  h_{\Pi} (\tau,u,q_a) \Vert_{\R^n} \,\DD \tau  \\
& + \int_{[r,\sigma^*_1(r))_\T} \left\Vert \frac{ \partial f}{\partial u} (\tau,q(\tau,u_k,q_{a,k}),u_k (r)) - \frac{\partial f}{\partial u} (\tau,q(\tau,u,q_a),u (r)) \right\Vert \Vert y-u(r) \Vert_{\R^m} \,\DD \tau .
\end{split}\end{equation*}
Since $\mu_{\DD_1} (\{ r\} ) = \mu_1 (r) > 0$, $u_k(r)$ converges to $u(r)$ as $k$ tends to $+\infty$. Moreover, from the Lebesgue dominated convergence theorem, $(u_k,q_{a,k})$ converges to $(u,q_a)$ in $(\E_R(u,q_a),\Vert \cdot \Vert_{\UU} )$ and, from Lemma~\ref{prop30-1-1}, $q(\cdot,u_k,q_{a,k})$ converges uniformly to $q(\cdot,u,q_a)$ on $[a,b]_\T$ as $k$ tends to $+\infty$. Since $\partial f / \partial q$ and $\partial f / \partial u$ are uniformly continuous on $K_R$, we conclude that $\Upsilon_k$ converges to $0$ as $k$ tends to $+\infty$. The lemma follows.
\end{proof}

\begin{lemma}\label{lem32-2w}
Let $R > \Vert u \Vert_{\L^\infty_\TU ([a,b)_\TU,\R^m)}$ and let $(u_k,q_{a,k})_{k \in \N}$ be a sequence of elements of $\E_R(u,q_a)$. If $u_k$ converges to $u$ $\DD_1$-a.e. on $[a,b)_\TU$ and $q_{a,k}$ converges to $q_a$ in $\R^n$ as $k$ tends to $+\infty$, then
$w_{\Pi}(\cdot,u_k,q_{a,k})$ converges uniformly to $w_\Pi(\cdot,u,q_a)$ on $[\sigma^*_1(r),b]_\T$ as $k$ tends to $+\infty$.
\end{lemma}

\begin{remark}
With the same arguments of the proof of Lemma~\ref{lemnumber1}, the hypotheses of Lemma~\ref{lem32-2w} give the convergences of $u^\Phi_k$ to $u^\Phi$ $\DD$-a.e. on $[a,b)_\T$ and, from the Lebesgue dominated convergence theorem, of $(u_k,q_{a,k})$ to $(u,q_a)$ in $(\E_R(u,q_a),\Vert \cdot \Vert_{\UU} )$.
\end{remark}

\begin{proof}[Proof of Lemma~\ref{lem32-2w}]
We use the notations $K_R$, $L_R$, $\nu_R$ and $\eta_R$, defined in Lemma~\ref{prop30-1} and in its proof. From Lemma~\ref{lem32-2h}, the case $\sigma^*_1(r) = b$ is already proved. As a consequence, we only focus here on the case $\sigma^*_1(r) = \sigma_1(r) < b$.

Let us consider the absolutely continuous function defined by $\varepsilon_k (\cdot) = w_{\Pi}(\cdot,u_k,q_{a,k}) - w_{\Pi}(\cdot,u,q_a)$ on $[\sigma_1(r),b]_\T$. Let us prove that $\varepsilon_k$ converges uniformly to $0$ on $[\sigma_1(r),b]_\T$ as $k$ tends to $+\infty$. One has
\begin{multline*}
\varepsilon_k (t) =  \varepsilon_k (\sigma_1(r)) + \int_{[\sigma_1(r),t)_\T} \frac{\partial f}{\partial q} (\tau,q(\tau,u_k,q_{a,k}),u^\Phi_k (\tau)) \, \varepsilon_k (\tau) \,\DD \tau \\
 + \int_{[\sigma_1(r),t)_\T} \left(\frac{ \partial f}{\partial q} (\tau,q(\tau,u_k,q_{a,k}),u^\Phi_k (\tau)) - \frac{\partial f}{\partial q} (\tau,q(\tau,u,q_a),u^\Phi(\tau)) \right) \, w_{\Pi} (\tau,u,q_a) \,\DD \tau  ,
\end{multline*}
for every $t \in [\sigma_1(r),b]_\T$ and every $k \in \N$. Since $(u_k,q_{a,k}) \in \E_R(u,q_a)$ for every $k \in \N$, it follows from Remark~\ref{rmK} that $(\t,q(\t,u_k,q_{a,k}),u^\Phi_k(\t)) \in K_R$ for $\DD$-a.e. $\t \in [\sigma_1(r),b)_\T$. Hence it follows from Lemma~\ref{lemgronwall} that
$$ \Vert \varepsilon_k (t) \Vert_{\R^n} \leq (\varepsilon_k (\sigma_1(r)) + \Upsilon_k ) e_{L_R} (b,\sigma_1(r)),$$
for every $t \in [\sigma_1(r),b]_\T$, where $\Upsilon_k$ is given by
\begin{equation*}
\Upsilon_k = \int_{[\sigma_1(r),b)_\T} \left\Vert \frac{ \partial f}{\partial q} (\tau,q(\tau,u_k,q_{a,k}),u^\Phi_k (\t)) - \frac{\partial f}{\partial q} (\tau,q(\tau,u,q_a),u^\Phi(\t)) \right\Vert \Vert  w_{\Pi} (\tau,u,q_a) \Vert_{\R^n} \,\DD \tau.
\end{equation*}
Since $(u_k,q_{a,k})$ converges to $(u,q_a)$ in $(\E_R(u,q_a),\Vert \cdot \Vert_{\UU})$ and from Lemma~\ref{prop30-1-1}, $q(\cdot,u_k,q_{a,k})$ converges uniformly to $q(\cdot,u,q_a)$ on $[a,b]_\T$ as $k$ tends to $+\infty$. Since $\partial f / \partial q$ is continuous and bounded on $K_R$ and since $u^\Phi_k$ converges to $u^\Phi$ $\DD$-a.e. on $[a,b)_\T$, the Lebesgue dominated convergence theorem concludes that $\Upsilon_k$ converges to $0$ as $k$ tends to $+\infty$. Finally, from Lemma~\ref{lem32-2h}, one has $\varepsilon_k (\sigma_1(r)) = w_{\Pi}(\sigma_1(r),u_k,q_{a,k}) - w_{\Pi}(\sigma_1(r),u,q_a) = h_{\Pi}(\sigma_1(r),u_k,q_{a,k}) - h_{\Pi}(\sigma_1(r),u,q_a) $
converges to $0$ as $k$ tends to $+\infty$. The lemma follows.
\end{proof}
%

\subsubsection{Needle-like variation of $u$ at a point $s \in \RD_1$}\label{section31RD}
Let $s \in \LL_{[a,b)_\T} (f(\cdot,q(\cdot,u,q_a),u^\Phi)) \cap \RD_1$ and $z \in \R^m$. Note that $s \in \TU$ and then $\Phi(s)=s$. We define the needle-like variation $\amalg = (s,z)$ of $u$ at $s$ by
\begin{equation*}
u_\amalg (t,\beta) = \left\{ \begin{array}{lcl}
z & \textrm{if} & t \in [s,s+\beta)_{\TU} , \\
u(t) & \textrm{if} & t \notin [s,s+\beta)_{\TU}.
\end{array} \right.
\end{equation*}
for $\Delta_1$-a.e. $t \in [a,b)_{\TU}$ and for every $\beta \in \V^{s,b}_1 \subset \V^{s,b}$.

\begin{lemma}\label{lem33-1}
There exists $\beta_0 > 0$ such that $(u_\amalg (\cdot,\beta),q_a) \in \UU$ for every $\beta \in \V^{s,b}_1 \cap [0,\beta_0]$.
\end{lemma}

\begin{proof}
Let $R = \max ( \Vert u \Vert_{\L^\infty_{\TU} ([a,b)_{\TU},\R^m)},\Vert z \Vert_{\R^m}) +1 > \Vert u \Vert_{\L^\infty_{\TU} ([a,b)_{\TU},\R^m)}$. We use the notations $K_R$, $L_R$, $\nu_R$ and $\eta_R$, defined in the proof of Lemma~\ref{prop30-1}.
For every $\beta \in \V^{s,b}_1$, we have $\Vert u_\amalg (\cdot,\beta) \Vert_{\L^\infty_{\TU} ([a,b)_{\TU},\R^m)} \leq R$ and
$$
\Vert u_\amalg (\cdot,\beta) - u \Vert_{\L^1_{\TU} ([a,b)_{\TU},\R^m)} = \int_{[s,s+\beta)_{\TU}} \Vert z - u(\tau) \Vert_{\R^m}  \,\DD_1 \tau \leq 2 R \beta.
$$
Hence, there exists $\beta_0 > 0$ such that for every $\beta \in \V^{s,b}_1 \cap [0,\beta_0]$, $\Vert u_\amalg (\cdot,\beta) - u \Vert_{\L^1_{\TU} ([a,b)_{\TU},\R^m)} \leq \nu_R$ and thus $(u_\amalg (\cdot,\beta),q_a) \in \E_R(u,q_a)$. The conclusion then follows from Lemma~\ref{prop30-1}.
\end{proof}

\begin{lemma}\label{lem33-1-1}
The mapping
\begin{equation*}
\fonction{F_\iP(u,q_a)}{(\V^{s,b}_1 \cap [0,\beta_0],\vert \cdot \vert)}{(\CC([a,b]_\T,\R^{n}),\Vert \cdot \Vert_\infty)}{\beta}{q (\cdot,u_\amalg (\cdot,\beta),q_a)}
\end{equation*}
is Lipschitzian.
In particular, for every $\beta \in \V^{s,b}_1 \cap [0,\beta_0]$, $q (\cdot,u_\amalg (\cdot,\beta),q_a)$ converges uniformly to $q(\cdot,u,q_a)$ on $[a,b]_\T$ as $\beta$ tends to $0$.
\end{lemma}

\begin{proof}
We use the notations of the proof of Lemma~\ref{lem33-1}. From Lemma~\ref{prop30-1-1}, there exists $C \geq 0$ (Lipschitz constant of $F_R(u,q_a)$) such that
\begin{equation*}\begin{split}
\Vert q (\cdot,u_\amalg (\cdot,\beta_2),q_a) - q (\cdot,u_\amalg (\cdot,\beta_1),q_a) \Vert_\infty
& \leq C \Vert (u_\amalg (\cdot,\beta_2),q_a) - (u_\amalg (\cdot,\beta_1),q_a) \Vert_{\UU} \\
& \leq 2 C R \vert \beta_2-\beta_1 \vert,
\end{split}\end{equation*}
for all $\beta_1$ and $\beta_2$ in $\V^{s,b}_1 \cap [0,\beta_0]$.
The lemma follows.
\end{proof}

According to \cite[Theorem 3]{bour10}, we define the \textit{variation vector} $w_{\amalg}(\cdot,u,q_a)$ associated with the needle-like variation $\amalg = (s,z)$ as the unique solution on $[s,b]_\T$ of the linear $\DD$-Cauchy problem
\begin{equation}\label{varvect_densew}
w^\DD(t)  = \frac{\partial f}{\partial q} (t,q(t,u,q_a),u^\Phi(t)) \, w(t) , \quad w(s) = f(s,q(s,u,q_a),z) - f(s,q(s,u,q_a),u(s)).
\end{equation}

\begin{proposition}\label{prop33-1}
For every $\delta \in \V^{s,b} \backslash \{ 0 \}$, the mapping
\begin{equation*}
\fonction{F_\iP(u,q_a)}{(\V^{s,b}_1 \cap [0,\beta_0],\vert \cdot \vert)}{(\CC ([s+\delta,b]_\T,\R^{n}),\Vert \cdot \Vert_\infty)}{\beta}{q (\cdot,u_\amalg (\cdot,\beta),q_a)}
\end{equation*}
is differentiable at $0$, and one has $ DF_\iP(u,q_a)(0) = w_\amalg (\cdot,u,q_a)$.
\end{proposition}

\begin{proof}
We use the notations of the proof of Lemma~\ref{lem33-1}. Recall that $(\t,q (\t,u_\amalg (\cdot,\beta),q_a),u^\Phi_\amalg (\t,\beta))$ and $(\t,q (\t,u_\amalg (\cdot,\beta),q_a),z)$ belong to $K_R$ for every $\beta \in \V^{s,b}_1 \cap [0,\beta_0]$ and for $\DD$-a.e. $\t \in [a,b)_\T$ (see Remark~\ref{rmK}). For every $\beta \in \V^{s,b}_1 \cap (0,\beta_0]$ and every $t \in [s+\beta,b]_\T$, we define
$$ \varepsilon_\amalg (t,\beta) = \frac{q (t,u_\amalg (\cdot,\beta),q_a) - q(t,u,q_a)}{ \beta } - w_{\amalg} (t,u,q_a). $$
It suffices to prove that $\varepsilon_\amalg (\cdot,\beta)$ converges uniformly to $0$ on $[s+\beta,b]_\T$ as $\beta$ tends to $0$. Note that, for every $\delta \in \V^{s,b} \backslash \{ 0 \}$, it suffices to consider $\beta \leq \delta$.
For every $\beta \in \V^{s,b}_1 \cap (0,\beta_0]$, the function $\varepsilon_\amalg (\cdot,\beta)$ is absolutely continuous on $[s+\beta,b]_\T$ and
$\varepsilon_\amalg (t,\beta) = \varepsilon_\amalg (s+\beta,\beta) +  \int_{[s+\beta,t)_\T} \varepsilon^\DD_\amalg (\tau,\beta) \, \DD \tau$,
for every $t \in [s+\beta,b]_\T$,
where
\begin{multline*}
\varepsilon_\amalg^\DD (\t,\beta) =  \frac{f(\t,q (\t,u_\amalg (\cdot,\beta),q_a),u^\Phi(\t))-f(\t,q (\t,u,q_a),u^\Phi(\t))}{\beta} \\
- \frac{\partial f}{\partial q} (\t,q(\t,u,q_a),u^\Phi(\t)) \, w_{\amalg} (\t,u,q_a),
\end{multline*}
for $\DD$-a.e. $\t \in [s+\beta,b)_\T$.
Using the Taylor formula with integral remainder, we get
\begin{multline*}
\varepsilon_\iP^\DD (\t,\beta) =  \di \int_0^1 \dfrac{\partial f}{\partial q} (\star_{\theta \beta \tau}) \,d\theta \, \varepsilon_\iP (\t,\beta)\\
+ \left( \di \int_0^1 \dfrac{\partial f}{\partial q} (\star_{\theta \beta \tau}) \,d\theta - \frac{\partial f}{\partial q} (\t,q(\t,u,q_a),u^\Phi(\t)) \right) \, w_{\iP} (\t,u,q_a),
\end{multline*}
where:
$$ \star_{\theta \beta \tau} = ( \t, q (\t,u,q_a) + \theta (q(\t,u_\iP(\cdot,\beta),q_a) - q (\t,u,q_a)),u^\Phi(\t) ) \in K_R.$$
It follows that
$\Vert \varepsilon_\iP^\DD (\t,\beta) \Vert_{\R^n} \leq  \chi_\iP (\t,\beta) +  L_R \Vert \varepsilon_\iP (\t,\beta) \Vert_{\R^n}$, where
\begin{equation*}\begin{split}
\chi_\iP (\t,\beta) = & \left\Vert \left( \di \int_0^1 \dfrac{\partial f}{\partial q} (\star_{\theta \beta \tau}) \,d\theta - \frac{\partial f}{\partial q} (\t,q(\t,u,q_a),u^\Phi(\t)) \right) \, w_{\iP} (\t,u,q_a) \right\Vert_{\R^n}.
\end{split}\end{equation*}
Therefore, one has
\begin{equation*}
\Vert \varepsilon_\amalg (t,\beta) \Vert_{\R^n}  \leq \Vert \varepsilon_\amalg (s+\beta,\beta) \Vert_{\R^n}  +  \int_{[s+\beta,b)_\T} \chi_\amalg (\tau,\beta) \,\DD \tau  +L_R \int_{[s+\beta,t)_\T} \Vert \varepsilon_\amalg (\tau,\beta) \Vert_{\R^n}  \, \DD \tau ,
\end{equation*}
for every $t \in [s+\beta,b]_\T$, and it follows from Lemma~\ref{lemgronwall} that
$\Vert \varepsilon_\amalg (t,\beta) \Vert_{\R^n} \leq \Upsilon_\amalg (\beta) e_{L_R} (b,s)$,
for every $t \in [s+\beta,b]_\T$, where
$ \Upsilon_\amalg (\beta) = \Vert \varepsilon_\amalg (s+\beta,\beta) \Vert_{\R^n}  +  \int_{[s+\beta,b)_\T} \chi_\amalg (\tau,\beta) \, \DD \tau$.

To conclude, it remains to prove that $\Upsilon_\amalg (\beta)$ converges to $0$ as $\beta$ tends to $0$. Since $q (\cdot,u_\amalg (\cdot,\beta),q_a)$ converges uniformly to $q(\cdot,u,q_a)$ on $[s+\beta,b]_\T$ as $\beta$ tends to $0$ (see Lemma~\ref{lem33-1-1}) and since $\partial f / \partial q$ is uniformly continuous on $K_R$, we first infer that $\int_{[s+\beta,b)_\T} \chi_\amalg (\tau,\beta) \, \DD \tau$ converges to $0$ as $\beta$ tends to $0$.
Secondly, let us prove that $\Vert \varepsilon_\amalg (s+\beta,\beta) \Vert_{\R^n} $ converges to $0$ as $\beta$ tends to $0$. By continuity, $w_\amalg (s+\beta,u,q_a)$ converges to $w_\amalg (s,u,q_a)$ as $\beta$ to $0$. Moreover, since $q (\cdot,u_\amalg (\cdot,\beta),q_a)$ converges uniformly to $q(\cdot,u,q_a)$ on $[a,b]_\T$ as $\beta$ tends to $0$ and since $f$ is uniformly continuous on $K_R$, it follows that $f(\cdot, q (\cdot,u_\amalg (\cdot,\beta),q_a),z)$ converges uniformly to $f(\cdot, q(\cdot,u,q_a),z)$ on $[a,b]_\T$ as $\beta$ tends to $0$. Therefore, it suffices to note that
$$\frac{1}{ \beta } \int_{[s,s+\beta)_\T} \left( f(\tau,q (\tau,u,q_a),z) - f(\tau,q(\tau,u,q_a),u^\Phi(\tau)) \right) \DD \tau$$
converges to $w_\amalg (s,u,q_a) = f(s,q(s,u,q_a),z) - f(s,q(s,u,q_a),u(s))$ as $\beta$ tends to $0$ since $s$ is a $\DD$-Lebesgue point of $f(\cdot,q(\cdot,u,q_a),z)$ by continuity and of $f(\cdot,q(\cdot,u,q_a),u^\Phi(\cdot))$ by hypothesis.
Then $\Vert \varepsilon_\amalg (s+\beta,\beta) \Vert_{\R^n}$ converges to $0$ as $\beta$ tends to $0$, and hence $\Upsilon_\amalg (\beta)$ converges to $0$ as well.
\end{proof}

\begin{lemma}\label{lem33-2}
Let $R > \Vert u \Vert_{\L^\infty_\TU ([a,b)_\TU,\R^m)}$ and let $(u_k,q_{a,k})_{k \in \N}$ be a sequence of elements of $\E_R(u,q_a)$. If $u_k$ converges to $u$ $\DD_1$-a.e. on $[a,b)_\TU$, $u_k(s)$ converges to $u(s)$ and $q_{a,k}$ converges to $q_a$ as $k$ tends to $+\infty$, then $w_{\amalg}(\cdot,u_k,q_{a,k})$ converges uniformly to $w_\amalg(\cdot,u,q_a)$ on $[s,b]_\T$ as $k$ tends to $+\infty$.
\end{lemma}

\begin{proof}
The proof is similar to the one of Lemma~\ref{lem32-2w}, replacing $\sigma_1(r)$ with $s$.
\end{proof}

\subsubsection{Variation of the initial condition $q_a$}\label{section32}
Let $q'_a \in \R^n$.

\begin{lemma}\label{lem34-1}
There exists $\gamma_0>0$ such that $(u,q_a+\gamma q'_a) \in \UU$ for every $\gamma \in [0,\gamma_0]$.
\end{lemma}

\begin{proof}
Let $R = \Vert u \Vert_{\L^\infty_{\TU} ([a,b)_{\TU},\R^m)}+1 > \Vert u \Vert_{\L^\infty_{\TU} ([a,b)_{\TU},\R^m)}$. We use the notations $K_R$, $L_R$, $\nu_R$ and $\eta_R$, defined in the proof of Lemma~\ref{prop30-1}.
There exists $\gamma_0>0$ such that $\Vert q_a+\gamma q'_a - q_a \Vert_{\R^n}  = \gamma \Vert q'_a \Vert_{\R^n}  \leq \eta_R$ for every $\gamma \in [0,\gamma_0]$, and hence $(u,q_a+\gamma q'_a) \in \E_R(u,q_a)$. Then the claim follows from Lemma~\ref{prop30-1}.
\end{proof}

\begin{lemma}\label{lem34-1-1}
The mapping
\begin{equation*}
\fonction{F_{q'_a}(u,q_a)}{([0,\gamma_0],\vert \cdot \vert)}{(\CC([a,b]_\T,\R^{n}),\Vert \cdot \Vert_\infty)}{\gamma}{q (\cdot,u,q_a+\gamma q'_a)}
\end{equation*}
is Lipschitzian.
In particular, for every $\gamma \in [0,\gamma_0]$, $q (\cdot,u,q_a+\gamma q'_a)$ converges uniformly to $q (\cdot,u,q_a)$ on $[a,b]_\T$ as $\gamma$ tends to $0$.
\end{lemma}

\begin{proof}
We use the notations of the proof of Lemma~\ref{lem34-1}. From Lemma~\ref{prop30-1-1}, there exists $C \geq 0$ (Lipschitz constant of $F_R(u,q_a)$) such that
\begin{equation*}\begin{split}
\Vert q (\cdot,u,q_a + \gamma_2 q'_a ) - q (\cdot,u,q_a + \gamma_1 q'_a ) \Vert_\infty
& \leq  C \Vert (u,q_a+\gamma_2 q'_a)-(u,q_a+\gamma_1 q'_a) \Vert_{\UU}  \\
& = C \vert \gamma_2 - \gamma_1 \vert \Vert q'_a \Vert_{\R^n} .
\end{split}\end{equation*}
for all $\gamma_1$ and $\gamma_2$ in $[0,\gamma_0]$.
\end{proof}

According to \cite[Theorem 3]{bour10}, we define the \textit{variation vector} $w_{q'_a}(\cdot,u,q_a)$ associated with the perturbation $q_a'$ as the unique solution on $[a,b]_\T$ of the linear $\DD$-Cauchy problem
\begin{equation}\label{varvect_pointinit}
w^\DD(t) = \frac{\partial f}{\partial q} (t,q (t,u,q_a),u^\Phi(t)) \, w(t), \quad w(a) = q'_a.
\end{equation}

\begin{proposition}\label{prop34-1}
The mapping
\begin{equation*}
\fonction{F_{q'_a}(u,q_a)}{([0,\gamma_0],\vert \cdot \vert)}{(\CC ([a,b]_\T,\R^{n}),\Vert \cdot \Vert_\infty)}{\gamma}{q (\cdot,u,q_a+\gamma q'_a )}
\end{equation*}
is differentiable at $0$, and one has $DF_{q'_a}(u,q_a)(0) = w_{q'_a} (\cdot,u,q_a)$.
\end{proposition}

\begin{proof}
We use the notations of the proof of Lemma~\ref{lem34-1}. Note that, from Remark~\ref{rmK}, $(\t,q (\t,u,q_a+\gamma q'_a),u^\Phi(\t)) \in K_R$ for every $\gamma \in [0,\gamma_0]$ and for $\DD$-a.e. $\t \in [a,b)_\T$.
For every $\gamma \in (0,\gamma_0]$ and every $t \in [a,b]_\T$, we define
\begin{equation*}
\varepsilon_{q'_a} (t,\gamma) = \frac{q (t,u,q_a+\gamma q'_a) - q(t,u,q_a)}{ \gamma } - w_{q'_a} (t,u,q_a).
\end{equation*}
It suffices to prove that $\varepsilon_{q'_a} (\cdot,\gamma)$ converges uniformly to $0$ on $[a,b]_\T$ as $\gamma$ tends to $0$.
For every $\gamma \in (0,\gamma_0]$, since the function $\varepsilon_{q'_a} (\cdot,\gamma)$ vanishes at $t=a$ and is absolutely continuous on $[a,b]_\T$, $\varepsilon_{q'_a} (t,\gamma) = \int_{[a,t)_\T} \varepsilon^\DD_{q'_a} (\tau,\gamma) \, \DD \tau$, for every $t \in [a,b]_\T$, where
\begin{multline*}
\varepsilon_{q'_a}^\DD (\t,\gamma) = \frac{f(\t,q (\t,u,q_a+\gamma q'_a),u^\Phi(\t))-f(\t,q (\t,u,q_a),u^\Phi(\t))}{\gamma} \\
- \frac{\partial f}{\partial q} (\t,q(\t,u,q_a),u^\Phi(\t),\t) \, w_{q'_a} (\t,u,q_a) ,
\end{multline*}
for $\DD$-a.e. $\t \in [a,b)_\T$.
Using the Taylor formula with integral remainder, we get
\begin{multline*}
\varepsilon_{q'_a}^\DD (\t,\gamma) =  \di \int_0^1 \frac{\partial f}{\partial q} (\star_{\theta \gamma \t}) \,d\theta \cdot \varepsilon_{q'_a} (\t,\gamma)  \\
+ \left( \di \int_0^1 \frac{\partial f}{\partial q} (\star_{\theta \gamma \t}) \,d\theta - \frac{\partial f}{\partial q} (\t,q(\t,u,q_a),u^\Phi(\t)) \right) \, w_{q'_a} (\t,u,q_a),
\end{multline*}
where
$$ \star_{\theta \gamma \t} = ( \t, q(\t,u,q_a) + \theta (q(\t,u,q_a+\gamma q'_a) - q(\t,u,q_a)) , u^\Phi(\t)) \in K_R. $$
It follows that $\Vert \varepsilon_{q'_a}^\DD (\t,\gamma) \Vert_{\R^n}  \leq \chi_{q'_a} (\t,\gamma) +  L_R \Vert \varepsilon_{q'_a} (\t,\gamma) \Vert_{\R^n}$, where
\begin{equation*}
\chi_{q'_a} (\t,\gamma) = \left\Vert \left( \di \int_0^1 \frac{\partial f}{\partial q} (\star_{\theta \gamma \t}) \,d\theta - \frac{\partial f}{\partial q} (\t,q(\t,u,q_a),u^\Phi(\t)) \right) \, w_{q'_a} (\t,u,q_a) \right\Vert_{\R^n} .
\end{equation*}
Hence
\begin{equation*}
\Vert \varepsilon_{q'_a} (t,\gamma) \Vert_{\R^n}  \leq  \int_{[a,b)_\T} \chi_{q'_a} (\tau,\gamma) \, \DD \tau  +L_R \int_{[a,t)_\T} \Vert \varepsilon_{q'_a} (\tau,\gamma) \Vert_{\R^n}  \, \DD \tau,
\end{equation*}
for every $t \in [a,b]_\T$, and it follows from Lemma~\ref{lemgronwall} that
$\Vert \varepsilon_{q'_a} (t,\gamma) \Vert_{\R^n} \leq \Upsilon_{q'_a} (\gamma) e_{L_R} (b,a)$,
for every $t \in [a,b]_\T$, where
$
\Upsilon_{q'_a} (\gamma) = \int_{[a,b)_\T} \chi_{q'_a} (\tau,\gamma) \, \DD \tau.
$

To conclude, it remains to prove that $\Upsilon_{q'_a} (\gamma)$ converges to $0$ as $\gamma$ tends to $0$. Since $q(\cdot,u,q_a+\gamma q'_a)$ converges uniformly to $q(\cdot,u,q_a)$ on $[a,b]_\T$ as $\gamma$ tends to $0$ (see Lemma~\ref{lem34-1-1}) and since $\partial f / \partial q$ is uniformly continuous on $K_R$, the conclusion follows.
\end{proof}

\begin{lemma}\label{lem34-2}
Let $R > \Vert u \Vert_{\L^\infty_\T([a,b)_\T,\R^m)}$ and let $(u_k,q_{a,k})_{k \in \N}$ be a sequence of elements of $\E_R(u,q_a)$. If $u_k$ converges to $u$ $\DD_1$-a.e. on $[a,b)_\TU$ and $q_{a,k}$ converges to $q_a$ in $\R^n$ as $k$ tends to $+\infty$, then $w_{q'_a}(\cdot,u_k,q_{a,k})$ converges uniformly to $w_{q'_a}(\cdot,u,q_a)$ on $[a,b]_\T$ as $k$ tends to $+\infty$.
\end{lemma}

\begin{proof}
The proof is similar to the one of Lemma~\ref{lem32-2w}, replacing $\sigma_1(r)$ with $a$.
\end{proof}

\subsection{Proof of Theorem~\ref{thmmain}}\label{annexe3}
We are now in a position to prove the PMP. In the sequel, we consider $q^*$ an optimal trajectory, associated with an optimal sampled-data control $u^*$ and with $b^* \in \T$, with $b^* = b$ if the final time is fixed. We set $q^*_a = q^*(a)$.

\subsubsection{The augmented system}
As in \cite{lee,pont}, we consider the \textit{augmented system} in $\R^{n+1}$
\begin{equation} \label{DD-CS_augm}
\bar q^\DD(t) = \bar f (t,\bar q(t),u^\Phi(t)),
\end{equation}
with $\bar q=(q,q^0)^\top$ the augmented state with values in $\R^n\times\R$, and $\bar f:\T\times\R^{n+1}\times\R^m \rightarrow\R^{n+1}$, the augmented dynamics, defined by
$\bar f(t,\bar q,u)=(f(t,q,u),f^0(t,q,u))^\top$.
Note that $\bar f$ does not depend on $q^0$. We will always impose as an initial condition $q^0(a)=0$, so that $q^0(b)= \int_{[a,b)_\T} f^0 ( \tau , q (\tau), u^\Phi(\tau) ) \, \DD \tau$. Hence, the additional coordinate $q^0$ stands for the cost.

Denoting $\bar q^*_a = (q^*_a,0)^\top$, we have $(u^*,\bar q_a^*) \in \UUEB$. We set $ \bar q^* = \bar q(\cdot,u^*,\bar q^*_a) = (q^*,q^{0*})^\top$. Hence, $\bar q^*$ is a solution of~\eqref{DD-CS_augm} on $[a,b^*]_\T$ associated with the control $u^*$, satisfying $q^{0*}(a) = 0$ and $ g( q^*(a), q^*(b^*)) \in \S$ and minimizing $ q^0(b) $
over all possible trajectories $\bar q$ solutions of~\eqref{DD-CS_augm} on $[a,b]_\T$, where $b \in \T$ and associated with a control $u\in \L^\infty_{\TU} ([a,b)_{\TU},\Omega)$, satisfying $q^0(a) = 0$ and $ g( q(a), q(b)) \in \S$.

\subsubsection{Application of the Ekeland variational principle}\label{section41}
For the completeness, we recall a simplified (but sufficient) version of the Ekeland variational principle.
\begin{theorem}[\cite{ekel}]
Let $(\E,d_\E)$ be a complete metric space and let $J : \E \rightarrow \R^+$, $\lambda \mapsto J(\lambda)$ be a continuous nonnegative mapping. Let $\varepsilon > 0$ and $\lambda^* \in \E$ such that $J(\lambda^*) \leq \varepsilon$.
Then, there exists $\lambda_\varepsilon \in \E$ such that $d_\E (\lambda_\varepsilon,\lambda^*) \leq \sqrt{\varepsilon}$ and, $-\sqrt{\varepsilon} \, d_\E (\lambda,\lambda_\varepsilon) \leq J(\lambda) - J(\lambda_\varepsilon)$, for every $\lambda \in \E$.
\end{theorem}

Let $R > \Vert u^* \Vert_{\L^\infty_\TU([a,b^*)_\TU,\R^m)}$. Recall that $\E_R(u^*,\bar{q}^*_a)\subset\UUEB$ (see Lemma~\ref{prop30-1}). We set
$$
\E^{\Omega,0}_R (u^*,\bar{q}^*_a) = \{ (u,\bar{q}_a) \in \E_R(u^*,\bar{q}^*_a) \ \vert \ u\in \L^\infty_\TU ([a,b^*)_\TU,\Omega), \, \bar{q}_a=(q_a,0)  \}.
$$
Note that $(u^*,\bar{q}^*_a) \in \E^{\Omega,0}_R (u^*,\bar{q}^*_a)$. Since $\Omega$ is closed, it follows from the (partial) converse of the Lebesgue dominated convergence theorem that $(\E^{\Omega,0}_R (u^*,\bar{q}^*_a),\Vert \cdot \Vert_{\UUEB})$ is a closed subset of $\L^1_\TU([a,b^*)_\TU,\R^m) \times \R^{n+1}$ and then is a complete metric space.
For every $\varepsilon > 0$, we define the functional $J^R_\varepsilon:(\E^{\Omega,0}_R (u^*,\bar{q}^*_a),\Vert \cdot \Vert_{\UUEB}) \rightarrow \R^+$ by
$$
J^R_\varepsilon(u,\bar{q}_a)
=
\left( \left( \left( q^0(b^*,u,\bar q_a)-q^{0*}(b^*) + \varepsilon \right)^+ \right)^2 + d^2_\S \left(  g \left( q_a,q(b^*,u,\bar q_a)  \right) \right) \right)^{1/2}  .
$$
Since $g$ and $d^2_{\S}$ are continuous and so is $F_R(u^*,\bar{q}^*_a)$ (see Lemma~\ref{prop30-1-1}), it follows that $J^R_\varepsilon$ is continuous on $(\E^{\Omega,0}_R (u^*,\bar{q}^*_a),\Vert \cdot \Vert_{\UUEB})$. Moreover, one has $J^R_\varepsilon (u^*,\bar{q}^*_a) = \varepsilon$ and, from optimality of $q^{0*}(b^*)$, $J^R_\varepsilon (u,\bar{q}_a) > 0$ for every $(u,\bar{q}_a) \in \E^{\Omega,0}_R (u^*,\bar{q}^*_a)$.
It follows from the Ekeland variational principle that, for every $\varepsilon > 0$, there exists $(u^R_\varepsilon,\bar{q}^R_{a,\varepsilon}) \in \E^{\Omega,0}_R (u^*,\bar{q}^*_a)$ such that $\Vert (u^R_\varepsilon,\bar{q}^R_{a,\varepsilon})-(u^*,\bar{q}^*_a) \Vert_{\UUEB}  \leq \sqrt{\varepsilon}$ and
\begin{equation}\label{eqconsequenceekeland}
-\sqrt{\varepsilon} \, \Vert (u,\bar{q}_a) - (u^R_\varepsilon,\bar{q}^R_{a,\varepsilon}) \Vert_{\UUEB} \leq  J^R_\varepsilon (u,\bar{q}_a) - J^R_\varepsilon (u^R_\varepsilon,\bar{q}^R_{a,\varepsilon}),
\end{equation}
for every $(u,\bar{q}_a) \in \E^{\Omega,0}_R (u^*,\bar{q}^*_a)$.
In particular, $u^R_\varepsilon$ converges to $u^*$ in $\L^1_\TU([a,b^*)_\TU,\R^m)$ and $\bar{q}^R_{a,\varepsilon}$ converges to $\bar{q}^*_a$ as $\varepsilon$ tends to $0$. Besides, setting
\begin{equation}\label{defp0R}
\psi^{0R}_\varepsilon = \frac{-1}{J^R_\varepsilon (u^R_\varepsilon,\bar{q}^R_{a,\varepsilon})} \big(q^0 (b^*,u^R_\varepsilon,\bar{q}^R_{a,\varepsilon}) - q^{0*}(b^*) + \varepsilon \big)^+ \leq 0
\end{equation}
and
\begin{equation}\label{defpsiepsR}
\psi^R_\varepsilon = \frac{-1}{J^R_\varepsilon (u^R_\varepsilon,\bar{q}^R_{a,\varepsilon})} \left( g (q^R_{a,\varepsilon},q(b^*,u^R_\varepsilon,\bar{q}^R_{a,\varepsilon})) - \P_{\S} \big( g (q^R_{a,\varepsilon},q(b^*,u^R_\varepsilon,\bar q^R_{a,\varepsilon})) \big) \right) \in \R^{j} ,
\end{equation}
note that $\vert \psi^{0R}_\varepsilon \vert^2 + \Vert \psi^R_\varepsilon \Vert_{\R^j}^2 = 1 $ and $-\psi^R_\varepsilon \in \O_\S [\P_{\S} (g (q^R_{a,\varepsilon},q(b^*,u^R_\varepsilon,\bar{q}^R_{a,\varepsilon})))]$.

Using a compactness argument, the continuity of $F_R(u^*,\bar{q}^*_a)$ (see Lemma~\ref{prop30-1-1}), the $\CC^1$-regularity of $g$ and the (partial) converse of the Lebesgue dominated convergence theorem, we infer that there exists a sequence $(\varepsilon_k)_{k \in \N}$ of positive real numbers converging to $0$ such that $u^R_{\varepsilon_k}$ converges to $u^*$ $\DD_1$-a.e. on $[a,b^*)_\TU$,
$\bar{q}^R_{a,\varepsilon_k}$ converges to $\bar{q}^*_a$,
$g(q^R_{a,\varepsilon_k},q(b^*,u^R_{\varepsilon_k},\bar{q}^R_{a,\varepsilon_k}))$ converges to $g(q^*_a,q^*(b^*)) \in \S$,
$\mathrm{d}g(q^R_{a,\varepsilon_k},q(b^*,u^R_{\varepsilon_k},\bar{q}^R_{a,\varepsilon_k}))$ converges to $\mathrm{d}g(q^*_a,q^*(b^*))$,
$\psi^{0R}_{\varepsilon_k}$ converges to some $\psi^{0R} \leq 0$,
and $\psi^R_{\varepsilon_k}$ converges to some $\psi^R \in \R^j$ as $k$ tends to $+\infty$,
with $\vert \psi^{0R} \vert^2 + \Vert \psi^R \Vert_{\R^j}^2 = 1 $ and $-\psi^R \in \O_{\S} [g (q^*_a,q^*(b^*))]$ (see Lemma~\ref{lemconvex}).

In the next lemmas, we use the inequality~\eqref{eqconsequenceekeland} respectively with needle-like variations of $u^R_{\varepsilon_k}$ at right-scattered points of $\TU$ and at right-dense points of $\TU$, and then variations of $\bar{q}^R_{a,\varepsilon_k}$. Hence, we infer some important inequalities by taking the limit in $k$.
Note that these variations were defined in Section~\ref{sectionvariations} for any dynamics $f$, and that we apply them here to the augmented system~\eqref{DD-CS_augm}, associated with the augmented dynamics $\bar f$.

\begin{lemma}
For every $r \in [a,b^*)_\TU \cap \RS_1$ and every $y \in \mathrm{D}^\Omega_\mathrm{stab}(u^*(r))$, considering the needle-like variation $\Pi = (r,y)$ as defined in Section~\ref{section31RS}, one has
\begin{equation}\label{inegfondamscattered}
\psi^{0R}  w^{0}_\Pi (b^*,u^*,\bar{q}^*_a) + \left\langle \Big( \frac{\partial g}{\partial q_2} (q^*_a,q^*(b^*)) \Big)^\top  \psi^R , w_\Pi (b^*,u^*,\bar{q}^*_{a}) \right\rangle_{\R^n} \leq 0 ,
\end{equation}
where the variation vector $\bar w_\Pi=(w_\Pi,w^0_\Pi)^\top$ is defined by~\eqref{varvect_scatteredw} (replacing $f$ with $\bar f$).
\end{lemma}

\begin{proof}
Since $u^R_{\varepsilon_k}$ converges to $u^*$ $\DD_1$-a.e. on $[a,b^*)_\TU$, it follows that $u^R_{\varepsilon_k} (r)$ converges to $u^*(r)$ as $k$ tends to $+\infty$, where $\Vert  u^*(r) \Vert_{\R^m} < R$. It follows that $ \Vert u^R_{\varepsilon_k} (r) \Vert_{\R^m} < R$ and $0$ is not isolated in $\Xi = \{ \alpha \in [0,1], \; u^R_{\varepsilon_k} (r)+\alpha (y-u^R_{\varepsilon_k} (r)) \in \Omega \}$ for any sufficently large $k$, see Definition~\ref{defstable}. Fixing such a large $k$, one has $u^R_{\varepsilon_k,\Pi} (\cdot,\alpha) \in \L^\infty_\TU ([a,b^*)_\TU,\Omega)$ and
\begin{equation*}\begin{split}
\Vert u^R_{\varepsilon_k,\Pi} (\cdot,\alpha) \Vert_{\L^\infty_\TU ([a,b^*)_\TU,\R^m)}
& \leq \max (\Vert u^R_{\varepsilon_k} \Vert_{\L^\infty_\TU ([a,b^*)_\TU,\R^m)}, \Vert u^R_{\varepsilon_k,\Pi} (r,\alpha) \Vert_{\R^m} ) \\
& \leq \max (R, \Vert u^R_{\varepsilon_k} (r) \Vert_{\R^m} + \alpha \Vert y - u^R_{\varepsilon_k} (r) \Vert_{\R^m}  ) ,
\end{split}\end{equation*}
for every $\alpha \in \Xi$. Moreover, one has
\begin{equation*}\begin{split}
\Vert u^R_{\varepsilon_k,\Pi} (\cdot,\alpha) - u^* \Vert_{\L^1_\TU ([a,b^*)_\TU,\R^m)}
& \leq \Vert u^R_{\varepsilon_k,\Pi} (\cdot,\alpha) - u^R_{\varepsilon_k} \Vert_{\L^1_\TU ([a,b^*)_\TU,\R^m)} + \Vert u^R_{\varepsilon_k}-u^* \Vert_{\L^1_\TU ([a,b^*)_\TU,\R^m)} \\
& \leq \alpha \mu_1 (r) \Vert y-u^R_{\varepsilon_k} (r) \Vert_{\R^m} + \sqrt{\varepsilon_k}.
\end{split}\end{equation*}
Therefore $(u^R_{\varepsilon_k,\Pi} (\cdot,\alpha),\bar{q}^R_{a,\varepsilon_k}) \in \E^{\Omega,0}_R (u^*,\bar{q}^*_a)$ for every $\alpha \in \Xi$ sufficiently small and every $k$ sufficiently large.
It then follows from~\eqref{eqconsequenceekeland} that
\begin{equation*}
-\sqrt{\varepsilon_k} \, \Vert u^R_{\varepsilon_k,\Pi} (\cdot,\alpha) - u^R_{\varepsilon_k} \Vert_{\L^1_\TU ([a,b^*)_\TU,\R^m)} \leq  J^R_{\varepsilon_k} (u^R_{\varepsilon_k,\Pi} (\cdot,\alpha),\bar{q}^R_{a,\varepsilon_k}) - J^R_{\varepsilon_k} (u^R_{\varepsilon_k},\bar{q}^R_{a,\varepsilon_k}),
\end{equation*}
and thus
\begin{equation*}
-\sqrt{\varepsilon_k} \, \mu_1 (r) \Vert y-u^R_{\varepsilon_k} (r) \Vert_{\R^m}
\leq \frac{J^R_{\varepsilon_k} (u^R_{\varepsilon_k,\Pi} (\cdot,\alpha),\bar{q}^R_{a,\varepsilon_k})^2 - J^R_{\varepsilon_k} (u^R_{\varepsilon_k},\bar{q}^R_{a,\varepsilon_k})^2}{\alpha ( J^R_{\varepsilon_k} (u^R_{\varepsilon_k,\Pi} (\cdot,\alpha),\bar{q}^R_{a,\varepsilon_k}) + J^R_{\varepsilon_k} (u^R_{\varepsilon_k},\bar{q}^R_{a,\varepsilon_k}) ) } .
\end{equation*}
Using Proposition~\ref{prop32-1}, we infer that
\begin{equation*}\begin{split}
& \lim_{\alpha \to 0} \frac{J^R_{\varepsilon_k} (u^R_{\varepsilon_k,\Pi} (\cdot,\alpha),\bar{q}^R_{a,\varepsilon_k})^2 - J^R_{\varepsilon_k} (u^R_{\varepsilon_k},\bar{q}^R_{a,\varepsilon_k})^2}{\alpha}
\\
& = 2 \big( q^0 (b^*,u^R_{\varepsilon_k},\bar{q}^R_{a,\varepsilon_k}) - q^{0*}(b^*) + \varepsilon_k \big)^+  w^{0}_\Pi (b^*,u^R_{\varepsilon_k},\bar{q}^R_{a,\varepsilon_k}) \\
& \qquad + 2 \Big\langle  g(q^R_{a,\varepsilon_k},q(b^*,u^R_{\varepsilon_k},\bar{q}^R_{a,\varepsilon_k})) - \P_{\S} \big( g(q^R_{a,\varepsilon_k},q(b^*,u^R_{\varepsilon_k},\bar{q}^R_{a,\varepsilon_k})) \big),   \\
& \qquad \qquad \qquad \qquad \qquad \qquad \qquad \frac{\partial g}{\partial q_2} (q^R_{a,\varepsilon_k},q(b^*,u^R_{\varepsilon_k},\bar{q}^R_{a,\varepsilon_k})) \, w_\Pi (b^*,u^R_{\varepsilon_k},\bar{q}^R_{a,\varepsilon_k}) \Big\rangle_{\R^j} .
\end{split}\end{equation*}
Since $J^R_{\varepsilon_k} (u^R_{\varepsilon_k,\Pi} (\cdot,\alpha),\bar{q}^R_{a,\varepsilon_k})$ converges to $J^R_{\varepsilon_k} (u^R_{\varepsilon_k},\bar{q}^R_{a,\varepsilon_k})$ as $\alpha$ tends to $0$, using~\eqref{defp0R} and~\eqref{defpsiepsR} it follows that
\begin{multline*}
-\sqrt{\varepsilon_k} \, \mu_1 (r) \Vert y-u^R_{\varepsilon_k} (r) \Vert_{\R^m} \leq -\psi^{0R}_{\varepsilon_k}  w^{0}_\Pi (b^*,u^R_{\varepsilon_k},\bar{q}^R_{a,\varepsilon_k}) \\
-  \left\langle \Big( \frac{\partial g}{\partial q_2} (q^R_{a,\varepsilon_k},q(b^*,u^R_{\varepsilon_k},\bar{q}^R_{a,\varepsilon_k})) \Big)^\top  \psi^R_{\varepsilon_k} , w_\Pi (b^*,u^R_{\varepsilon_k},\bar{q}^R_{a,\varepsilon_k}) \right\rangle_{\R^n} .
\end{multline*}
By letting $k$ tend to $+\infty$ and using Lemma~\ref{lem32-2w}, the lemma follows.
\end{proof}

We define the sets
\begin{equation*}
\begin{split}
A &= \left\lbrace t \in [a,b^*)_{\TU} \mid u^{R}_{\varepsilon_k}(t) \,\text{does not converge to} \,u^*(t) \,\text{when} \,k \,\text{tends to} \,+\infty \right\rbrace ,\\
A_k &= \left\lbrace t \in [a,b^*)_{\TU} \mid t \notin \LL_{[a,b^*)_\T} (\bar f( \cdot, \bar q(\cdot,u^R_{\varepsilon_k},\bar q^R_{a,k}),u^{R\Phi}_{\varepsilon_k} )) \right\rbrace,\quad  k \in \N.
\end{split}
\end{equation*}

\begin{lemma}
We have $ \mu_{\DD_1} \left( A \cup \bigcup_{k \in \N} A_k \right) = 0. $
\end{lemma}

\begin{proof}
Since $\mu_{\DD_1} (A)=0$, it suffices to prove that $\mu_{\DD_1} ( A_k ) = 0$ for every $k \in \N$. Let $k \in \N$. We set
$$ B_k = \left\lbrace t \in [a,b^*)_{\T} \mid t \notin \LL_{[a,b^*)_\T} (\bar f(\cdot , \bar q(\cdot,u^R_{\varepsilon_k},\bar q^R_{a,k}),u^{R\Phi}_{\varepsilon_k})) \right\rbrace. $$
We know that $\mu_{\DD} (B_k)=0$. Hence, $B_k \subset \RD$ and consequently, $\mu_\DD (B_k) = \mu_L(B_k) = 0$. Since $A_k \subset B_k$, $\mu_L(A_k) = 0$. To conclude, it suffices to prove that $\mu_{\DD_1} ( A_k ) = \mu_L (A_k)$. To see this, let us prove that $A_k \subset \RD_1$. By contradiction, let us assume that there exists $t \in A_k \cap \RS_1$. Since $t \in A_k \subset \RD$, we conclude that $u^{R\Phi}_{\varepsilon_k}$ is constant on $[t,\min(\sigma_1(t),b^*))_{\T}$, where $(t,\min(\sigma_1(t),b^*))_\T \neq \emptyset$. As a consequence, $\bar f( \cdot , \bar q(\cdot,u^R_{\varepsilon_k},\bar q^R_{a,k}),u^{R\Phi}_{\varepsilon_k})$ is continuous on $[t,\min(\sigma_1(t),b^*))_{\T}$ and consequently $t \in \LL_{[a,b^*)_\T} (\bar f(\cdot, \bar q(\cdot,u^R_{\varepsilon_k},\bar q^R_{a,k}),u^{R\Phi}_{\varepsilon_k}))$. This leads to a contradiction.
\end{proof}

We define the set of Lebesgue times $t \in [a,b^*)_{\TU}$ by
$$ \LL^{R,1}_{[a,b^*)_\T} = [a,b^*)_{\TU} \bs \left( A \cup \bigcup_{k \in \N} A_k \right). $$
Note that $\mu_{\DD_1} (\LL^{R,1}_{[a,b^*)_\T}) = \mu_{\DD_1} ([a,b^*)_{\TU}) $ and then $ \mu_{\DD_1} (\LL^{R,1}_{[a,b^*)_\T} \cap \RD_1) = \mu_{\DD_1} ([a,b^*)_{\TU} \cap \RD_1)$.

\begin{lemma}
For every $s \in \LL^{R,1}_{[a,b^*)_\T} \cap \RD_1$ and for every $z \in \Omega \cap \B_{\R^m}(0,R)$,
considering the needle-like variation $\amalg = (s,z)$ as defined in Section~\ref{section31RD}, one has
\begin{equation}\label{inegfondamdense}
\psi^{0R}  w^{0}_\amalg (b^*,u^*,\bar{q}^*_a) + \left\langle \Big( \frac{\partial g}{\partial q_2} (q^*_a,q^*(b^*)) \Big)^\top  \psi^R , w_\amalg (b^*,u^*,\bar{q}^*_a) \right\rangle_{\R^n} \leq 0 ,
\end{equation}
where the variation vector $\bar w_\amalg=(w_\amalg,w^0_\amalg)^\top$ is defined by~\eqref{varvect_densew} (replacing $f$ with $\bar f$).
\end{lemma}

\begin{proof}
For every $k \in \N$ and any $\beta \in \V^{s,b^*}_1$, we recall that $u^R_{\varepsilon_k,\amalg} (\cdot,\beta) \in \L^\infty_\TU ([a,b^*)_\TU,\Omega) $ and
\begin{equation*}
\Vert u^R_{\varepsilon_k,\amalg} (\cdot,\beta) \Vert_{\L^\infty_\TU ([a,b^*)_\TU,\R^m)} \leq \max (\Vert u^R_{\varepsilon_k} \Vert_{\L^\infty_\TU ([a,b^*)_\TU,\R^m)} , \Vert z \Vert_{\R^m} ) \leq R  ,
\end{equation*}
and
\begin{equation*}\begin{split}
\Vert u^R_{\varepsilon_k,\amalg} (\cdot,\beta) - u^* \Vert_{\L^1_\TU ([a,b^*)_\TU,\R^m)}
& \leq \Vert u^R_{\varepsilon_k,\amalg} (\cdot,\beta) - u^R_{\varepsilon_k} \Vert_{\L^1_\TU ([a,b^*)_\TU,\R^m)} + \Vert u^R_{\varepsilon_k}-u^* \Vert_{\L^1_\TU ([a,b^*)_\TU,\R^m)} \\
& \leq 2R \beta + \sqrt{\varepsilon_k}.
\end{split}\end{equation*}
Therefore $(u^R_{\varepsilon_k,\amalg} (\cdot,\beta),\bar{q}^R_{a,\varepsilon_k}) \in \E^{\Omega,0}_R (u^*,\bar{q}^*_a)$ for every $\beta$ sufficiently small and every $k$ sufficiently large.
It then follows from~\eqref{eqconsequenceekeland} that
\begin{equation*}
-\sqrt{\varepsilon_k} \, \Vert u^R_{\varepsilon_k,\amalg} (\cdot,\beta) - u^R_{\varepsilon_k} \Vert_{\L^1_\TU ([a,b^*)_\TU,\R^m)} \leq  J^R_{\varepsilon_k} (u^R_{\varepsilon_k,\amalg} (\cdot,\beta),\bar{q}^R_{a,\varepsilon_k}) - J^R_{\varepsilon_k} (u^R_{\varepsilon_k},\bar{q}^R_{a,\varepsilon_k}) ,
\end{equation*}
and thus
\begin{equation*}
-2 R \sqrt{\varepsilon_k}
\leq \frac{J^R_{\varepsilon_k} (u^R_{\varepsilon_k,\amalg} (\cdot,\beta),\bar{q}^R_{a,\varepsilon_k})^2 - J^R_{\varepsilon_k} (u^R_{\varepsilon_k},\bar{q}^R_{a,\varepsilon_k})^2}{\beta ( J^R_{\varepsilon_k} (u^R_{\varepsilon_k,\amalg} (\cdot,\beta),\bar{q}^R_{a,\varepsilon_k}) + J^R_{\varepsilon_k} (u^R_{\varepsilon_k},\bar{q}^R_{a,\varepsilon_k}) ) } .
\end{equation*}
Using Proposition~\ref{prop33-1}, we infer that
\begin{equation*}\begin{split}
& \lim_{\beta \to 0} \frac{J^R_{\varepsilon_k} (u^R_{\varepsilon_k,\amalg} (\cdot,\beta),\bar{q}^R_{a,\varepsilon_k})^2 - J^R_{\varepsilon_k} (u^R_{\varepsilon_k},\bar{q}^R_{a,\varepsilon_k})^2}{\beta} \\
& = 2 \big( q^0 (b^*,u^R_{\varepsilon_k},\bar{q}^R_{a,\varepsilon_k}) - q^{0*}(b^*) + \varepsilon_k \big)^+ w^{0}_\amalg (b^*,u^R_{\varepsilon_k},\bar{q}^R_{a,\varepsilon_k}) \\
& \qquad + 2 \Big\langle g(q^R_{a,\varepsilon_k},q(b^*,u^R_{\varepsilon_k},\bar{q}^R_{a,\varepsilon_k})) - \P_{\S} \big( g(q^R_{a,\varepsilon_k},q(b^*,u^R_{\varepsilon_k},\bar{q}^R_{a,\varepsilon_k})) \big), \\
& \qquad \qquad \qquad \qquad \qquad \qquad \qquad  \frac{\partial g}{\partial q_2}(q^R_{a,\varepsilon_k},q(b^*,u^R_{\varepsilon_k},\bar{q}^R_{a,\varepsilon_k})) \, w_\amalg (b^*,u^R_{\varepsilon_k},\bar{q}^R_{a,\varepsilon_k}) \Big\rangle_{\R^j}.
\end{split}\end{equation*}
Since $J^R_{\varepsilon_k} (u^R_{\varepsilon_k,\amalg} (\cdot,\beta),\bar{q}^R_{a,\varepsilon_k})$ converges to $J^R_{\varepsilon_k} (u^R_{\varepsilon_k},\bar{q}^R_{a,\varepsilon_k})$ as $\beta$ tends to $0$, using~\eqref{defp0R} and~\eqref{defpsiepsR} it follows that
\begin{equation*}
-2R \sqrt{\varepsilon_k} \leq -\psi^{0R}_{\varepsilon_k}  w^{0}_\amalg (b^*,u^R_{\varepsilon_k},\bar{q}^R_{a,\varepsilon_k})
- \left\langle \Big( \frac{\partial g}{\partial q_2} (q^R_{a,\varepsilon_k},q(b^*,u^R_{\varepsilon_k},\bar{q}^R_{a,\varepsilon_k})) \Big)^\top \psi^R_{\varepsilon_k} , w_\amalg (b^*,u^R_{\varepsilon_k},\bar{q}^R_{a,\varepsilon_k}) \right\rangle_{\R^n}.
\end{equation*}
By letting $k$ tend to $+\infty$, and using Lemma~\ref{lem33-2}, the lemma follows.
\end{proof}

\begin{lemma}
For every $\bar{q}_a \in \R^n \times \{ 0 \}$, considering the variation of initial point as defined in Section~\ref{section32}, one has
\begin{multline}\label{inegfondampoint}
\psi^{0R}  w^0_{\bar{q}_a} (b^*,u^*,\bar{q}^*_a)+ \left\langle \Big( \frac{\partial g}{\partial q_2} (q^*_a,q^*(b^*)) \Big)^\top \psi^R , w_{\bar{q}_a} (b^*,u^*,\bar{q}^*_a) \right\rangle_{\R^n}
\\
\leq - \left\langle \Big( \frac{\partial g}{\partial q_1} (q^*_a,q^*(b^*)) \Big)^\top  \psi^R , q_a \right\rangle_{\R^n}  ,
\end{multline}
where the variation vector $\bar w_{\bar{q}_a}=(w_{\bar{q}_a},w^0_{\bar{q}_a})^\top$ is defined by~\eqref{varvect_pointinit} (replacing $f$ with $\bar f$).
\end{lemma}

\begin{proof}
For every $k \in \N$ and every $\gamma\geq 0$, one has
$$
\Vert \bar{q}^R_{a,\varepsilon_k}+\gamma \bar{q}_a - \bar{q}^*_a \Vert_{\R^{n+1}} \leq  \gamma \Vert \bar{q}_a \Vert_{\R^{n+1}} + \Vert \bar{q}^R_{a,\varepsilon_k} - \bar{q}^*_a \Vert_{\R^{n+1}} \leq \gamma \Vert \bar{q}_a \Vert_{\R^{n+1}} + \sqrt{\varepsilon_k}.
$$
Therefore $(u^R_{\varepsilon_k},\bar{q}^R_{a,\varepsilon_k}+\gamma \bar{q}_a) \in \E^{\Omega,0}_R (u^*,\bar{q}^*_a)$ for every $\gamma$ sufficiently small and every $k$ sufficiently large.
It then follows from~\eqref{eqconsequenceekeland} that
\begin{equation*}
-\sqrt{\varepsilon_k} \, \Vert \bar{q}^R_{a,\varepsilon_k}+\gamma \bar{q}_a - \bar{q}^R_{a,\varepsilon_k} \Vert_{\R^{n+1}} \leq  J^R_{\varepsilon_k} (u^R_{\varepsilon_k},\bar{q}^R_{a,\varepsilon_k}+\gamma \bar{q}_a) - J^R_{\varepsilon_k} (u^R_{\varepsilon_k},\bar{q}^R_{a,\varepsilon_k}),
\end{equation*}
and thus
\begin{equation*}
-\sqrt{\varepsilon_k} \, \Vert \bar{q}_a \Vert_{\R^{n+1}}
\leq \frac{J^R_{\varepsilon_k} (u^R_{\varepsilon_k},\bar{q}^R_{a,\varepsilon_k}+\gamma \bar{q}_a)^2 - J^R_{\varepsilon_k} (u^R_{\varepsilon_k},\bar{q}^R_{a,\varepsilon_k})^2}{\gamma ( J^R_{\varepsilon_k} (u^R_{\varepsilon_k},\bar{q}^R_{a,\varepsilon_k}+\gamma \bar{q}_a) + J^R_{\varepsilon_k} (u^R_{\varepsilon_k},\bar{q}^R_{a,\varepsilon_k}) )  } .
\end{equation*}
Using Proposition~\ref{prop34-1}, we infer that
\begin{equation*}\begin{split}
& \lim_{\gamma \to 0} \frac{J^R_{\varepsilon_k} (u^R_{\varepsilon_k},\bar{q}^R_{a,\varepsilon_k}+\gamma \bar{q}_a)^2 - J^R_{\varepsilon_k} (u^R_{\varepsilon_k},\bar{q}^R_{a,\varepsilon_k})^2}{\gamma} \\
& = 2 \big( q^0 (b^*,u^R_{\varepsilon_k},\bar{q}^R_{a,\varepsilon_k}) - q^{0*}(b^*) + \varepsilon_k \big)^+ w^0_{\bar{q}_a} (b^*,u^R_{\varepsilon_k},\bar{q}^R_{a,\varepsilon_k}) \\
&  \quad + 2 \Big\langle g(q^R_{a,\varepsilon_k},q(b^*,u^R_{\varepsilon_k},\bar{q}^R_{a,\varepsilon_k})) - \P_{\S} \big( g(q^R_{a,\varepsilon_k},q(b^*,u^R_{\varepsilon_k},\bar{q}^R_{a,\varepsilon_k})) \big), \\
& \qquad\qquad  \frac{\partial g}{\partial q_1} (q^R_{a,\varepsilon_k},q(b^*,u^R_{\varepsilon_k},\bar{q}^R_{a,\varepsilon_k})) \, q_a + \frac{\partial g}{\partial q_2} (q^R_{a,\varepsilon_k},q(b^*,u^R_{\varepsilon_k},\bar{q}^R_{a,\varepsilon_k})) \, w_{\bar{q}_a} (b^*,u^R_{\varepsilon_k},\bar{q}^R_{a,\varepsilon_k}) \Big\rangle_{\R^j}.
\end{split}\end{equation*}
Since $J^R_{\varepsilon_k} (u^R_{\varepsilon_k},\bar{q}^R_{a,\varepsilon_k}+\gamma \bar{q}_a)$ converges to $J^R_{\varepsilon_k} (u^R_{\varepsilon_k},\bar{q}^R_{a,\varepsilon_k})$ as $\gamma$ tends to $0$, using~\eqref{defp0R} and~\eqref{defpsiepsR} it follows that
\begin{multline*}
-\sqrt{\varepsilon_k} \, \Vert \bar{q}_a \Vert_{\R^{n+1}} \leq -\psi^{0R}_{\varepsilon_k}  w^{0}_{\bar{q}_a} (b^*,u^R_{\varepsilon_k},\bar{q}^R_{a,\varepsilon_k}) - \left\langle \Big( \frac{\partial g}{\partial q_1} (q^R_{a,\varepsilon_k},q(b^*,u^R_{\varepsilon_k},\bar{q}^R_{a,\varepsilon_k})) \Big)^\top  \psi^R_{\varepsilon_k} , q_a \right\rangle_{\R^n} \\
- \left\langle \Big( \frac{\partial g}{\partial q_2} (q^R_{a,\varepsilon_k},q(b^*,u^R_{\varepsilon_k},\bar{q}^R_{a,\varepsilon_k})) \Big)^\top  \psi^R_{\varepsilon_k} , w_{\bar{q}_a} (b^*,u^R_{\varepsilon_k},\bar{q}^R_{a,\varepsilon_k}) \right\rangle_{\R^n}.
\end{multline*}
By letting $k$ tend to $+\infty$, and using Lemma~\ref{lem34-2}, the lemma follows.
\end{proof}

At this step, we have obtained in the previous lemmas the three inequalities~\eqref{inegfondamscattered},~\eqref{inegfondamdense} and~\eqref{inegfondampoint}, valuable for any $R> \Vert u^* \Vert_{\L^\infty_\TU ([a,b)_\TU,\R^m)}$. Recall that $\vert \psi^{0R} \vert^2 + \Vert \psi^R \Vert_{\R^j}^2 = 1 $ and that $-\psi^R \in \O_{\S} [g (q^*_a,q^*(b^*))]$.
Then, considering a sequence of real numbers $R_\ell$ converging to $+\infty$ as $\ell$ tends to $+\infty$, we infer that there exist $\psi^0 \leq 0$ and $\psi \in \R^j$ such that $\psi^{0R_\ell}$ converges to $\psi^0$ and $\psi^{R_\ell}$ converges to $\psi$ as $\ell$ tends to $+\infty$, and moreover $\vert \psi^0 \vert^2 + \Vert \psi \Vert_{\R^j}^2 = 1 $ and $-\psi \in \O_{\S} [g (q^*_a,q^*(b^*))] $.

We set $\LL^{1}_{[a,b^*)_\T} = \bigcap_{\ell \in \N} \LL^{R_\ell,1}_{[a,b^*)_\T}$. Note that $\mu_{\DD_1} (\LL^{1}_{[a,b^*)_\T}) = \mu_{\DD_1} ([a,b^*)_{\TU}) $ and consequently $ \mu_{\DD_1} (\LL^{1}_{[a,b^*)_\T} \cap \RD_1) = \mu_{\DD_1} ([a,b^*)_{\TU} \cap \RD_1)$. Taking the limit in $\ell$ in~\eqref{inegfondamscattered},~\eqref{inegfondamdense} and~\eqref{inegfondampoint}, we get the following lemma.

\begin{lemma}
For every $r \in [a,b^*)_\TU \cap \RS_1$, and every $y \in \mathrm{D}^\Omega_\mathrm{stab}(u^*(r))$, one has
\begin{equation}\label{varineqscattered}
\psi^{0}  w^{0}_\Pi (b^*,u^*,\bar{q}^*_a) + \left\langle \Big( \frac{\partial g}{\partial q_2} (q^*_a,q^*(b^*)) \Big)^\top \psi , w_\Pi (b^*,u^*,\bar{q}^*_{a}) \right\rangle_{\R^n} \leq 0 ,
\end{equation}
where the variation vector $\bar w_\Pi=(w_\Pi,w^0_\Pi)^\top$ associated with the needle-like variation $\Pi = (r,y)$ of $u$ is defined by~\eqref{varvect_scatteredw} (replacing $f$ with $\bar f$).

For every $s \in \LL^{1}_{[a,b^*)_\T} \cap \RD_1$ and every $z \in \Omega$, one has
\begin{equation}\label{varineqdense}
\psi^{0} w^{0}_\amalg (b^*,u^*,\bar{q}^*_a) + \left\langle \Big( \frac{\partial g}{\partial q_2} (q^*_a,q^*(b^*)) \Big)^\top  \psi , w_\amalg (b^*,u^*,\bar{q}^*_a) \right\rangle_{\R^n} \leq 0 ,
\end{equation}
where the variation vector $\bar w_\amalg=(w_\amalg,w^0_\amalg)^\top$ associated with the needle-like variation $\amalg = (s,z)$ of $u$ is defined by~\eqref{varvect_densew} (replacing $f$ with $\bar f$);

For every $\bar{q}_a \in \R^n \times \{ 0 \}$, one has
\begin{multline}\label{varineqpointinit}
\psi^{0}  w^0_{\bar{q}_a} (b^*,u^*,\bar{q}^*_a)+ \left\langle \Big( \frac{\partial g}{\partial q_2} (q^*_a,q^*(b^*)) \Big)^\top  \psi , w_{\bar{q}_a} (b^*,u^*,\bar{q}^*_a) \right\rangle_{\R^n}
 \\
\leq - \left\langle \Big( \frac{\partial g}{\partial q_1} (q^*_a,q^*(b^*))  \Big)^\top  \psi , q_a \right\rangle_{\R^n}  ,
\end{multline}
where the variation vector $\bar w_{\bar{q}_a}=(w_{\bar{q}_a},w^0_{\bar{q}_a})^\top$ associated with the variation $\bar{q}_a$ of the initial point $\bar{q}_a^*$ is defined by~\eqref{varvect_pointinit} (replacing $f$ with $\bar f$).
\end{lemma}

This result concludes the application of the Ekeland variational principle.
The last step of the proof consists of deriving the PMP from these inequalities.

\subsubsection{Proof of Remark~\ref{remarknonsubmersive}}\label{section46}
In this subsection, we prove the formulation of the PMP mentioned in Remark~\ref{remarknonsubmersive}. Note that we do not prove, at this step, the transversality condition on the final time.

We define $\bar{p} =(p ,p^0 )^\top$ as the unique solution on $[a,b^*]_\T$ of the backward shifted linear $\DD$-Cauchy problem
\begin{equation}\label{equationpourpdanspreuve}
\bar{p}^\DD(t) = - \left( \frac{\partial \bar{f}}{\partial \bar{q}} (t, \bar{q}^*(t),u^{*\Phi}(t)) \right)^\top  \bar{p}^\sigma(t), \quad
\bar{p}(b^*) = \left( \Big( \frac{\partial g}{\partial q_2} (q^*_a,q^*(b^*)) \Big)^\top  \psi,\psi^0 \right)^\top.
\end{equation}
The existence and uniqueness of $\bar{p}$ are ensured by \cite[Theorem 6]{bour10}. Since $\bar f$ does not depend on $q^0$, it is clear that $p^0$ is constant with $p^0=\psi^0 \leq 0$.

\paragraph{Right-dense points.} Let $s \in \LL^1_{[a,b^*)_\T} \cap \RD_1$ and $z \in \Omega$. Since the function $\langle \bar{p} , \bar{w}_\iP(\cdot,u^*,\bar{q}^*_a) \rangle_{\R^{n+1}} $ is absolutely continuous and satisfies $\langle \bar{p} , \bar{w}_\iP(\cdot,u^*,\bar{q}^*_a) \rangle_{\R^{n+1}}^\DD = 0$ $\DD$-a.e. on $[s,b^*)_\T$ from the Leibniz formula~\eqref{eqleibniz}, this function is constant on $[s,b^*]_\T$. It thus follows from~\eqref{varineqdense} that
\begin{equation*}
\begin{split}
\langle \bar{p}(s) , \bar{w}_\amalg(s,u^*,\bar{q}^*_a) \rangle_{\R^{n+1}}
& = \langle \bar{p}(b^*) , \bar{w}_\amalg(b^*,u^*,\bar{q}^*_a) \rangle_{\R^{n+1}} \\
& = \psi^{0} w^{0}_\amalg (b^*,u^*,\bar{q}^*_a) + \left\langle \Big( \frac{\partial g}{\partial q_2} (q^*_a,q^*(b^*)) \Big)^\top  \psi , w_\amalg (b^*,u^*,\bar{q}^*_a) \right\rangle_{\R^n} \leq 0,
\end{split}
\end{equation*}
and since $\bar{w}_\amalg(s,u^*,\bar{q}^*_a) = \bar{f} (s,\bar{q}^*(s),z) - \bar{f} (s,\bar{q}^*(s),u^*(s))$, we finally get
$$
\langle  \bar{p}(s) , \bar{f} (s,\bar{q}^*(s),z) \rangle_{\R^{n+1}} \leq \langle \bar{p}(s) , \bar{f} (s,\bar{q}^*(s),u^*(s)) \rangle_{\R^{n+1}}.
$$
Since this inequality holds for every $z \in \Omega$, we have obtained the maximization condition
$$ H(s,q^*(s),p(s),p^0,u^*(s)) = \max_{z \in \Omega} H(s,q^*(s),p(s),p^0,z) . $$

\paragraph{Right-scattered points.} Let $r \in [a,b^*)_\TU \cap \RS_1$ and $y \in \mathrm{D}^\Omega_\mathrm{stab}(u^*(r))$. Since $\langle \bar{p} , \bar{w}_\Pi(\cdot,u^*,\bar{q}^*_a) \rangle_{\R^{n+1}}$ is absolutely continuous and satisfies $\langle \bar{p} , \bar{w}_\Pi(\cdot,u^*,\bar{q}^*_a) \rangle_{\R^{n+1}}^\DD = 0 $ $\DD$-a.e. on $[\sigma^*_1(r),b^*)_\T$ from the Leibniz formula~\eqref{eqleibniz}, this function is constant on $[\sigma^*_1(r),b^*]_\T$. It thus follows from~\eqref{varineqscattered} that
\begin{multline*}
\langle \bar{p}(\sigma^*_1(r)) , \bar{w}_\Pi(\sigma^*_1(r),u^*,\bar{q}^*_a) \rangle_{\R^{n+1}}
= \langle \bar{p}(b^*) , \bar{w}_\Pi(b^*,u^*,\bar{q}^*_a) \rangle_{\R^{n+1}} \\
= \psi^{0}  w^{0}_\Pi (b^*,u^*,\bar{q}^*_a) + \left\langle \Big( \frac{\partial g}{\partial q_2} (q^*_a,q^*(b^*)) \Big)^\top \psi  , w_\Pi (b^*,u^*,\bar{q}^*_{a}) \right\rangle_{\R^n} \leq 0 .
\end{multline*}
We recall that $\bar{w}_\Pi(\sigma^*_1(r),u^*,\bar{q}^*_a) = \bar{h}_\Pi(\sigma^*_1(r),u^*,\bar{q}^*_a)$ where the variation vector $\bar h_\Pi=(h_\Pi,h^0_\Pi)^\top$ associated with the needle-like variation $\Pi = (r,y)$ of $u$ is defined by~\eqref{varvect_scatteredh} (replacing $f$ by $\bar{f}$). 
Since $\bar{h}_\Pi(r,u^*,\bar{q}^*_a) = 0$, it follows that
$$ \int_{[r,\sigma^*_1(r))_\T} \langle \bar{p} , \bar{h}_\Pi(\cdot,u^*,\bar{q}^*_a) \rangle_{\R^{n+1}}^\DD (\t) \; \DD \t = \langle \bar{p}(\sigma^*_1(r)) , \bar{h}_\Pi(\sigma^*_1(r),u^*,\bar{q}^*_a) \rangle_{\R^{n+1}} \leq 0. $$
Using the Leibniz formula~\eqref{eqleibniz} in the integrand and using~\eqref{varvect_scatteredh} (replacing $f$ by $\bar{f}$) and \eqref{equationpourpdanspreuve}, we finally get
\begin{equation*}
\left\langle \int_{[r,\sigma^*_1(r))_\T} \frac{\partial H}{\partial u} (\t,q^*(\t),p^\sigma(\t),p^0,u^*(r)) \; \DD \t \; , \; y-u^*(r) \right\rangle_{\R^m}  \leq 0.
\end{equation*}

\paragraph{Transversality conditions.}
The transversality condition on the adjoint vector $p$ at the final time $b^*$ has been obtained by definition (note that $-\psi \in \O_\S [g(q^*_a,q^*(b^*))] $ as mentioned previously). Let us now establish the transversality condition on the adjoint vector $p$ at the initial time $a$.
Let $\bar{q}_a \in \R^n \times \{ 0 \}$. With the same arguments as before, we prove that the function $\langle \bar{p}, \bar{w}_{\bar{q}_a}(\cdot,u^*,\bar{q}^*_a) \rangle_{\R^{n+1}}$ is constant on $[a,b^*]_\T$. It thus follows from~\eqref{varineqpointinit} that
\begin{equation*}
\begin{split}
\langle \bar{p}(a) , \bar{w}_{\bar{q}_a}(a,u^*,\bar{q}^*_a) \rangle_{\R^{n+1}}
& = \langle \bar{p}(b^*) , \bar{w}_{\bar{q}_a}(b^*,u^*,\bar{q}^*_a) \rangle_{\R^{n+1}} \\
& = \psi^{0}  w^0_{\bar{q}_a} (b^*,u^*,\bar{q}^*_a)+ \left\langle \Big( \frac{\partial g}{\partial q_2} (\bar{q}^*_a,\bar{q}^*(b^*)) \Big)^\top \psi , w_{\bar{q}_a} (b^*,u^*,\bar{q}^*_a) \right\rangle_{\R^n}  \\
&  \leq - \left\langle \Big( \frac{\partial g}{\partial q_1} (q^*_a,q^*(b^*)) \Big)^\top \psi , q_a \right\rangle_{\R^n}  ,
\end{split}
\end{equation*}
and since $\bar{w}_{\bar{q}_a}(a,u^*,\bar{q}^*_a)=\bar{q}_a=(q_a,0)$, we finally get
\begin{equation*}
\left\langle p(a)+ \Big( \frac{\partial g}{\partial q_1} (q^*_a,q^*(b^*)) \Big)^\top \psi , q_a \right\rangle_{\R^n} \leq 0.
\end{equation*}
Since this inequality holds for every $\bar{q}_a \in \R^n \times \{ 0 \}$, the left-hand equality of~\eqref{transv_cond} follows.

\subsubsection{End of the proof}\label{section47}
In this subsection, we conclude the proof of Theorem~\ref{thmmain}. Note that we prove the fourth item of Theorem~\ref{thmmain} in Remark~\ref{remfindepreuvefreefinaltime} below. To conclude the proof of Theorem~\ref{thmmain}, we will use the result claimed in Remark~\ref{remarknonsubmersive}. Let us separate two cases.

Firstly, let us assume that $g$ is submersive at $(q^*_a,q^*(b^*))$. In that case, we have just to prove that the couple $(p,p^0)$ is not trivial. Let us assume that $p$ is trivial. Then $p(a)=p(b^*)=0$ and, from the transversality conditions on the adjoint vector, $\psi$ belongs to the kernels of $( \partial_{q_1} g  (q^*_a,q^*(b^*)) )^\top $ and $( \partial_{q_2} g  (q^*_a,q^*(b^*)) )^\top $. It follows that $\psi$ belongs to the orthogonal of the image of the differential of $g$ at $(q^*_a,q^*(b^*))$. Since $g$ is submersive at $(q^*_a,q^*(b^*))$, it implies that $\psi = 0$. Since the couple $(\psi,p^0)$ is not trivial, we conclude that $p^0 \neq 0$ and then $(p,p^0)$ is not trivial. This concludes the proof of Theorem~\ref{thmmain} in this first case.

Secondly, let us assume that $g$ is not necessarily submersive at $(q^*_a,q^*(b^*))$. In that case, one has to note that if $q^*$ is an optimal solution of $\OSCP$ associated with the function $g$ and with the closed convex $\S$, then $q^*$ is also an optimal solution of $\OSCP$ associated with the function $\tilde{g}$ defined by $\tilde{g}(q_1,q_2) = (q_1,q_2)$ (which is submersive at any point) and with the closed convex $\tilde{\S} = \{ q^*_a \} \times \{ q^*(b^*) \}$. Then, we get back to the above first case, but with a different function $g$ and a different closed convex $\S$ that would impact only the third item of Theorem~\ref{thmmain}.

\begin{remark}\label{remfindepreuvefreefinaltime}
Throughout this remark, we assume that all asumptions of the fourth item of Theorem~\ref{thmmain} are satisfied. As mentioned in Remark~\ref{remmultiscale}, Theorem~\ref{thmmain} can be easily extended (without the fourth item, for now) to the framework of dynamical systems on time scales with several sampled-data controls with different sets of controlling times. To prove the transversality condition on the final time, we use the \textit{multiscale version} (without the fourth item) of Theorem~\ref{thmmain}.

Let $\delta>0$ be such that $[b^* - \delta , b^* + \delta] \subset \T$ and such that $(f,f^0)$ is of class $\CC^1$ on $[b^* - \delta , b^* + \delta] \times \R^n \times \R^m$. Obviously, the trajectory $q^*$, associated with $b^*$ and $u^*$, is an optimal solution of the optimal sampled-data control problem, with free final time $b \in (b^* - \delta, b^* + \delta)$, defined by
\begin{align}[left=\OF_{[b^* - \delta, b^* + \delta]}  \qquad\empheqlbrace]
& \min  \int_{b^* - \delta}^b f^0 (\t,q (\t), u^\Phi(\t) ) \, d \t ,  \notag \\
& \dot{q}(t) = f (t,q(t),u^\Phi(t)),  \quad \text{for a.e. $t \in [b^* - \delta,b)$, } \notag  \\
& u\in \L^\infty_{\TU} (\TU,\Omega) , \notag \\
& (q(b^* - \delta),g(q^*(a),q(b)))  \in \{ q^*(b^* - \delta) \} \times \S. \notag
\end{align}
We set $\rho^* (s) = \delta (s-1) + b^*$, $x^*(s) = q^* \circ \rho^* (s)$ and $w^* (s) = \delta$, for every $s \in [0,1]$. With the change of variable $x = q \circ \rho$ and with $\dot{\rho} = w$, it is clear that the augmented trajectory $(\rho^*,x^*)$, associated with the augmented control $(w^*,u^*)$, is an optimal solution of the optimal sampled-data control problem, with fixed final time $s=1$, defined by
\begin{align}[left=\OF_{[0,1]}  \qquad\empheqlbrace]
& \min  \int_{0}^1 w(\t) f^0 (\rho(\t), x (\t), u^{\Phi \circ \rho }(\t) ) \, d \t ,  \notag \\
& (\dot{\rho} (t),\dot{x}(t)) = (w(t),w(t) f ( \rho (t) , x(t) , u^{\Phi \circ \rho }(t)) ),  \quad \text{for a.e. $t \in [0,1)$, } \notag  \\
& (w,u) \in \L^\infty ( [0,1], [ \delta/2, + \infty ) ) \times \L^\infty_{\TU} (\TU,\Omega) , \notag \\
& (\rho(0),x(0)) = (b^* - \delta,q^*(b^* - \delta)), \notag \\
& (\rho (1) , g(q^*(a),x(1)) ) \in [b^* - \delta , b^* + \delta ] \times \S . \notag
\end{align}
In this new optimal sampled-data control problem, the sampled-data controls $w$ and $u$ are defined on different time scales. Its Hamiltonian $\tilde{H}$ is 
$ \tilde{H}(\rho,x,p_\rho,p_x,p^0,w,u) = p_\rho w + \langle p_x , w f(\rho,x,u) \rangle_{\R^n} +  p^0 w f^0 (\rho,x,u).$
Applying the \textit{multiscale version} (without the fourth item) of Theorem~\ref{thmmain}, since $w^*$ takes its values in the interior of the constraint set $[\delta/2,+\infty)$, it follows in particular that
$$ \dfrac{\partial \tilde{H}}{\partial w} (\rho^* (s) , x^*(s) , p_{\rho} (s) , p_x (s),  p^0 , w^*(s), u^{*,\Phi \circ \rho^* }(s) ) = 0, $$
for almost every $s \in [0,1)$. Moreover, since $\rho^* (1) = b^*$ belongs to the interior of $(b^* - \delta , b^* + \delta )$, we can select $(p_\rho,p_x)$ such that $p_\rho (1) = 0$ (usual transversality condition on the adjoint vector). It follows that
$$ \langle p_x (s) ,  f(\rho^*(s),x^*(s),u^{*,\Phi \circ \rho^* }(s) \rangle_{\R^n} + p^0 f^0 (\rho^*(s),x^*(s),u^{*,\Phi \circ \rho^* }(s))  = - p_\rho (s) ,$$
for almost every $s \in [0,1]$.
We have thus proved that the function $t \mapsto H(t,q^*(t),p(t),p^0,u^{*\Phi} (t) )$ coincides, almost everywhere on $[b^*-\delta,b^*)$, with a continuous function vanishing at $t=b^*$.
\end{remark}

\begin{remark}
Another method in order to prove the fourth item of Theorem~\ref{thmmain} consists of considering variations of the final time $b$ in a neighborhood of $b^*$ and to modify accordingly the functional of Section~\ref{section41} to which Ekeland's variational principle is applied. 
\end{remark}

\subsection{Proof of Theorem~\ref{propexistence}}\label{proofexistence}
The following proof can be read independently of the rest of Section~\ref{annexe}. It is inspired from \cite[Chap. 6]{trel}.
We only treat the case where the final time $b \in \T$ is fixed. The proof can be easily adapted to the case of a free final time.

Let us consider a sequence $(q_k)_{k \in \N}$ of $\mathcal{M}$, associated with sampled-data controls $u_k \in \L^\infty_{\TU}(\TU,\Omega)$, minimizing the cost considered in $\OSCP$. It follows from the assumptions that the sequence $( f (\cdot,q_k,u^\Phi_k ),f^0 (\cdot,q_k,u^\Phi_k ) )_{k \in \N}$ is bounded in $\L^2_{\T} ([a,b)_\T,\R^{n+1})$. Hence a subsequence (that we do not relabel) converges in the weak topology of $\L^2_{\T} ([a,b)_\T,\R^{n+1})$ to a function $(F,F^0)$. Moreover, a subsequence of $( q_k (a) )_{k \in \N}$ (that we do not relabel) converges in $\R^n$ and we denote the limit by $x(a)$. We define the absolutely continuous functions $ x(t) = x(a) + \int_{[a,t)_\T} F(\t)  \DD \t$ and $x^0 (t) =  \int_{[a,t)_\T} F^0(\t) \DD \t $. In particular, note the pointwise convergence of $( q_k )_{k \in \N}$ to $x$ on $[a,b]_\T$. Since $g$ is continuous and since $\S$ is closed, we have $g(x(a),x(b)) \in \S$. Note that $x^0 (b)$ is equal to the infimum of admissible costs. To conclude the proof, we have to prove the existence of $u \in \L^\infty_{\TU} (\TU,\Omega)$ such that $ x(t) = x(a) + \int_{[a,t)_\T} f(\t,x(\t),u^\Phi(\t))  \DD \t$ and $x^0 (t) =  \int_{[a,t)_\T} f^0(\t,x(\t),u^\Phi(\t)) \DD \t $.

The sequence $ ( f (\cdot,x,u^\Phi_k ),f^0 (\cdot,x,u^\Phi_k ) )_{k \in \N}$ is bounded in $\L^2_{\T} ([a,b)_\T,\R^{n+1})$. Hence, a subsequence (that we do not relabel) converges in the weak topology of $\L^2_{\T} ([a,b)_\T,\R^{n+1})$ to some $(\tilde{F},\tilde{F}^0)$. It follows from the Lebesgue dominated convergence theorem and from the global Lipschitz continuity of $f$ in $q$ on $[a,b]_\T \times \B_{\R^n}(0,M) \times \Omega$ that $ \int_{[a,b)_\T} \varphi(\t) F(\t) \DD \t = \int_{[a,b)_\T} \varphi(\t) \tilde{F}(\t) \DD \t $ for every $\varphi \in \L^2_{\T} ([a,b)_\T,\R)$. Hence $F=\tilde{F}$ (and similarly $F^0=\tilde{F}^0$) $\DD$-a.e. on $[a,b)_\T$.

We define the set $ \mathcal{W}$ of functions $\Theta \in \L^2_{\T} ([a,b)_\T,\R^{n+1})$ such that $\Theta (t) \in \mathcal{W}(t,x(t))$ for $\DD$-a.e. $t \in [a,b)_\T$. It follows from the assumptions that $\mathcal{W}$ is a closed (convex) subset of $\L^2_{\T} ([a,b)_\T,\R^{n+1})$ with its usual topology, and thus with its weak topology as well (see \cite{brez}). We infer that $(F,F^0)=(\tilde{F},\tilde{F}^0) \in \mathcal{W}$. Hence, for $\DD$-a.e. $t \in [a,b)_\T$, there exists $v(t) \in \Omega$ such that $ (F(t),F^0(t)) = (f (t,x(t),v(t) ) , f^0 (t,x(t),v(t) ) ) $. Note that $v$ can be selected $\mu_\DD$-measurable on $[a,b)_\T$.\footnote{Indeed, similarly to the proof of Proposition~\ref{propnumber1}, one has just to consider the extension function $\tilde{v}$ defined on $[a,b)$ by $\tilde{v} (t) = v(t)$ if $t \in \T$ and by $\tilde{v}(t) = v(r)$ if $t \in (r,\sigma(r))$, where $r \in \RS$. Then, the function $\tilde{v}$ can be selected $\mu_L$-measurable on $[a,b)$ from \cite[Lemma~3A p. 161]{lee}, which implies the $\mu_\Delta$-measurability of $v$ on $[a,b)_\T$ from~\cite{caba}.}

Since $\RS_1 \cap [a,b)_\T$ is at most countable, using a diagonal argument, there exists a subsequence of $(u_k)_{k \in \N}$ (that we do not relabel) such that $u_k (r)$ converges to some $u_r \in \Omega$ for every $r \in \RS_1 \cap [a,b)_\T$. It follows from the Lebesgue dominated convergence theorem that  $ x(t) = x(a) + \int_{[a,t)_\T} f(\t,x(\t),u^\Phi(\t))  \DD \t$ and $x^0 (t) =  \int_{[a,t)_\T} f^0(\t,x(\t),u^\Phi(\t)) \DD \t $, where $u$ is defined by $u(t)=u_r$ if $t=r \in \RS_1 \cap [a,b)_\T$ and $u(t)=v(t)$ if $t \in \RD_1 \cap [a,b)_\T$. It is easy to prove that $u \in \L^\infty_\TU ([a,b)_\TU,\Omega)$ (the $\mu_{\DD_1}$-measurability of $u$ on $[a,b)_\TU$ is established as in the proof of Proposition~\ref{propnumber1}). This concludes the proof of Theorem~\ref{propexistence}.

\bigskip

\noindent{\bf Acknowledgment.}
The second author was partially supported by the Grant FA9550-14-1-0214 of the EOARD-AFOSR.

\bibliographystyle{plain}

\end{document}